\documentclass[11pt]{article}
\usepackage{verbatim}
\usepackage{graphicx}
\usepackage{amsmath,amssymb,amsfonts,euscript,enumerate,latexsym}
\usepackage{amsthm}
\usepackage{color,wrapfig}
\usepackage{pxfonts}
\usepackage{fancyhdr}
\usepackage{mathrsfs}
\usepackage{mathtools}
\usepackage{epsfig,subfigure}
\usepackage{marvosym}
\usepackage{txfonts}
\usepackage{hyperref}
\usepackage{times}
\usepackage{enumerate}

\def\titlerunning#1{\gdef\titrun{#1}}
\makeatletter
\def\author#1{\gdef\autrun{\def\and{\unskip, }#1}\gdef\@author{#1}}
\def\address#1{{\def\and{\\\hspace*{18pt}}\renewcommand{\thefootnote}{}%
\footnote {#1}}%
\markboth{\autrun}{\titrun}}
\makeatother
\def\email#1{e-mail: #1}
\def\subjclass#1{{\renewcommand{\thefootnote}{}%
\footnote{\emph{Mathematics Subject Classification (2010):} #1}}}
\def\keywords#1{\par\medskip
\noindent\textbf{Keywords.} #1}

\usepackage{etoolbox}
\makeatletter
\patchcmd{\@maketitle}
  {\ifx\@empty\@dedicatory}
  {\ifx\@empty\@date \else {\vskip3ex \centering\footnotesize\@date\par\vskip1ex}\fi
   \ifx\@empty\@dedicatory}
  {}{}
\patchcmd{\@adminfootnotes}
  {\ifx\@empty\@date\else \@footnotetext{\@setdate}\fi}
  {}{}{}
\makeatother

\newtheorem{thm}{Theorem}[section]
\newtheorem{cor}[thm]{Corollary}
\newtheorem{lemm}[thm]{Lemma}

\newtheorem{remar}[thm]{Remark}

\newcommand{\R}{{\mathbb{R}}}
\newcommand{\Z}{{\mathbb{Z}}}

\newcommand{\La}{\triangle}
\newcommand{\bs}{\backslash}

\newcommand{\1}{\partial}
\newcommand{\2}{\overline}
\newcommand{\3}{\varepsilon}

\begin{document}

\baselineskip=16pt
\titlerunning{}
\title{Existence and large time behaviour of finite points blow-up solutions of the fast diffusion equation}

\author{Kin Ming Hui \and Sunghoon Kim}

\date{}

\maketitle

\address{K. M. Hui : 
Institute of Mathematics, Academia Sinica, Taipei, 10617, Taiwan, R.O.C., \\
\email{kmhui@gate.sinica.edu.tw}
\and
S. Kim (\Letter):
Department of Mathematics, School of Natural Sciences, The Catholic University of Korea,\\
43 Jibong-ro, Wonmi-gu, Bucheon-si, Gyeonggi-do, 14662, Republic of Korea, \\
\email{math.s.kim@catholic.ac.kr}
}

\subjclass{Primary 35K61, 35K60, 35B40 Secondary 35K67, 35B44}

\begin{abstract}
Let $\Omega\subset\R^n$ be a smooth bounded domain and let $a_1,a_2,\dots,a_{i_0}\in\Omega$, $\widehat{\Omega}=\Omega\setminus\{a_1,a_2,\dots,a_{i_0}\}$ and $\widehat{R^n}=\R^n\setminus\{a_1,a_2,\dots,a_{i_0}\}$. We prove the existence of solution $u$ of the fast diffusion equation $u_t=\Delta u^m$, $u>0$, in $\widehat{\Omega}\times (0,\infty)$ ($\widehat{R^n}\times (0,\infty)$ respectively) which satisfies  $u(x,t)\to\infty$ as $x\to a_i$ for any $t>0$ and $i=1,\cdots,i_0$, when $0<m<\frac{n-2}{n}$, $n\geq 3$, and the initial value satisfies $0\le u_0\in L^p_{loc}(\2{\Omega}\setminus\{a_1,\cdots,a_{i_0}\})$ ($u_0\in L^p_{loc}(\widehat{R^n})$ respectively) for some constant $p>\frac{n(1-m)}{2}$ and  $u_0(x)\ge \lambda_i|x-a_i|^{-\gamma_i}$ for  $x\approx a_i$ and some constants $\gamma_i>\frac{2}{1-m},\lambda_i>0$, for all $i=1,2,\dots,i_0$. We also find the blow-up rate of such solutions near the blow-up points $a_1,a_2,\dots,a_{i_0}$, and obtain the asymptotic large time behaviour of such singular solutions. More precisely we prove that if $u_0\ge\mu_0$ on $\widehat{\Omega}$ ($\widehat{R^n}$, respectively) for some constant $\mu_0>0$  and $\gamma_1>\frac{n-2}{m}$, then the singular solution $u$ converges locally uniformly on every compact subset of $\widehat{\Omega}$ (or $\widehat{R^n}$ respectively) to infinity as $t\to\infty$. If $u_0\ge\mu_0$ on $\widehat{\Omega}$ ($\widehat{R^n}$, respectively) for some constant $\mu_0>0$ and satisfies $\lambda_i|x-a_i|^{-\gamma_i}\le u_0(x)\le \lambda_i'|x-a_i|^{-\gamma_i'}$ for  $x\approx a_i$ and some constants $\frac{2}{1-m}<\gamma_i\le\gamma_i'<\frac{n-2}{m}$,  $\lambda_i>0$, $\lambda_i'>0$, $i=1,2,\dots,i_0$, we prove that $u$ converges in $C^2(K)$ for any compact subset $K$ of $\2{\Omega}\setminus\{a_1,a_2,\dots,a_{i_0}\}$ (or $\widehat{R^n}$ respectively) to a harmonic function as $t\to\infty$.
\keywords{fast diffusion equation, blow-up solution, blow-up rate, asymptotic large time behaviour}
\end{abstract}

\vskip 0.2truein

\setcounter{equation}{0}
\setcounter{section}{0}

\section{Introduction}\label{section-intro}
\setcounter{equation}{0}
\setcounter{thm}{0}

We will study the existence and asymptotic large time behaviour of singular solution $u$ of the fast diffusion equation
\begin{equation}\label{fde}
u_t=\La u^m
\end{equation}
in  bounded and unbounded domains where $0<m<\frac{n-2}{m}$, $n\geq 3$, with nonnegative initial value that blows up at a finite number of points in the domain. Recently there are a lot of research on \eqref{fde} because this equation arises in many physical and geometrical applications \cite{A}, \cite{DK}, \cite{V2}. When $m>1$, \eqref{fde}  is called porous medium equation which appears in the modeling of the flow of an ideal gas in a homogeneous porous media and the filtration of incompressible fluids through a porous medium \cite{A}, \cite{V3}. When $m=1$, \eqref{fde} is the heat equation.  When  $0<m<1$,  \eqref{fde} is called the fast diffusion equation. When $m=\frac{n-2}{n+2}$ and $n\ge 3$,  \eqref{fde} arises in the study of Yamabe flow on $\R^n$ \cite{DS2}, \cite{PS}. Note that the metric $g_{ij}=u^{\frac{4}{n+2}}dx^2$, $u>0$, $n\ge 3$, is a solution of the Yamabe flow  \cite{DS2}, \cite{PS},
\begin{equation*}
\frac{\1 g_{ij}}{\1 t}=-Rg_{ij} \quad \mbox{ in }\R^n
\end{equation*}
if and only if  $u$ is a solution of 
\begin{equation*}
u_t=\frac{n-1}{m}\Delta u^m
\end{equation*}
with $m=\frac{n-2}{n+2}$ where $R(\cdot,t)$ is the scalar curvature of the metric $g_{ij}(\cdot,t)$. Recently Huang, Pan and Wang \cite{HPW}, T.~Luo and H.~Zeng \cite{LZ} have shown that \eqref{fde} with $m>1$ is also the large time asymptotic limit solution of the compressible Euler equation with damping. F.~Golse and F.~Salvarani \cite{GS}, B.~Choi and K.~Lee \cite{CL}, have shown that \eqref{fde} also appears as the nonlinear diffusion limit for the generalized Carleman models.

As observed by L.~Peletier \cite{P} and J.L~Vazquez \cite{V1} there is a big difference on the behaviour of solutions
of \eqref{fde} for $(n-2)/n<m<1$, $n\ge 3$,
and for $0<m\le (n-2)/n$, $n\ge 3$. For example there is a $L^1-L^{\infty}$ regularizing effect  (\cite{HP}, \cite{DaK}) for the solutions
of
\begin{equation}\label{Cauchy-problem}
\left\{\begin{aligned}
u_t=&\Delta u^m, u\ge 0,\quad\mbox{ in }\R^n\times (0,T)\\
u(x,0)=&u_0\qquad\qquad\quad\mbox{ in }\R^n
\end{aligned}\right.
\end{equation}
with $0\le u_0\in L^1_{loc}(\R^n)$ for any $(n-2)/n<m<1$. However there is no such $L^1-L^{\infty}$ regularizing effect \cite{V2} for  solutions of \eqref{Cauchy-problem} when $0<m\le (n-2)/n$ and $n\ge 3$.

Although there are a lot of study (\cite{A}, \cite{DK}, \cite{V3}) on the existence and various properties of the solutions of \eqref{fde} for $m>\frac{(n-2)_+}{n}$,  there are not many results on \eqref{fde} for the case $0<m<\frac{(n-2)_+}{n}$. When $\frac{(n-2)_+}{n}<m<1$, existence and uniqueness of global weak solution of \eqref{Cauchy-problem}
for any $0\le u_0\in L^1_{loc}(\R^n)$ has been proved by M.A.~Herrero and M.~Pierre in \cite{HP}. When $0<m\le (n-2)/n$ and $n\ge 3$,
existence of positive smooth solutions of \eqref{Cauchy-problem} for any $0\le u_0\in L_{loc}^p(\R^n)$, $p>(1-m)n/2$, satisfying the condition,
\begin{equation*}
\liminf_{R\to\infty}\frac{1}{R^{n-\frac{2}{1-m}}}\int_{|x|\le R}u_0\,dx
\ge C_1T^{\frac{1}{1-m}}
\end{equation*}
for some constant $C_1>0$ is proved by S.Y.~Hsu in \cite{Hs}.

Let $\Omega\subset\R^n$ be a smooth bounded domain and let $a_1,a_2,\dots,a_{i_0}\in\Omega$, $\widehat{\Omega}=\Omega\setminus\{a_1,a_2,\dots,a_{i_0}\}$, and $\widehat{R^n}=\R^n\setminus\{a_1,a_2,\dots,a_{i_0}\}$. When $0<m\le\frac{n-2}{n}$ and $n\ge 3$, existence of singular solutions of \eqref{fde} in $\widehat{\Omega}\times (0,T)$ which blows up at $\{a_1,a_2,\dots,a_{i_0}\}\times (0,T)$ was proved by K.M.~Hui and Sunghoon Kim in \cite{HK2} when the initial value $u_0$ satisfies 
\begin{equation*}
u_0(x)\approx |x-a_i|^{-\gamma_i}\quad\mbox{ for }x\approx a_i\quad\forall i=1,2,\dots,i_0
\end{equation*} 
for some constants $\gamma_i>\max\left(\frac{n}{2m},\frac{n-2}{m}\right)$ for any $i=1,2,\dots,i_0$. When $0<m\le\frac{n-2}{n}$, $n\ge 3$ and $0\in\Omega$, existence of singular solutions and asymptotic large time behaviour of \eqref{fde} in $(\Omega\setminus\{0\})\times (0,T)$ which blows up at $\{0\}\times (0,\infty)$ when the initial value $u_0$ satisfies $c_1|x|^{-\gamma_1}\le u_0(x)\le c_2 |x|^{-\gamma_2}$ for  some constants $c_1>0$, $c_2>0$ and $\gamma_2\ge\gamma_1>\frac{2}{1-m}$ were proved by J. L. Vazquez and M. Winkler in \cite{VW1}.  When $0<m\le\frac{n-2}{n}$ and $n\ge 3$, existence  of singular solutions of \eqref{fde} in $\widehat{\R^n}\times (0,T)$ which blows up at $\{0\}\times (0,\infty)$  when the initial value $u_0$ satisfies $c_1|x|^{-\gamma}\le u_0(x)\le c_2 |x|^{-\gamma}$ for any $x\in\R^n\setminus\{0\}$ and some constant $\frac{2}{1-m}<\gamma<\frac{n-2}{m}$ was proved by K. M. Hui and Soojung Kim in \cite{HKs}. Asymptotic large time behaviour of such solution was also proved by K. M. Hui and Soojung Kim in \cite{HKs} when $\frac{2}{1-m}<\gamma<n$. In this paper we will extend the results of \cite{HK2}, \cite{HKs} and \cite{VW1} to the case when the initial value $u_0$ satisfies 
\begin{equation}\label{u0-lower-blow-up-rate}
u_0(x)\geq \frac{\lambda_i}{|x-a_i|^{\gamma_i}} \qquad \forall 0<|x-a_i|<\delta_1,\,\,i=1, \cdots, i_0.
\end{equation}
for some constants $0<\delta_1<\delta_0$, $\lambda_1$, $\cdots$, $\lambda_{i_0}\in\R^+$ and $\gamma_1,\cdots,\gamma_{i_0}\in \left(\frac{2}{1-m},\infty\right)$.

For any $x_0\in\R^n$ and $R>0$, let $B_R(x_0)=B\left(x_0,R\right)=\{x\in\R^n:|x-x_0|<R\}$, $B_R=B_R(0)$, $\widehat{B}_R(x_0)=B_R(x_0)\bs\left\{x_0\right\}$ and $\widehat{B}_R=\widehat{B}_R(0)$. We choose $R_0>0$ such that  $a_1,\cdots,a_{i_0}\in B_{R_0}$. For any $\delta>0$, let $\Omega_{\delta}=\Omega\bs\left(\cup_{i=1}^{i_0}B_{\delta}(a_i)\right)$ and  $\R^n_{\delta}=\R^n\bs\left(\cup_{i=1}^{i_0}B_{\delta}(a_i)\right)$. Let $\delta_0(\Omega)=\frac{1}{3}\min_{1\leq i,j\leq i_0}\left(\textbf{dist}(a_i,\Omega),\left|a_i-a_j\right|\right)$ and $\delta_0(\R^n)=\frac{1}{3}\min_{1\leq i,j\leq i_0}|a_i-a_j|$. When there is no ambiguity we will drop the parameter and write $\delta_0$ instead of $\delta_0(\Omega)$ or $\delta_0(\R^n)$. Unless stated otherwise we will assume that $0<m<\frac{n-2}{n}$ and $n\ge 3$ for the rest of the paper.
 
In this paper we will prove the existence of solution $u$ of \eqref{fde} in $\widehat{\Omega}\times (0,\infty)$ ($\widehat{R^n}\times (0,\infty)$ respectively) which satisfies 
\begin{equation}\label{u-blow-up-infty}
u(x,t)\to\infty\quad\mbox{ as }x\to a_i\quad\forall t>0,i=1,\cdots,i_0, 
\end{equation}
when $0<m<\frac{n-2}{n}$, $n\geq 3$, and the initial value satisfies $0\le u_0\in L^p_{loc}(\widehat{\Omega})$ ($u_0\in L^p_{loc}(\widehat{R^n})$ respectively) for some constant $p>\frac{n(1-m)}{2}$ such that \eqref{u0-lower-blow-up-rate} holds
for some constants $0<\delta_1<\delta_0$, $\lambda_1$, $\cdots$, $\lambda_{i_0}\in\R^+$ and $\gamma_1,\cdots,\gamma_{i_0}\in \left(\frac{2}{1-m},\infty\right)$.
 We also find the blow-up rate of such solutions near the blow-up points $a_1,a_2,\dots,a_{i_0}$, and obtain the asymptotic large time behaviour of such singular solutions. 
 
We find that the asymptotic large time behaviour of such  solutions  depends on the blow-up rate of the initial value $u_0$ at the singular points $a_1, a_2,\dots, a_{i_0}$, and the lower bound of $u_0$. 
We prove that if the initial value satisfies $u_0\ge\mu_0$ on $\widehat{\Omega}$ ($\widehat{R^n}$, respectively) for some constant $\mu_0>0$ and \eqref{u0-lower-blow-up-rate} holds
for some constants $0<\delta_1<\delta_0$, $\lambda_1$, $\cdots$, $\lambda_{i_0}\in\R^+$, 
and
\begin{equation}\label{gamma1-large}
\gamma_1>\frac{n-2}{m},\quad
\gamma_i>\frac{2}{1-m}\quad\forall i=2,\dots,i_0,
\end{equation}
then the singular solution  converges locally uniformly on every compact subset of $\widehat{\Omega}$ (or $\widehat{R^n}$ respectively) to infinity as $t\to\infty$. 

When $u_0\ge\mu_0$ on $\widehat{\Omega}$ ($\widehat{R^n}$, respectively) for some constant $\mu_0>0$ satisfies satisfy \eqref{u0-lower-blow-up-rate} and 
\begin{equation}\label{upper-blow-up-rate-initial-data}
u_0(x)\leq \frac{\lambda_i'}{|x-a_i|^{\gamma_i'}} \qquad \forall 0<|x-a_i|<\delta_3,i=1, \cdots, i_1
\end{equation}
 with $i_1=i_0$ for some constants $0<\delta_3<\delta_1<\delta_0$, $\lambda_1$, $\cdots$, $\lambda_{i_0}$, $\lambda_1'$, $\cdots$, $\lambda_{i_0}'\in\R^+$, and
\begin{equation}\label{gamma-upper-lower-bd3}
\frac{2}{1-m}<\gamma_i\le\gamma_i'<\frac{n-2}{m}\quad\forall i=1,2,\dots,i_0,
\end{equation}   
we prove that $u$ converges  converges in $C^2(K)$ for any compact subset $K$  of $\2{\Omega}\setminus\{a_1,a_2,\dots,a_{i_0}\}$ (or $\widehat{R^n}$ respectively) to a harmonic function as $t\to\infty$. More precisely we prove the following existence and convergence results.

\begin{thm}\label{first-main-existence-thm}
Let $n\geq 3$, $0<m<\frac{n-2}{n}$, $0<\delta_1<\delta_0$, $0\le f\in L^{\infty}(\partial\Omega\times[0,\infty))$ and $0\leq u_0\in L_{loc}^p(\2{\Omega}\setminus\{a_1,\cdots,a_{i_0}\})$ for some constant  $p>\frac{n(1-m)}{2}$ be such that \eqref{u0-lower-blow-up-rate}
holds for some constants $\lambda_1$, $\cdots$, $\lambda_{i_0}\in\R^+$ and $\gamma_1,\cdots,\gamma_{i_0}\in \left(\frac{2}{1-m},\infty\right)$. Then there exists a solution $u$ of 
\begin{equation}\label{Dirichlet-blow-up-problem}
\begin{cases}
\begin{aligned}
u_t&=\La u^m \quad \mbox{in $\widehat{\Omega}\times(0,\infty)$}\\
u&=f \qquad \mbox{ on $\partial\Omega\times(0,\infty)$}\\
u(a_i,t)&=\infty \qquad \forall t>0, i=1,\cdots,i_0\\
u(x,0)&=u_0(x) \quad \mbox{in $\widehat{\Omega}$}
\end{aligned}
\end{cases}
\end{equation}
such that for any $T>0$ and $\delta_2\in (0,\delta_1)$ there exists a constant $C_1>0$  such that
\begin{equation}\label{eq-lower-limit-of-u-near-blow-up-points}
u(x,t) \geq \frac{C_1}{|x-a_i|^{\gamma_i}}\quad \forall 0<\left|x-a_i\right|<\delta_2, 0<t<T.
\end{equation}
Moreover if there exists a constant $T_0\ge 0$ such that 
\begin{equation}\label{eq-condition-monotone-decreasingness-of-f2-34}
f(x,t)\mbox{ is monotone decreasing in $t$ on }\partial\Omega\times(T_0,\infty),
\end{equation} 
then $u$ satisfies
\begin{equation}\label{Aronson-Bernilan-ineqn}
u_t\leq\frac{u}{(1-m)(t-T_0)} \quad \mbox{ in }\widehat{\Omega}\times (T_0,\infty).
\end{equation}
\end{thm}

\begin{thm}\label{second-main-existence-thm}
Let $n\geq 3$, $0<m<\frac{n-2}{n}$, $0<\delta_1<\delta_0$ and $0\leq u_0\in L^p_{loc}\left(\widehat{\R^n}\right)$ for some constant $p>\frac{n(1-m)}{2}$ such that  \eqref{u0-lower-blow-up-rate} holds for some constants $\lambda_1$, $\cdots$, $\lambda_{i_0}\in\R^+$ and $\gamma_1,\cdots,\gamma_{i_0}\in \left(\frac{2}{1-m},\infty\right)$. Then there exists a solution $u$ of 
\begin{equation}\label{cauchy-blow-up-problem}
\begin{cases}
\begin{aligned}
u_t&=\La u^m \quad \mbox{in $\widehat{\R^n}\times(0,\infty)$}\\
u(a_i,t)&=\infty \qquad \forall i=1,\cdots,i_0,\,\,t>0\\
u(x,0)&=u_0(x) \quad \mbox{in $\widehat{\R^n}$}
\end{aligned}
\end{cases}
\end{equation}
such that for any $T>0$ and $\delta_2\in (0,\delta_1)$ there exists a constant $C_1>0$  such that
\eqref{eq-lower-limit-of-u-near-blow-up-points} and
\begin{equation}\label{eq-Aronson-Bernilan-on-R-n}
u_t\leq\frac{u}{(1-m)t} \quad \mbox{ in }\,\,\widehat{\R^n}\times (0,\infty)
\end{equation}
hold.
\end{thm}

\begin{thm}\label{convergence-thm1}
Suppose that $n\geq 3$, $0<m<\frac{n-2}{n}$ and $\mu_0>0$. Let $\mu_0\leq u_0\in L^{p}_{loc}\left(\overline{\Omega}\bs\left\{a_1,\cdots,a_{i_0}\right\}\right)$ for some constant $p>\frac{n(1-m)}{2}$ satisfy \eqref{u0-lower-blow-up-rate} and \eqref{upper-blow-up-rate-initial-data}
with $i_1=i_0$ for some constants  $0<\delta_3<\delta_1<\min (\delta_0,1)$, $\lambda_1$, $\cdots$, $\lambda_{i_0}$, $\lambda_1'$, $\cdots$, $\lambda_{i_0}'\in\R^+$ and  
\begin{equation}\label{gamma-upper-lower-bd1}
\frac{2}{1-m}<\gamma_i\le\gamma_i'<n\quad\forall i=1,2,\dots,i_0.
\end{equation} 
Let $f\in L^{\infty}(\partial\Omega\times\left(0,\infty\right))\cap C^3(\partial\Omega\times\left(T_1,\infty\right))$ for some constant $T_1>0$ satisfy
\begin{equation}\label{eq-first-condition-on-f-constant-mu-sub-zero}
f\geq\mu_0 \quad \mbox{ on }\partial\Omega\times\left(0,\infty\right)
\end{equation}
and
\begin{equation}\label{eq-second-condition-on-f-convergence-to-mu-sub-one}
f(x,t)\to\mu_0 \quad \mbox{uniformly in }\,\, C^3(\partial\Omega)\mbox{ as }t\to\infty.
\end{equation}
Let $u$ be the solution of \eqref{Dirichlet-blow-up-problem} given by Theorem \ref{first-main-existence-thm}. 
Then
\begin{equation}\label{u-infty-to-mu0}
u(x,t)\to \mu_0 \quad \mbox{ in }C^2(K) \quad \mbox{ as }t\to\infty
\end{equation}
for any compact subset $K$ of $\overline{\Omega}\bs\left\{a_1,\cdots,a_{i_0}\right\}$.
\end{thm}

\begin{thm}\label{convergence-thm2}
Suppose that $n\geq 3$, $0<m<\frac{n-2}{n+2}$ and $\mu_1\geq\mu_0>0$. Let $\mu_0\leq u_0\in L^{p}_{loc}\left(\overline{\Omega}\bs\left\{a_1,\cdots,a_{i_0}\right\}\right)$  for some constant $p>\frac{n(1-m)}{2}$ satisfy \eqref{u0-lower-blow-up-rate} and \eqref{upper-blow-up-rate-initial-data} with $i_1=i_0$ for some constants  $0<\delta_3<\delta_1<\min(1,\delta_0)$, $\lambda_1$, $\cdots$, $\lambda_{i_0}$  and  
\begin{equation}\label{gamma-upper-lower-bd2}
\frac{2}{1-m}<\gamma_i\le\gamma_i'<\frac{n}{m+1}\quad\forall i=1,2,\dots,i_0.
\end{equation} 
Let $f\in L^{\infty}(\partial\Omega\times\left(0,\infty\right))\cap C^3(\partial\Omega\times\left(T_1,\infty\right))$ and $f_t\in L^1\left(\partial\Omega\times(T_1,\infty)\right)$ for some constant $T_1>0$  satisfy \eqref{eq-first-condition-on-f-constant-mu-sub-zero} and
\begin{equation}\label{eq-for-thm-a-second-condition-on-f-convergence-to-mu-sub-one}
 f(x.t)\to\mu_1 \mbox{ uniformly in }\,\, C^3(\partial\Omega)\mbox{ as }t\to\infty.
\end{equation}
Let $u$ be the solution of \eqref{Dirichlet-blow-up-problem} given by Theorem \ref{first-main-existence-thm}. Then
\begin{equation}\label{eq-for-thm-a-global-behaviour-of-solution-u-blow-up=0-above-by}
u(x,t)\to \mu_1 \quad\mbox{ in }C^2(K)\mbox{ as }t\to\infty
\end{equation}
for any compact subset $K$ of $\overline{\Omega}\bs\left\{a_1,\cdots,a_{i_0}\right\}$.
\end{thm}

By using the Aronson-Bernilan inequality \eqref{Aronson-Bernilan-ineqn} we also prove the following extension of Theorem \ref{convergence-thm2}.

\begin{thm}\label{convergence-thm3}
Suppose that $n\geq 3$, $0<m<\frac{n-2}{n}$ and $\mu_0>0$. Let $\mu_0\leq u_0\in L^{p}_{loc}\left(\overline{\Omega}\bs\left\{a_1,\cdots,a_{i_0}\right\}\right)$ for some constant $p>\frac{n(1-m)}{2}$ satisfy \eqref{u0-lower-blow-up-rate} and \eqref{upper-blow-up-rate-initial-data} with $i_1=i_0$ for some constants satisfying \eqref{gamma-upper-lower-bd3} and $0<\delta_3<\delta_1<\delta_0$, $\lambda_1$, $\cdots$, $\lambda_{i_0}$, $\lambda_1'$, $\cdots$, $\lambda_{i_0}'\in\R^+$. Let $f\in L^{\infty}(\partial\Omega\times\left(0,\infty\right))\cap C^3(\partial\Omega\times\left(T_1,\infty\right))$ for some constant $T_1>0$ satisfy \eqref{eq-first-condition-on-f-constant-mu-sub-zero}
and
\begin{equation}\label{eq-for-thm-a-second-condition-on-f-convergence-to-mu-sub-one-general}
f(x.t)\to g(x) \qquad\mbox{uniformly in $C^3\left(\partial\Omega\right)$ as $t\to\infty$}.
\end{equation}
for some function $g\in C^{3}\left(\partial\Omega\right)$, $g\geq\mu_0$ on $\partial\Omega$. Let $u$ be the solution of \eqref{Dirichlet-blow-up-problem} given by Theorem \ref{first-main-existence-thm}. Let $\phi$ be the solution of 
\begin{equation}\label{harmonic-eqn}
\begin{cases}
\begin{aligned}
\La\phi&=0 \quad\,\,\,\, \mbox{ in }\Omega\\
\phi&=g^m \quad \mbox{ on }\partial\Omega.
\end{aligned}
\end{cases}
\end{equation}
Then 
\begin{equation}\label{u-phi-1/m-limit1}
u(x,t)\to \phi^{\frac{1}{m}} \quad \mbox{ in } C^2(K) \quad \mbox{ as }t\to\infty
\end{equation}
for any compact subset $K$ of $\overline{\Omega}\bs\left\{a_1,\cdots,a_{i_0}\right\}$.
\end{thm}

\begin{thm}\label{convergence-thm4}
Suppose that $n\geq 3$, $0<m<\frac{n-2}{n}$ and $\mu_0>0$. Let $\mu_0\leq u_0\in L^{p}_{loc}\left(\widehat{\R^n}\right)$  for some constant $p>\frac{n(1-m)}{2}$ satisfy \eqref{u0-lower-blow-up-rate}, \eqref{upper-blow-up-rate-initial-data} with $i_1=i_0$ for some constants  satisfying \eqref{gamma-upper-lower-bd1}, $0<\delta_3<\delta_1<\delta_0$, $\lambda_1$, $\cdots$, $\lambda_{i_0}$, $\lambda_1'$, $\cdots$, $\lambda_{i_0}'\in\R^+$.  Suppose there exist constants $R_1>R_0$ and $C_1>0$ such that
\begin{equation}\label{u0-infty-behaviour}
u_0(x)\leq C_1\qquad \forall |x|\ge R_1
\end{equation} 
and 
\begin{equation}\label{eq-initial-condition-over-ball-L-1-bound-of-u-0-and-mu-0}
\int_{\R^n\setminus B_{R_1}}|u_0-\mu_0|\,dx<\infty
\end{equation}
hold. Let $u$ be the solution of \eqref{cauchy-blow-up-problem} given by Theorem \ref{second-main-existence-thm}. Then \eqref{u-infty-to-mu0} holds for any compact subset $K$ of $\widehat{\R^n}$.
\end{thm}

\begin{thm}\label{convergence-thm5}
Suppose that $n\geq 3$, $0<m<\frac{n-2}{n}$ and $\mu_0>0$. Let $\mu_0\leq u_0\in L^{p}_{loc}\left(\widehat{\R^n}\right)$  for some constant $p>\frac{n(1-m)}{2}$ satisfy \eqref{u0-lower-blow-up-rate} and  \eqref{upper-blow-up-rate-initial-data} with $i_1=i_0$ for some constants satisfying \eqref{gamma-upper-lower-bd3} and $0<\delta_3<\delta_1<\delta_0$, $\lambda_1$, $\cdots$, $\lambda_{i_0}$, $\lambda_1'$, $\cdots$, $\lambda_{i_0}'\in\R^+$.  Suppose $u_0$ satisfies
\begin{equation}\label{eq-initial-condition-of-u-0-to-mu-0-at-infty}
u_0\left(r\sigma\right)\to\mu_0 \qquad \mbox{uniformly in $S^{n-1}$ as $|x|=r\to\infty$.}
\end{equation}
Let $u$ be the solution of \eqref{cauchy-blow-up-problem} given by Theorem \ref{second-main-existence-thm}. Then \eqref{u-infty-to-mu0} holds for any compact subset $K$ of $\widehat{\R^n}$.
\end{thm}

\begin{thm}\label{convergence-thm6}
Suppose that $n\geq 3$, $0<m<\frac{n-2}{n}$ and $\mu_0>0$. Let $\mu_0\leq u_0\in L^p_{loc}\left(\widehat{\Omega}\right)$ for some constant $p>\frac{n(1-m)}{2}$ satisfy \eqref{u0-lower-blow-up-rate} for some constants satisfying \eqref{gamma1-large} and $0<\delta_1<\delta_0$, $\lambda_1$, $\cdots$, $\lambda_{i_0}\in\R^+$ and let $f\in L^{\infty}(\partial\Omega\times(0,\infty))$ satisfy \eqref{eq-first-condition-on-f-constant-mu-sub-zero}. Let $u$ be the solution of \eqref{Dirichlet-blow-up-problem}  given by Theorem \ref{first-main-existence-thm}. Then
\begin{equation}\label{eq-blow-up-of-soluiton-u-when-gamma-greater-than-n-2-over-m-2c}
u(x,t)\to\infty \quad \mbox{ on } K\quad \mbox{ as }t\to\infty
\end{equation}
for any compact subset $K$ of $\widehat{\Omega}$. 
\end{thm}

\begin{thm}\label{convergence-thm7}
Suppose that $n\geq 3$, $0<m<\frac{n-2}{n}$ and $\mu_0>0$. Let $\mu_0\leq u_0\in L^p_{loc}\left(\widehat{\R^n}\right)$ for some constant $p>\frac{n(1-m)}{2}$ satisfy \eqref{u0-lower-blow-up-rate} for some constants satisfying \eqref{gamma1-large} and $0<\delta_1<\delta_0$, $\lambda_1$, $\cdots$, $\lambda_{i_0}\in\R^+$. Let $u$ be the solution of \eqref{cauchy-blow-up-problem} given by Theorem \ref{second-main-existence-thm}. Then \eqref{eq-blow-up-of-soluiton-u-when-gamma-greater-than-n-2-over-m-2c} holds for any compact subset $K$ of $\widehat{\R^n}$. 
\end{thm}

\begin{remar}
In the paper \cite{VW2} J.L.~Vazquez and M.~Winkler proved that for any smooth domain $\Omega\subset\R^n$ containing the origin and any constant $\mu_0>0$, $n\ge 3$ and $0<m<\frac{n-2}{n}$, there exists initial data $\mu_0\le u_0\in C^{\infty}(\2{\Omega}\bs\{0\})$, $u_0(x)=\mu_0$ for all $x\in\1\Omega$,  such that the problem
\begin{equation*}
\begin{cases}
\begin{aligned}
u_t&=\Delta u^m\quad \mbox{ in }(\Omega\setminus\{0\})\times(0,\infty)\\
u(x,t)&=\mu_0 \qquad  \mbox{ on }\partial\Omega\times(0,\infty)\\
u(0,t)&=\infty \qquad  \forall t>0\\
u(x,0)&=u_0(x) \quad  \mbox{ in }\Omega\setminus\{0\}
\end{aligned}
\end{cases}
\end{equation*}
has a solution that oscillates between the values $\mu_0$ and $\infty$ in $L_{loc}^{\infty}(\Omega\setminus\{0\})$ as $t\to\infty$. We conjecture that similar result should hold for the family of solutions of \eqref{fde} which blows up at a finite number of points in the domain. More precisely we conjecture that for any constant $\mu_0>0$, $n\ge 3$ and $0<m<\frac{n-2}{n}$, and any set of points $a_1,\dots, a_{i_0}\in\Omega$ there exists initial data $\mu_0\le u_0\in C^{\infty}(\2{\Omega}\bs\{0\})$, $u_0(x)=\mu_0$ for all $x\in\1\Omega$,  such that the problem
\begin{equation*}
\begin{cases}
\begin{aligned}
u_t&=\Delta u^m\quad \mbox{ in }\widehat{\Omega}\times(0,\infty)\\
u(x,t)&=\mu_0 \qquad  \mbox{ on }\partial\Omega\times(0,\infty)\\
u(a_i,t)&=\infty \qquad  \forall t>0, i=1,\cdots,i_0\\
u(x,0)&=u_0(x) \quad  \mbox{ in }\widehat{\Omega}
\end{aligned}
\end{cases}
\end{equation*}
has a solution that oscillates between the values $\mu_0$ and $\infty$ in $L_{loc}^{\infty}(\widehat{\Omega})$ as $t\to\infty$.
\end{remar}

The plan of the paper is as follows.  In section \ref{existence-soln-section} we will  prove the existence of solutions of \eqref{Dirichlet-blow-up-problem} and \eqref{cauchy-blow-up-problem} with initial data $u_0$ that blows up at a finite number of points in the domain. We also obtain the blow-up rate near the blow-up points and prove that the singularities of the solutions of \eqref{Dirichlet-blow-up-problem} and \eqref{cauchy-blow-up-problem} are preserved for all positive time. We will prove the asymptotic large time behaviour of solutions of \eqref{Dirichlet-blow-up-problem} and \eqref{cauchy-blow-up-problem} in section \ref{Asymptotic-behaviour-of-solution}.

We start with some definitions. For any $t_2>t_1$, we say that $u$ is a solution of \eqref{fde} in $\Omega\times(t_1,t_2)$ if $u$ is positive in $\Omega\times(t_1,t_2)$ and satisfies \eqref{fde} in $\Omega\times(t_1,t_2)$ in the classical sense. For any $T>0$, $0\leq f\in L^1(\partial\Omega\times(0,T))$ and $0\le u_0\in L^1_{loc}(\widehat{\Omega})$, we say that $u$ is a solution (subsolution, supersolution respectively) of 
\begin{equation}\label{fde-Dirichlet-blow-up-problem}
\begin{cases}
\begin{aligned}
u_t&=\La u^m \qquad \mbox{in $\widehat{\Omega}\times(0,T)$}\\
u&=f \qquad\quad \mbox{on $\partial\Omega\times(0,T)$}\\
u(x,0)&=u_0(x) \quad\,\,\, \mbox{in $\widehat{\Omega}$}
\end{aligned}
\end{cases}
\end{equation}
if $u$ is positive in $\widehat{\Omega}\times(0,T)$ and satisfies \eqref{fde} in $\widehat{\Omega}\times(0,T)$ ($\le, \ge$ respectively) in the classical sense with
\begin{equation}\label{u-initial-value}
\|u(\cdot,t)-u_0\|_{L^1(K)}\to 0\quad \mbox{ as }t\to 0 
\end{equation}
for any compact set $K\subset\widehat{\Omega}$ and
\begin{equation}\label{eq-formula-for-weak-sols-of-main-problem}
\begin{aligned}
&\int_{t_1}^{t_2}\int_{\widehat{\Omega}}\left(u\eta_t+u^m\La\eta\right)\,dxdt\\
&\qquad \qquad =\int_{t_1}^{t_2}\int_{\partial\Omega}f^m\frac{\partial\eta}{\partial\nu}\,d\sigma dt+\int_{\widehat{\Omega}}u(x,t_2)\eta(x,t_2)\,dx-\int_{\widehat{\Omega}}u(x,t_1)\eta(x,t_1)\,dx
\end{aligned}
\end{equation}
($\geq, \leq$ respectively) for any $0<t_1<t_2<T$ and $\eta\in C_c^2((\overline{\Omega}\bs\left\{a_1,\cdots,a_{i_0}\right\})\times(0,T))$ satisfying $\eta\equiv 0$ on $\partial\Omega\times(0,T)$.  We say that $u$ is a solution (subsolution, supersolution respectively) of \eqref{Dirichlet-blow-up-problem} if $u$ is a solution of \eqref{fde-Dirichlet-blow-up-problem} for $T=\infty$ and satisfies \eqref{u-blow-up-infty}.

For any $0\leq f\in L^1(\partial\Omega\times(0,T))$ and $0\le u_0\in L^{1}_{loc}(\Omega)$, we say that $u$ is a solution (subsolution, supersolution respectively) of 
\begin{equation}\label{fde-Dirichlet-problem}
\begin{cases}
\begin{aligned}
u_t&=\La u^m \qquad \mbox{in $\Omega\times(0,T)$}\\
u(x,t)&=f \qquad \quad \mbox{on $\partial\Omega\times(0,T)$}\\
u(x,0)&=u_0(x) \quad\,\,\, \mbox{in $\Omega$}.
\end{aligned}
\end{cases}
\end{equation}
if $u$ is positive in $\Omega\times(0,T)$ and satisfies \eqref{fde} in $\Omega\times(0,T)$ ($\leq, \geq$ respectively) in the classcal sense, \eqref{u-initial-value} holds
for any compact set $K\subset\Omega$ and
\begin{equation}\label{eq-alinged-formula-for-weak-solution-of-u}
\int_{t_1}^{t_2}\int_{\Omega}(u\eta_t+u^m\La\eta)\,dx\,dt=\int_{t_1}^{t_2}\int_{\partial\Omega}f^m\frac{\partial\eta}{\partial\nu}\,d\sigma\, dt+\int_{\Omega}u\eta\,dx\bigg|_{t_1}^{t_2}
\end{equation}
($\geq, \leq$ respectively) for any $0<t_1<t_2<T$ and $\eta\in C_c^2(\overline{\Omega}\times(0,T))$ satisfying $\eta=0$ on $\partial\Omega\times(0,T)$. 

For any $0\le u_0\in L^1_{loc}(\widehat{\R^n})$ we say that $u$ is a solution (subsolution, supersolution respectively) of \eqref{cauchy-blow-up-problem} if $u$ is positive in $\widehat{\R^n}\times(0,\infty)$ and satisfies \eqref{fde} in $\widehat{\Omega}\times(0,\infty)$ ($\leq, \geq$ respectively) in the classical sense and \eqref{u-blow-up-infty}, \eqref{u-initial-value}, holds for any compact set $K\subset\widehat{\R^n}$. We say that $u$ is a solution of
\begin{equation*}
\begin{cases}
\begin{aligned}
u_t&=\La u^m \qquad \mbox{in $\Omega\times (0,\infty)$}\\
u(x,t)&=\infty \quad \quad\,\,\, \mbox{ on }\partial\Omega\times (0,\infty)\\
u(x,0)&=u_0(x) \quad\,\,\, \mbox{in $\Omega$}.
\end{aligned}
\end{cases}
\end{equation*}
if $u$ is positive in $\Omega\times(0,\infty)$ and satisfies \eqref{fde} in $\Omega\times (0,\infty)$ in the classical sense,  \eqref{u-initial-value} holds for any compact set $K\subset\Omega$, and
\begin{equation*}
\lim_{\substack{(y,s)\to (x,t)\\
(y,s)\in\Omega\times (0,\infty)}}u(y,s)=\infty\quad\forall (x,t)\in\1\Omega\times (0,\infty).
\end{equation*}
For any set $A\subset\R^n$, we let $\chi_A$ be the characteristic function of the set $A$. For any $a\in\R$, we let $a_+=\max 0,a)$. For any $x_0\in\R^n$, $\alpha>0$ and $h>0$, let
$\Gamma_{\alpha}(x_0)=\{(x,t):|x-x_0|<\alpha\sqrt{t}\}$ 
and $\Gamma_{\alpha}^h(x_0)=\{(x,t):(x,t)\in \Gamma_{\alpha}(x_0), 0<t<h\}$. 

\section{Existence of solutions and a priori estimates}\label{existence-soln-section}
\setcounter{equation}{0}
\setcounter{thm}{0}

In this section \ref{existence-soln-section} we will  use a modification of the technique of \cite{VW1} to prove the existence of solutions of \eqref{Dirichlet-blow-up-problem} and \eqref{cauchy-blow-up-problem} with initial data $u_0$ that blows up at a finite number of points in the domain.

For any non-negative functions $u_0\in L^1(\Omega)$ or $L^1(\R^n)$, $0\leq f\in L^{\infty}(\partial\Omega\times(0,\infty))$, and constants $M>0$, $0<\3<1$, let 
\begin{equation}\label{eq-def-of-u-0-epsilon-M}
\begin{cases}
\begin{aligned}
u_{0,M}(x)&=\min\left(u_0(x),M\right)\\
u_{0,\3,M}(x)&=\min\left(u_0(x),M\right)+\3  
\end{aligned}
\end{cases}
\end{equation}
and
\begin{equation}\label{eq-definition-of-f-epsilon-by-f}
f_{\3}(x,t)=f(x,t)+\3 \qquad \forall (x,t)\in\partial\Omega\times(0,\infty).
\end{equation}
By an argument similar to the proof of Lemma 2.3 of \cite{DaK}, we have the following result.
\begin{lemm}\label{comparison-lem-bd-domain}
Let $n\geq 1$, $0<m<1$ and $\Omega\subset\R^n$ be a smooth bounded domain. Let $0\leq u_{0,1}\leq u_{0,2}\in L^{\infty}(\Omega)$ and $0\leq f_1\le f_2\in L^{\infty}(\partial\Omega\times(0,T))$ be such that $\inf_{\partial\Omega_1\times(0,T)}f_i>0$ for $i=1,2$. Let $u_1$, $u_2\in L^{\infty}(\Omega\times(0,T))$ be subsolution and supersolution of \eqref{fde-Dirichlet-problem} with $f=f_1$, $f_2$ and $u_0=u_{0,1}$, $u_{0,2}$ respectively. If 
\begin{equation*}
\inf_{\Omega\times(0,T)}u_i>0\quad\forall i=1, 2,
\end{equation*} 
then
\begin{equation*}
u_1\leq u_2 \quad \mbox{ in }\Omega\times(0,T).
\end{equation*}
\end{lemm}

By Lemma \ref{comparison-lem-bd-domain} and an argument similar to the proof of Theorem 3.5 in \cite{Hu2}  and the proof of the Aronson-Benilan inequality in Theorem 2.2 of \cite{Hu3}, we have the following result.

\begin{lemm}\label{approx-soln-lem1}
Let $n\geq 1$, $0<m<1$, $0\le f\in L^{\infty}(\partial\Omega\times[0,\infty))$ and $0\leq u_0\in L^1_{loc}(\widehat{\Omega})$. Then for any $\3>0$ and $0<M<\infty$ there exists a unique solution $u_{\3,M}$ of 
\begin{equation}\label{eq-aligned-problem-related-with-epsilon-main-one}
\begin{cases}
\begin{aligned}
u_t&=\La u^m \qquad  \mbox{in $\Omega\times(0,\infty)$}\\
u(x,t)&=f_{\3} \qquad\,\,\, \mbox{ on $\partial\Omega\times(0,\infty)$}\\
u(x,0)&=u_{0,\3,M}  \quad \mbox{ in $\Omega$}.
\end{aligned}
\end{cases}
\end{equation}
which satisfies
\begin{equation}\label{eq-upper-and-lower-bound-of-u-epsilon-M-09}
\3\le u_{\3,M}\le\max\left(M,\left\|f\right\|_{L^{\infty}\left(\partial\Omega\times[0,\infty)\right)}\right)+\3.
\end{equation}
Moreover if there exists a constant $T_0>0$ such that $f(x,t)$ is monotone decreasing in $t$ on $\partial\Omega\times\left(T_0,\infty\right)$, then $u_{\3,M}$ also satisfies the Aronson-Benilan inequality,
\begin{equation}\label{Aronson-Benilan-ineqn2}
u_t\le\frac{u}{(1-m)(t-T_0)}
\end{equation} 
in $\Omega\times (T_0,\infty)$.
\end{lemm}

In the following lemma, we will construct an upper bound for the solution $u_{\epsilon, M}$.
\begin{lemm}\label{lem-cf-lemma-2-3-of-cite-HK1}
Let $n\ge 1$, $0<m<1$ and $0<\delta_3<\min(1,\delta_0)$. Let $0\le f\in L^{\infty}(\partial\Omega\times[0,\infty))$  and $0\leq u_0\in L_{loc}^1(\widehat{\Omega})$ satisfy \eqref{upper-blow-up-rate-initial-data} for some integer $1\le i_1\le i_0$ and constants $\lambda_1',\cdots,\lambda_{i_1}'\in\R^+$, $\gamma_i', \cdots, \gamma_{i_1}'\in\left[\frac{2}{1-m},\infty\right)$. For any $0<\3<1$ and $M>0$, let $u_{\3,M}$ be the solution of \eqref{eq-aligned-problem-related-with-epsilon-main-one} given by Lemma \ref{approx-soln-lem1}. Then  there exists a  constant $A_0>0$ such that
\begin{equation}\label{u-upper-bd-phi-tidle}
u_{\3,M}(x,t)\leq \phi_{i,A_0}(x-a_i,t) \quad \forall 0<|x-a_i|<\delta_3,t\ge 0, 0<\3<1, M>0, i=1,\cdots,i_1
\end{equation}
holds where
\begin{equation}\label{phi-A-defn}
\phi_{i, A_0}(x,t)=\frac{A_0(1+t)^{\frac{1}{1-m}}}{|x|^{\gamma_i'}(\delta_3-|x|)^{\frac{2}{1-m}}} \qquad \forall i=1,\cdots,i_1.
\end{equation}
\end{lemm}

\begin{proof}
We will use a modification of the proof of Lemma 2.3 of \cite{HK2} to prove the lemma. Without loss of generality it suffices to prove \eqref{u-upper-bd-phi-tidle} for $i=i_1=1$. Let
\begin{equation}\label{eq-condition-of-A-1-by-max-1}
A_0=\max\left\{\lambda_1'+1,\left(\frac{m\left((1-m)^2(m\gamma_1'+1)\gamma_1'+2(1-m)(n-1)+2(1+m)\right)}{1-m}\right)^{\frac{1}{1-m}}\right\}.
\end{equation}
Then by \eqref{upper-blow-up-rate-initial-data},
\begin{equation}\label{eq-compare-at-initial-time-between-phi-A-1-and-u-0-epsilon-M}
\phi_{1,A_0}(x-a_1,0)\geq u_0(x)+1 \geq u_{0,\3,M}(x) \qquad \forall 0<|x-a_1|<\delta_3,0<\3<1,M>0.
\end{equation}
Since
\begin{equation*}
\phi_{1, A_0}(x-a_1,t)=\frac{A_0(1+t)^{\frac{1}{1-m}}}{r^{\gamma_1'}(\delta_3-r)^{\frac{2}{1-m}}}\to\infty
\end{equation*}
uniformly on $t\in [0,\infty)$ when $r=|x-a_1|\to\delta_3^-$ or $r=|x-a_1|\to 0^+$,
there exist constants $\delta'$, $\delta''\in(0,\delta_3)$ such that
\begin{equation}\label{eq-condition-of-tilde-phi-on-boundary-1}
\phi_{1, A_0}(x-a_1,t)>\max\left(M,\|f\|_{L^{\infty}\left(\partial\Omega\times(0,\infty)\right)}\right)+1\qquad\forall x\in (\2{B_{\delta'}(a_1)}\setminus\{a_1\})\cup (B_{\delta_3}(a_1)\setminus B_{\delta''}(a_1)),t\geq 0.
\end{equation}
By \eqref{eq-condition-of-A-1-by-max-1},
\begin{align}\label{eq-condition-of-tilde-phi-in-interior-1}
\La\phi^m_{1, A_0}(x,t)=&A_0^m(1+t)^{\frac{m}{1-m}}\left[\frac{m\gamma_1'\left(m\gamma_1'-n+2\right)}{|x|^{m\gamma_1'+2}\left(\delta_3-|x|\right)^{\frac{2m}{1-m}}}+\frac{2m(n-1-2m\gamma_1')}{(1-m)|x|^{m\gamma_1'+1}(\delta_3-|x|)^{\frac{1+m}{1-m}}} +\frac{2m(1+m)}{(1-m)^2|x|^{m\gamma_1'}(\delta_3-|x|)^{\frac{2}{1-m}}}\right]\notag\\
\le&\frac{mA_0^m\left[(1-m)^2(m\gamma_1'+1)\gamma_1'+2(1-m)(n-1)+2(1+m)\right]}{(1-m)^2|x|^{m\gamma_1'+2}(\delta_3-|x|)^{\frac{2}{1-m}}}(1+t)^{\frac{m}{1-m}}\notag\\
\le&\phi_{A_0,t}(x,t)\qquad\qquad\qquad\qquad\qquad\qquad \forall 0<|x|<\delta_3, t\geq 0.
\end{align}
Let $\delta_1'\in (0,\delta')$ and $\delta_1''\in (\delta'',\delta_3)$. By \eqref{eq-compare-at-initial-time-between-phi-A-1-and-u-0-epsilon-M}, \eqref{eq-condition-of-tilde-phi-on-boundary-1} and \eqref{eq-condition-of-tilde-phi-in-interior-1}, $\phi_{A_0}(x-a_1,t)$ is a supersolution of
\begin{equation}\label{w-eqn}
\begin{cases}
w_t=\La w^m\qquad\qquad \qquad\qquad\qquad\,\,\mbox{ in }\left(B_{\delta_1''}(a_1)\bs B_{\delta_1'}(a_1)\right)\times(0,\infty)\\
w=\max\left(M,\|f\|_{L^{\infty}(\partial\Omega\times(0,\infty))}\right)+1 \quad\,\mbox{ on }\partial \left(B_{\delta_1''}(a_1)\bs B_{\delta_1'}(a_1)\right)\times(0,\infty)\\
w(x,0)=u_{0,\3,M}(x)\qquad\qquad\qquad\quad\mbox{ in }B_{\delta_1''}(a_1)\bs B_{\delta_1'}(a_1).
\end{cases}
\end{equation}
Since $u_{\3,M}$ is a subsolution of \eqref{w-eqn} for any $0<\3<1$ and $M>0$, by the Lemma \ref{comparison-lem-bd-domain},
\begin{equation}\label{eq-compare-between-u-epsilon-M-amd-phi-on-B-delta-bs-B-delta-and-t-1-rt-2}
u_{\3, M}(x,t)\leq \phi_{1,A_0}(x-a_1,t) \qquad \forall \delta_1'<|x-a_1|<\delta_1'',t\ge 0, 0<\3<1, M>0.
\end{equation}
Letting $\delta'\to 0$ and $\delta''\to\delta_3$ in \eqref{eq-compare-between-u-epsilon-M-amd-phi-on-B-delta-bs-B-delta-and-t-1-rt-2}, we get \eqref{u-upper-bd-phi-tidle} and the lemma follows.
\end{proof}

\begin{lemm}\label{lem-aligned-bound-of-integral-term-of-psi-to-some-power-and-La-psi}
Let $n\geq 3$, $0<m<\frac{n-2}{n}$, $\delta_0>\delta_1>\delta_2>0$ and $0\leq\eta\in C_c^{\infty}(B_{\delta_0})$ be such that 
\begin{equation}\label{eta-defn}
\begin{cases}
\begin{aligned}
\eta(x)&>0\quad\quad\,\, \forall  |x|<\delta_1\\
\eta(x)&=0 \quad\quad\,\, \forall \delta_1\leq|x|\leq\delta_0\\
\eta(x)&=\delta_2^{-\beta_1/b_1} \quad \forall |x|=\delta_2
\end{aligned}
\end{cases}
\end{equation}
for some constants $b_1>\frac{2}{1-m}$ and $\beta_1\in\left[0,n-\frac{2}{1-m}\right)$. Let $\psi\in C^{\infty}(\widehat{B_{\delta_0}})$ be such that
\begin{equation}\label{psi-defn}
\begin{cases}
\begin{aligned}
\psi(x)&=|x|^{-\beta_1}\quad\, \forall 0<|x|\leq \delta_2\\
\psi(x)&=\eta^{b_1}(x) \quad \forall \delta_2\leq |x|\leq \delta_0.
\end{aligned}
\end{cases}
\end{equation}
Then there exists a constant $C_{\psi}>0$ such that
\begin{equation}\label{eq-aligned-bound-of-integral-term-of-psi-to-some-power-and-La-psi}
\int_{\widehat{B}_{\delta_1}}\psi^{-\frac{m}{1-m}}\left|\La\psi\right|^{\frac{1}{1-m}}dx\leq C_{\psi}<\infty.
\end{equation}
\end{lemm}
\begin{proof}
By \eqref{psi-defn},
\begin{equation}\label{eq-aligned-split-of-integral-term-of-psi-to-some-power-and-La-psi-intwo-tow}
\begin{aligned}
\int_{\widehat{B}_{\delta_1}}\psi^{-\frac{m}{1-m}}\left|\La\psi\right|^{\frac{1}{1-m}}dx&=\int_{\widehat{B}_{\delta_2}}|x|^{\frac{m\beta_1}{1-m}}\left|\La\left(|x|^{-\beta_1}\right)\right|^{\frac{1}{1-m}}dx+\int_{B_{\delta_1}\bs B_{\delta_2}}\eta^{-\frac{mb_1}{1-m}}\left|\La\left(\eta^{b_1}\right)\right|^{\frac{1}{1-m}}dx\\
&=:I_1+I_{2}.
\end{aligned}
\end{equation}
Since $\beta_1+\frac{2}{1-m}<n$,
\begin{equation}\label{eq-aligned-computation-for-I-1-in-integral-term-of-psi-La-psi}
\begin{aligned}
I_1=\beta_1^{\frac{1}{1-m}}(n-\beta_1-2)^{\frac{1}{1-m}}\int_{\widehat{B}_{\delta_2}}|x|^{-\beta_1-\frac{2}{1-m}}dx=\frac{\omega_n\beta_1^{\frac{1}{1-m}}(n-\beta_1-2)^{\frac{1}{1-m}}}{n-\beta_1-\frac{2}{1-m}}\cdot\delta_2^{n-\beta_1-\frac{2}{1-m}}
\end{aligned}
\end{equation}
and
\begin{equation}\label{eq-aligned-computation-for-I-2-in-integral-term-of-psi-La-psi}
\begin{aligned}
I_2&=b_1^{\frac{1}{1-m}}\int_{B_{\delta_1}\bs B_{\delta_2}}\eta^{b_1-\frac{2}{1-m}}\left|\eta\La\eta+\left(b_1-1\right)\left|\nabla\eta\right|^2\right|^{\frac{1}{1-m}}dx\\
&\leq b_1^{\frac{1}{1-m}}\left\|\eta\right\|^{b_1-\frac{2}{1-m}}_{L^{\infty}}\left(\left\|\eta\right\|_{L^{\infty}}\left\|\La\eta\right\|_{L^{\infty}}+\left(b_1-1\right)\left\|\nabla\eta\right\|^2_{L^{\infty}}\right)^{\frac{1}{1-m}}\left|B_{\delta_1}\bs B_{\delta_2}\right|.
\end{aligned}
\end{equation}
where $\omega_n$ is the surface area of the unit sphere $S^{n-1}$ in $\R^n$. By \eqref{eq-aligned-split-of-integral-term-of-psi-to-some-power-and-La-psi-intwo-tow}, \eqref{eq-aligned-computation-for-I-1-in-integral-term-of-psi-La-psi} and \eqref{eq-aligned-computation-for-I-2-in-integral-term-of-psi-La-psi}, \eqref{eq-aligned-bound-of-integral-term-of-psi-to-some-power-and-La-psi} holds for some constant $C_{\psi}>0$ and the lemma follows.
\end{proof}

\begin{lemm}\label{lem-similar-L-1-contraction-for-positivity-of-sol-u-epsilon-M}
Let $n\ge 3$, $0<m<\frac{n-2}{n}$, $0\le f\in L^{\infty}(\partial\Omega\times[0,\infty))$  and  $0\le u_0\in L^1_{loc}(\widehat{\Omega})$.  Let $\psi$ be given by \eqref{psi-defn} with $\eta$ given by \eqref{eta-defn} for some constants $0<\delta_2<\delta_1<\delta_0$, $b_1>\frac{2}{1-m}$ and $\beta_1\in\left[0,n-\frac{2}{1-m}\right)$. Let  $\psi_{x_0}(x)=\psi(x-x_0)$ for any $x_0\in\Omega$. Then there exists a constant $C_1>0$ such that for any $\3\in\left(0,1\right)$ and $M>0$ the solution $u_{\3,M}$ of \eqref{eq-aligned-problem-related-with-epsilon-main-one} satisfies
\begin{equation}\label{soln-loc-integral-lower-bd1}
\int_{\widehat{B}_{\delta_1}\left(a_i\right)}u_{\3,M}(x,t)\psi_{a_i}(x)\,dx\geq e^{-t}\int_{\widehat{B}_{\delta_1}\left(a_i\right)}u_{0,\3,M}(x)\psi_{a_i}(x)\,dx-C_1\left(1-e^{-t}\right) \qquad \forall t>0,\,\,i=1,2,\cdots,i_0.
\end{equation}
\end{lemm}

\begin{proof}
We will use a modification of the proof of Lemma 3.5 of \cite{VW1} to prove the lemma. Without loss of generality it suffices to prove \eqref{soln-loc-integral-lower-bd1} for $i=i_0=1$ and $a_1=\left(0,\cdots,0\right)$. For any $\3\in\left(0,1\right)$ and $M>0$ let $v_{\3,M}$ be the solution of \eqref{eq-aligned-problem-related-with-epsilon-main-one} with $f_{\3}$ being replaced by $\3$ that satisfies 
\begin{equation}\label{v-epsilon-m-lower-upper-bd0}
\3\le v_{\3,M}\le M+\3.
\end{equation} 
Then by Lemma \ref{comparison-lem-bd-domain},
\begin{equation}\label{v-u-epsilon-m-compare}
u_{\3,M}\ge v_{\3,M}\quad\mbox{ in }\widehat{\Omega}\times (0,\infty)\quad\forall 0<\3<1,M>0.
\end{equation}
By \eqref{eq-aligned-problem-related-with-epsilon-main-one} and the Green theorem for any $0<\delta<\delta_2$ and $t>0$,
\begin{align}\label{eq-integration-after-multi-test-function}
\frac{d}{dt}\left(\int_{B_{\delta_1}\bs B_{\delta}}v_{\3,M}(x,t)\psi_{a_1}(x)\,dx\right)
=&\int_{B_{\delta_1}\bs B_{\delta}} v^m_{\3,M}(x,t)\La\psi_{a_1}(x)\,dx
-\int_{\partial B_{\delta}}\psi_{a_1}(\sigma) \frac{\partial v^m_{\3,M}(\sigma,t)}{\partial r}\,d\sigma\notag\\
&\qquad +\int_{\partial B_{\delta}} v^m_{\3,M}(\sigma,t)\frac{\partial\psi_{a_1}(\sigma)}{\partial r}\,d\sigma\notag\\
=&:I_1+I_2+I_3.
\end{align}
By \eqref{v-epsilon-m-lower-upper-bd0} the equation \eqref{fde} for $v_{\3,M}$ is uniformly parabolic on $\2{\Omega}\times[0,\infty)$. Hence by the parabolic Schauder estimates \cite{LSU} for any $t_1>0$ there exists a constant $C_1>0$ such that
\begin{equation}\label{eq-nabla-u-ep-M-bounded-by-constant-depending-on-ep-and-M}
|\nabla v_{\3,M}|\leq C_1 \qquad \mbox{in $B_{\delta_1}\times[t_1,\infty)$}.
\end{equation}
By \eqref{psi-defn}, \eqref{v-epsilon-m-lower-upper-bd0}  and \eqref{eq-nabla-u-ep-M-bounded-by-constant-depending-on-ep-and-M},
\begin{equation}\label{eq-control-of-I-2-as-delta-to-zero}
|I_2|\leq mC_1\3^{m-1}\omega_n\delta^{n-\beta_1-1}\le C_2\delta^{n-\beta_1-2} \qquad \forall 0<\delta<\delta_2,\,\, t_1\leq t<T.
\end{equation}
and
\begin{equation}\label{eq-control-of-I-3-as-delta-to-zero}
|I_3|\leq C_2\delta^{n-\beta_1-2} \qquad \forall 0<\delta<\delta_2,\,\, 0<t<T
\end{equation}
for some constant $C_2>0$. By Lemma \ref{lem-aligned-bound-of-integral-term-of-psi-to-some-power-and-La-psi} and Young's inequality,
\begin{align}\label{eq-aligned-lower-bound-of-I-1=after-youngs-inequality}
\left|I_1\right|&\leq \left(\int_{B_{\delta_1}\bs B_{\delta}}v_{\3,M}(x,t)\psi_{a_1}(x)\,dx\right)^{m}\left(\int_{B_{\delta_1}\bs B_{\delta}}\psi_{a_1}^{-\frac{m}{1-m}}\left|\La\psi_{a_1}\right|^{\frac{1}{1-m}}dx\right)^{1-m}\notag\\
&\leq \int_{B_{\delta_1}\bs B_{\delta}}v_{\3,M}(x,t)\psi_{a_1}(x)\,dx+\int_{\widehat{B}_{\delta_1}}\psi_{a_1}^{-\frac{m}{1-m}}\left|\La\psi_{a_1}\right|^{\frac{1}{1-m}}dx\notag\\
&\leq \int_{B_{\delta_1}\bs B_{\delta}}v_{\3,M}(x,t)\psi_{a_1}(x)\,dx+C_{\psi}, \qquad  \forall0<\delta<\delta_1, 0<t<T.
\end{align}
By\eqref{eq-integration-after-multi-test-function}, \eqref{eq-control-of-I-2-as-delta-to-zero}, \eqref{eq-control-of-I-3-as-delta-to-zero} and \eqref{eq-aligned-lower-bound-of-I-1=after-youngs-inequality},
\begin{equation}\label{eq-after-applying-conclusion-of-I1-I2-I3}
\frac{d}{dt}\left(\int_{B_{\delta_1}\bs B_{\delta}}v_{\3,M}(x,t)\psi_{a_1}(x)\,dx\right)\geq -\int_{B_{\delta_1}\bs B_{\delta}} v_{\3,M}(x,t)\psi_{a_1}(x)\,dx-\left(C_{\psi}+2C_2\delta^{n-\beta_1-2}\right) \quad \forall 0<\delta<\delta_2,t_1\leq t<T.
\end{equation}
Integrating \eqref{eq-after-applying-conclusion-of-I1-I2-I3} over $\left(t_1,t\right)$, $t_1\leq t<T$,
\begin{equation}\label{eq-after-integration-w-r-t-time-in-after-apply-I-1-3}
e^t\int_{B_{\delta_1}\bs B_{\delta}}v_{\3,M}(x,t)\psi_{a_1}(x)\,dx\geq e^{t_1}\int_{B_{\delta_1}\bs B_{\delta}}v_{\3,M}(x,t_1)\psi_{a_1}(x)\,dx-(C_{\psi}+2C_2\delta^{n-\beta_1-2})(e^t-e^{t_1}) \quad \forall t_1\le t<T.
\end{equation}
Since $\beta_1<n-\frac{2}{1-m}$, letting first $t_1\to 0$ and then $\delta\to 0$  in \eqref{eq-after-integration-w-r-t-time-in-after-apply-I-1-3}, 
\begin{equation}\label{v-soln-loc-integral-lower-bd1}
\int_{\widehat{B}_{\delta_1}(a_i)}v_{\3,M}(x,t)\psi_{a_i}(x)\,dx\ge e^{-t}\int_{\widehat{B}_{\delta_1}(a_i)}u_{0,\3,M}(x)\psi_{a_i}(x)\,dx-C_{\psi}(1-e^{-t}) \qquad \forall t>0,\,\,i=1,2,\cdots,i_0.
\end{equation} 
By \eqref{v-u-epsilon-m-compare} and \eqref{v-soln-loc-integral-lower-bd1} we get \eqref{soln-loc-integral-lower-bd1} and the lemma follows. 
\end{proof}

\begin{lemm}\label{lem-property-of-u-epsilon-M-with-test-function}
Let $n\ge 3$, $0<m<\frac{n-2}{n}$, $0<\delta_2<\delta_1<\delta_0$, $0\le f\in L^{\infty}(\partial\Omega\times[0,\infty))$ and $0\leq u_0\in L^p_{loc}\left(\widehat{\Omega}\right)$ for some constant $p>\frac{n(1-m)}{2}$ satisfy \eqref{u0-lower-blow-up-rate} for some constants  $\lambda_1$, $\cdots$, $\lambda_{i_0}\in\R^+$, $\gamma_1$, $\cdots$, $\gamma_{i_0}\in\left(\frac{2}{1-m},\infty\right)$.  Let $\psi$ be given by \eqref{psi-defn} with $\eta$ given by \eqref{eta-defn} for some constants $b_1>\frac{2}{1-m}$ and $\beta_1\in\left[0,n-\frac{2}{1-m}\right)$. Let $\psi_{a_i}(x)=\psi(x-a_i)$ with $\beta_1=\left(n-\gamma_i\right)_+$ for all $i=1,2,\cdots,i_0$. Then for any $T>0$ and $C_2>0$ there exists a constant $M_0=M_0(T,C_2)>0$ such that for any $M\ge M_0$ the solution $u_{\3,M}$ of \eqref{eq-aligned-problem-related-with-epsilon-main-one} satisfies
\begin{equation}\label{eq-strictly-positivity-of-solution-u-ep-M}
\int_{\widehat{B}_{\delta_1}\left(a_i\right)}u_{\3,M}(x,t)\psi_{a_i}(x)\,dx>C_2 \qquad \forall 0<t<T, 0<\3<1,i=1,2,\cdots,i_0.
\end{equation}
\end{lemm}
\begin{proof} 
Without loss of generality it suffices to prove \eqref{eq-strictly-positivity-of-solution-u-ep-M} for $i=i_0=1$, $a_1=\left(0,\cdots,0\right)$ and $\beta_1=\left(n-\gamma_1\right)_+$. Note that $0\le\beta_1<n-\frac{2}{1-m}$ and $\gamma_1+\beta_1\geq n$. Let $C_1>0$ be given by Lemma \ref{lem-similar-L-1-contraction-for-positivity-of-sol-u-epsilon-M},
\begin{equation*}
\delta=\left\{\begin{aligned}
&\left(\omega_n^{-1}\lambda_1^{-1}(C_1+C_2)(\gamma_1+\beta_1)e^T+\delta_2^{n-\gamma_1-\beta_1}\right)^{-\frac{1}{\gamma_1+\beta_1-n}}\quad\mbox{ if }\gamma_1+\beta_1>n\\
&\delta_2\,\mbox{exp}\,\left(-\omega_n^{-1}\lambda_1^{-1}(C_1+C_2)e^T\right)\qquad\qquad\qquad\qquad\mbox{ if }\gamma_1+\beta_1=n\\
\end{aligned}\right.
\end{equation*}
and $M_0=\lambda_1\delta^{-\gamma_1}$. Then $\delta\in (0,\delta_2)$. Let $M\ge M_0$. Then by \eqref{u0-lower-blow-up-rate} and \eqref{psi-defn},
\begin{align}\label{eq-aligned-lower-bound-of-integral-u-0-epsilon-M-denoted-by-I-1}
\int_{\widehat{B}_{\delta_1}}u_{0,\3,M}(x)\psi_{a_i}(x)\,dx\ge&\lambda_1\int_{B_{\delta_2}\bs B_{\delta}}|x|^{-\gamma_1-\beta_1}\,dx
=\omega_n\lambda_1\int_{\delta}^{\delta_2}r^{n-\gamma_1-\beta_1-1}\,dr\notag\\
=&\left\{\begin{aligned}
&\frac{\omega_n\lambda_1}{\gamma_1+\beta_1-n}\left(\delta^{n-\gamma_1-\beta_1}-\delta_2^{n-\gamma_1-\beta_1}\right)\quad\mbox{ if }\gamma_1+\beta_1>n\\
&\omega_n\lambda_1\left(\log\delta_2-\log\delta\right)\qquad\qquad\quad\,\,\mbox{ if }\gamma_1+\beta_1=n\\
\end{aligned}\right.\notag\\
\ge&(C_1+C_2)e^T.
\end{align}
Hence by \eqref{soln-loc-integral-lower-bd1} and \eqref{eq-aligned-lower-bound-of-integral-u-0-epsilon-M-denoted-by-I-1},  \eqref{eq-strictly-positivity-of-solution-u-ep-M} holds and the lemma follows.
\end{proof}

\begin{lemm}\label{lem-property-of-u-M-with-test-function}
Let $n\ge 3$, $0<m<\frac{n-2}{n}$, $0<\delta_2<\delta_1<\delta_0$, $0\le f\in L^{\infty}(\partial\Omega\times[0,\infty))$ and $0\leq u_0\in L^p_{loc}\left(\widehat{\Omega}\right)$ for some constant $p>\frac{n(1-m)}{2}$ satisfy \eqref{u0-lower-blow-up-rate} for some constants  $\lambda_1$, $\cdots$, $\lambda_{i_0}\in\R^+$, $\gamma_1$, $\cdots$, $\gamma_{i_0}\in\left(\frac{2}{1-m},\infty\right)$.  Let $\psi$ be given by \eqref{psi-defn} with $\eta$ given by \eqref{eta-defn} for some constants $b_1>\frac{2}{1-m}$ and $\beta_1\in\left[0,n-\frac{2}{1-m}\right)$. Let $\psi_{a_i}(x)=\psi(x-a_i)$ with $\beta_1=\left(n-\gamma_i\right)_+$ for all $i=1,2,\cdots,i_0$.  Let $T>0$, $\mathcal{C}_2>0$ and let $M_0=M_0(T,\mathcal{C}_2)$ be given by Lemma \ref{lem-property-of-u-epsilon-M-with-test-function}. Then for any $M\ge M_0$, as $\3\to 0$, the solution $u_{\3,M}$ of \eqref{eq-aligned-problem-related-with-epsilon-main-one} decreases and converges to a solution $u_M$ of 
\begin{equation}\label{eq-aligned-problem-related-with-epsilon-main-one-as-epsilon-to-zero}
\begin{cases}
\begin{aligned}
u_t&=\La u^m \quad\mbox{ in }\Omega\times (0,T)\\
u(x,t)&=f \quad\quad \mbox{ on $\partial\Omega\times(0,T)$}\\
u(x,0)&=u_{0,M} \quad \mbox{in $\Omega$}
\end{aligned}
\end{cases}
\end{equation}
uniformly in $C^{2,1}(K)$ for any every compact subset $K$ of $\Omega\times(0,T)$  which satisfies
\begin{equation}\label{eq-comparison-w-r-t-M}
u_{M_1}(x,t)\leq u_{M_2}(x,t)\le\max(M_2,\|f\|_{L^{\infty}}) \quad \mbox{ in }\Omega\times(0,T)\quad\forall M_0\le M_1<M_2,
\end{equation}
and
 \begin{equation}\label{eq-strictly-positivity-of-solution-u-ep-M-1-over-2}
\int_{\widehat{B}_{\delta_1}\left(a_i\right)}u_M(x,t)\psi_{a_i}(x)\,dx\ge C_2 \qquad \forall 0<t<T,\,\,i=1,2,\cdots,i_0.
\end{equation}
Moreover if there exists a constant $T_0\in (0,T)$ such that $f(x,t)$ is monotone decreasing in $t$ on $\partial\Omega\times\left(T_0,\infty\right)$, then $u_M$ satisfies \eqref{Aronson-Benilan-ineqn2} in $\Omega\times (T_0,T)$ for any $M\ge M_0$.
\end{lemm}

\begin{proof}
Since 
\begin{equation*}
\3_1\leq u_{0,\3_{1},M_1}\leq u_{0,\3_2,M_2}\leq \max(M_2,\|f\|_{L^{\infty}})+\3_2 \quad \mbox{ in }\Omega\quad\forall 0<\3_1\leq\3_2<1,M_2>M_1>0,
\end{equation*}
by Lemma \ref{comparison-lem-bd-domain} and Lemma \ref{approx-soln-lem1},
\begin{equation}\label{eq-increasing-of-u-epsilon-M-ubove-by-u-M}
\begin{aligned}
&\3_1\leq u_{\3_1, M_1}\leq u_{\3_2, M_2}\leq \max\left(\left\|f\right\|_{L^{\infty}}, M_2\right)+\3_2\quad \mbox{ in }\Omega\times [0,T)\quad\forall M_0\le M_1<M_2,0<\3_1\leq\3_2<1.
\end{aligned}
\end{equation}
By \eqref{eq-increasing-of-u-epsilon-M-ubove-by-u-M} and the result of \cite{S} the sequence $\left\{u_{\3,M}\right\}_{0<\3<1}$ is equi-H\"older continuous on every compact subset of $\Omega\times  (0,T)$. Hence by \eqref{eq-increasing-of-u-epsilon-M-ubove-by-u-M}, the Ascoli theorem and the Dini Theorem for any $M\ge M_0$ the sequence $\left\{u_{\3,M}\right\}_{0<\3<1}$ decreases and converges uniformly to some continuous function $u_M$ on every compact subset of $\Omega\times  (0,T)$ as $\3\to 0$. 

Let $M\ge M_0$.
Letting $\3_2=\3_1\to 0$ in \eqref{eq-increasing-of-u-epsilon-M-ubove-by-u-M}, we get \eqref{eq-comparison-w-r-t-M}. Putting $u=u_{\3,M}$, $f=f_\3$ in \eqref{eq-alinged-formula-for-weak-solution-of-u} and letting $\3\to 0$, $u_{M}$ satisfies \eqref{eq-alinged-formula-for-weak-solution-of-u}. By Lemma \ref{lem-property-of-u-epsilon-M-with-test-function}, \eqref{eq-strictly-positivity-of-solution-u-ep-M} holds. Letting $\3\to 0$ in \eqref{eq-strictly-positivity-of-solution-u-ep-M}, we get \eqref{eq-strictly-positivity-of-solution-u-ep-M-1-over-2}.
 Since $u_M$ is a continuous distribution solution of \eqref{fde} in $\Omega\times(0,T)$, by \eqref{eq-strictly-positivity-of-solution-u-ep-M-1-over-2} and  Lemma 3.3 of \cite{HK2}, 
\begin{equation}\label{uM-+}
u_M(x,t)>0 \qquad \forall (x,t)\in\Omega\times(0,T), M\ge M_0.
\end{equation} 
Since $u_M$ is continuous, by \eqref{uM-+} for any compact subset $K$ of $\Omega\times(0,T)$, there exists a constant $c_K>0$ such that
\begin{equation}\label{eq-uniform-bound-of-u-M-on-compact-subset}
c_K\leq u_{M}(x,t)\leq u_{\3,M}(x,t)\leq \max\left(\|f\|_{L^{\infty}},M\right)+1 \qquad \forall (x,t)\in K,\,\,0<\3<1.
\end{equation} 
By \eqref{eq-uniform-bound-of-u-M-on-compact-subset} for any $M\ge M_0$ the equation for $\left\{u_{\3,M}\right\}_{0<\3<1}$ is uniformly parabolic on every compact subset of $\Omega\times(0,T)$. Hence by the Schauder estimates \cite{LSU} the sequence $\left\{u_{\3,M}\right\}_{0<\3<1}$  is equi-H\"older continuous in $C^{2,1}(K)$ for any compact set $K\subset\Omega\times(0,T)$. Hence by \eqref{eq-increasing-of-u-epsilon-M-ubove-by-u-M} the sequence $\left\{u_{\3,M}\right\}_{0<\3<1}$ decreases and converges  uniformly in $C^{2,1}(K)$ for any  every compact subset $K$ of $\Omega\times(0,T)$ to $u_M$ as $\3\to 0$. Hence $u_M$ satisfies \eqref{fde} in $\Omega\times(0,T)$ . By \eqref{eq-comparison-w-r-t-M} and  an argument similar in the proof of Theorem 1.1 of \cite{HK2}, $u_M$ has initial value $u_{0,M}$. Therefore $u_M$ is a solution of \eqref{eq-aligned-problem-related-with-epsilon-main-one-as-epsilon-to-zero} which satisfies \eqref{eq-strictly-positivity-of-solution-u-ep-M-1-over-2}. 

If there exists a constant $T_0\in (0,T)$ such that $f(x,t)$ is monotone decreasing in $t$ on $\partial\Omega\times\left(T_0,\infty\right)$, then $u_{\3,M}$ satisfies \eqref{Aronson-Benilan-ineqn2} in $\Omega\times (T_0,T)$ for any $M\ge M_0$ and $0<\3<1$. Putting $u=u_{\3,M}$ in \eqref{Aronson-Benilan-ineqn2} and letting $\3\to 0$ we get that
$u_M$ satisfies \eqref{Aronson-Benilan-ineqn2} in $\Omega\times (T_0,T)$ for any $M\ge M_0$ and the lemma follows.
\end{proof}

\begin{remar}\label{remark-about-maximal-time-T-M-to-infty-as-M-to-infty}
For any $M>0$, let $T_M>0$ be the maximal existence time of the solution $u_M$ of \eqref{eq-aligned-problem-related-with-epsilon-main-one-as-epsilon-to-zero}. Then by  Lemma \ref{lem-property-of-u-M-with-test-function}, $T_{M}\to\infty$ as $M\to\infty$.
\end{remar}

Note that by  the discussion on P.445 of \cite{Hs}, Theorem 1.6, Lemma 1.7 and Lemma 1.9 of \cite{Hu3} remains valid for $n\ge 3$, $0<m\le\frac{n-2}{n}$, $p>\frac{n(1-m)}{2}$, and
$0\le g\in L^{\infty}(\1\Omega\times (0,\infty))$. Hence by Theorem 1.6, Lemma 1.7 and Lemma 1.9 of \cite{Hu3} and an argument similar to that of \cite{Hs} and \cite{Hu3} we have the following two results.

\begin{lemm}[cf. Theorem 2.3 of \cite{Hs}]\label{lem-Cor-2-2-of-Hs1}
Let $n\ge 3$, $0<m\leq\frac{n-2}{n}$, $0\le f\in L^{\infty}(\partial\Omega\times[0,\infty))$  and $0\le u_0\in L^p_{loc}(\2{\Omega}\setminus\{a_1,\cdots,a_{i_0}\})$ for some constant $p>\frac{n(1-m)}{2}$. Suppose $u$ is a solution of  \eqref{fde-Dirichlet-blow-up-problem}.
Then for any $0<\delta_6<\delta_5<\delta_0$ and $0<t_1<T$ there exist constants $C>0$ and $\theta>0$ such that
\begin{equation*}
\left\|u\right\|_{L^{\infty}(\Omega_{\delta_5}\times[t_1,T))}\leq C\left(k_f^p|\Omega|+\int_{\Omega_{\delta_6}}u_0^p\,dx\right)^{\theta/p}+k_f
\end{equation*}
where $k_f=\max (1,\|f\|_{L^{\infty}})$.
\end{lemm}

\begin{lemm}[cf. Corollary 2.2 of \cite{Hs}]\label{lem-Cor-2-2-of-Hs1a}
Let $n\geq 3$, $0<m\leq\frac{n-2}{n}$ and $0\leq u_0\in L^p_{loc}(\widehat{\Omega})$ for some constant $p>\frac{n(1-m)}{2}$. Suppose $u$ is a solution of 
\begin{equation}\label{fde-cauchy-problem5}
\begin{cases}
\begin{aligned}
u_t&=\La u^m \qquad \mbox{in $\Omega\times(0,T)$}\\
u(x,0)&=u_0(x) \qquad \mbox{in $\Omega$}.
\end{aligned}
\end{cases}
\end{equation}  
Then for any $B_{R_1}(x_0)\subset \overline{B_{R_2}(x_0)}\subset\Omega$ and $0<t_1<T$ there exist constants $C>0$ and $\theta>0$ such that
\begin{equation*}
\left\|u\right\|_{L^{\infty}\left(B_{R_1}(x_0)\times\left[t_1,T\right)\right)}\leq C\left(1+\int_{B_{R_2}(x_0)}u_0^p\,dx\right)^{\theta}.
\end{equation*}
\end{lemm}

\begin{lemm}\label{lem-um-converge-to-u}
Let $n\geq 3$, $0<m<\frac{n-2}{n}$, $0<\delta_2<\delta_1<\delta_0$, $0\le f\in L^{\infty}(\partial\Omega\times[0,\infty))$ and $0\leq u_0\in L^p_{loc}(\2{\Omega}\setminus\{a_1,\cdots,a_{i_0}\})$ for some constant $p>\frac{n(1-m)}{2}$ satisfy \eqref{u0-lower-blow-up-rate} for some constants  $\lambda_1$, $\cdots$, $\lambda_{i_0}\in\R^+$, $\gamma_1$, $\cdots$, $\gamma_{i_0}\in\left(\frac{2}{1-m},\infty\right)$.  Let $\psi$ be given by \eqref{psi-defn} with $\eta$ given by \eqref{eta-defn} for some constants $b_1>\frac{2}{1-m}$ and $\beta_1\in\left[0,n-\frac{2}{1-m}\right)$. Let $\psi_{a_i}(x)=\psi(x-a_i)$ with $\beta_1=\left(n-\gamma_i\right)_+$ for all $i=1,2,\cdots,i_0$. Then the solution $u_M$ of \eqref{eq-aligned-problem-related-with-epsilon-main-one-as-epsilon-to-zero} in $\Omega\times(0,T_M)$ increases and converges uniformly in $C^{2,1}(K)$ for every compact subset $K$ of $\widehat{\Omega}\times(0,\infty)$ to a solution $u$ of \eqref{fde-Dirichlet-blow-up-problem} in $\widehat{\Omega}\times\left(0,\infty\right)$ as $M\to\infty$. 

Moreover $u$ satisfies
 \begin{equation}\label{eq-strictly-positivity-of-solution-u-ep-M-infty}
\int_{\widehat{B}_{\delta_1}\left(a_i\right)}u(x,t)\psi_{a_i}(x)\,dx=\infty \qquad \forall t>0,\,\,i=1,2,\cdots,i_0.
\end{equation}
If there exists a constant $T_0>0$ such that $f(x,t)$ is monotone decreasing in $t$ on $\partial\Omega\times\left(T_0,\infty\right)$, then $u$ satisfies \eqref{Aronson-Bernilan-ineqn}.
\end{lemm}
\begin{proof}
Let $0<\delta<\delta_2$ and $T>t_1>0$.
By Remark \ref{remark-about-maximal-time-T-M-to-infty-as-M-to-infty} there exists $M_0'=M_0'(T)>0$ such that $T_M>T$ for all $M\geq M_0'$.
By Lemma \ref{lem-Cor-2-2-of-Hs1} there exist constants $C>0$ and $\theta>0$ such that
\begin{equation}\label{eq-L-infty-bounded-of-u-M-0-strictly-inside-of-domain}
\left\|u_M\right\|_{L^{\infty}(\Omega_{\delta}\times(t_1,T])}\leq C\left\{\left(1+\int_{\Omega_{\frac{\delta}{2}}}u_0^p\,dx\right)^{\theta/p}+1\right\}\quad\forall M\geq M_0'.
\end{equation}
By \eqref{eq-comparison-w-r-t-M} and \eqref{eq-L-infty-bounded-of-u-M-0-strictly-inside-of-domain} the equation \eqref{fde} for $\left\{u_M\right\}_{M>M_0'}$ is uniformly parabolic on $\Omega_{\delta}\times(t_1,T]$ for any $0<\delta<\delta_2$ and $0<t_1<T$. Then by the Schauder estimates \cite{LSU} the sequence $\{u_M\}$ is equi-H\"older continuous in $C^{2,1}(K)$ for any compact subset $K$ of $\widehat{\Omega}\times (0,\infty)$. Hence by the Ascoli theorem, \eqref{eq-comparison-w-r-t-M},  \eqref{eq-L-infty-bounded-of-u-M-0-strictly-inside-of-domain} and a diagonalization argument the sequence $\left\{u_M\right\}$ increases and converges uniformly in $C^{2,1}(K)$ for any compact subset $K$ of $\widehat{\Omega}\times(0,\infty)$ to a solution $u$ of \eqref{fde} in $\widehat{\Omega}\times(0,\infty)$ as $M\to\infty$. Putting $u=u_M$ in \eqref{eq-formula-for-weak-sols-of-main-problem} and letting $M\to\infty$, we get that $u$ satisfies \eqref{eq-formula-for-weak-sols-of-main-problem} for any $t_2>t_1>0$ and $\eta\in C_c^2((\overline{\Omega}\bs\left\{a_1,\cdots,a_{i_0}\right\})\times(0,\infty))$ satisfying $\eta\equiv 0$ on $\partial\Omega\times(0,\infty)$.

By an argument similar in the proof of Theorem 1.1 of \cite{HK2} and Theorem 1.1 of \cite{Hs}, $u$ has initial value $u_0$. By Lemma \ref{lem-property-of-u-M-with-test-function} for any constants $T>0$ and $C_2>0$ there exists a constant $M_0>0$ such that \eqref{eq-strictly-positivity-of-solution-u-ep-M-1-over-2} holds for all $M\ge M_0$. Hence letting $M\to\infty$ in \eqref{eq-strictly-positivity-of-solution-u-ep-M-1-over-2},
\begin{equation}\label{eq-side-condition-of-solution-u-weighted-by-psi-i}
\int_{\widehat{B}_{\delta_1}\left(a_i\right)}u(x,t)\psi_{a_i}(x)\,dx>\mathcal{C}_2 \qquad \forall 0<t<T,C_2>0,i=1,2,\cdots,i_0.
\end{equation}
Letting first $C_2\to\infty$ and then $T\to\infty$ in \eqref{eq-side-condition-of-solution-u-weighted-by-psi-i}, we get \eqref{eq-strictly-positivity-of-solution-u-ep-M-infty}.  

If there exists a constant $T_0>0$ such that $f(x,t)$ is monotone decreasing in $t$ on $\partial\Omega\times\left(T_0,\infty\right)$, then $u_M$ satisfies \eqref{Aronson-Benilan-ineqn2} in $\Omega\times (T_0,T_M)$. Putting $u=u_M$ in \eqref{Aronson-Benilan-ineqn2} and letting $M\to\infty$ we get that
$u$ satisfies \eqref{Aronson-Bernilan-ineqn} and the lemma follows.
\end{proof}

\begin{lemm}\label{u-upper-bd-in-tangential-region-lem}
Let $n\ge 1$, $0<m<1$ and $\alpha>0$. Let $\Omega_1\subset\Omega\subset\R^n$ be smooth bounded domains with $\2{\Omega}_1\subset \Omega$ and $0\le u_0\in L_{loc}^{\infty}(\Omega)$. Let $h_1=\mbox{dist}\, (\2{\Omega}_1,\1\Omega)/2$ and $0<h<\min\left(T,\frac{h_1}{4},\frac{h_1^2}{64\alpha^2}\right)$. Suppose $u$ is a solution of \eqref{fde-cauchy-problem5}. Then there exists a constant $M>0$ such that 
\begin{equation}\label{u-bdness-in-tangential-region}
u(x,t)\le M\quad\forall (x,t)\in\bigcup_{x_0\in\2{\Omega}_1}\Gamma_{\alpha}^h(x_0).
\end{equation}
\end{lemm}
\begin{proof}
We will use a modification of the proof of Lemma 2.1 of \cite{DFK} to prove this lemma.
 Since $\2{\Omega}_1\subset\bigcup_{x_0\in\2{\Omega}_1}B_h(x_0)$ and $\2{\Omega}_1$ is compact, there exist $x_1,\cdots, x_{k_0}\in\2{\Omega}_1$ such that $\2{\Omega}_1\subset\bigcup_{i=1}^{k_0}B_h(x_i)$. For any $i\in\{1,\cdot,k_0\}$, $y_0\in B_h(x_i)$ and $0<t_0<h$, consider the function
\begin{equation}\label{u-rescaled-v-defn}
v(x,t)=\alpha^{\frac{2}{1-m}}u(y_0+\alpha\sqrt{t_0}x,t_0t)\quad\mbox{ in  }B_2\times [0,1].
\end{equation}
Then $v$ is a solution of \eqref{fde} in $B_2\times [0,1]$ and for any $p>\frac{n(1-m)}{2}$,
\begin{equation}\label{v-initial-lp-bd}
\|v(x,0)\|_{L^p(B_2)}\le C_1\alpha^{\frac{2}{1-m}}\|u_0\|_{L^{\infty}(\Omega_2)}
\end{equation}
where $C_1=|B_2|^{1/p}$ and $\Omega_2=\{x\in\Omega:\mbox{dist}\,(x,\1\Omega)>h_1/2\}$.
By \eqref{u-rescaled-v-defn}, \eqref{v-initial-lp-bd} and Lemma \ref{lem-Cor-2-2-of-Hs1a}, there exists a constant $C_{\alpha}$ depending on $\alpha$ such that
\begin{equation*}
\alpha^{\frac{2}{1-m}}\sup_{|z-y_0|<\alpha\sqrt{t_0}}u(z,t_0)=\sup_{|x|\le 1}v(x,1)\le C_{\alpha}\quad\forall 0<t_0<h
\end{equation*}
and \eqref{u-bdness-in-tangential-region} follows.
\end{proof}

By an argument similar to the proof of Lemma 2.5 of \cite{DFK} but with equation (1.28) of \cite{BV} replacing Theorem 1.1 in the proof there we get the following result.

\begin{lemm}\label{u-lower-bd-in-tangential-region-lem}
Let $n\ge 1$, $0<m<1$ and $\Omega_1\subset\Omega\subset\R^n$ be smooth bounded domains with $\2{\Omega}_1\subset \Omega$ and $0\le u_0\in L_{loc}^{\infty}(\Omega)$. Let $h_1=\mbox{dist}\, (\2{\Omega}_1,\1\Omega)/2$. Suppose $u$ is a solution of \eqref{fde-cauchy-problem5} and $E=\{x\in \2{\Omega}_1:u_0(x)>0\}$. Then there exist constants $0<h<h_1$ and $\alpha>0$ such that for almost every $x_0\in E$,
\begin{equation}
\underset{\substack{(x,t)\to x_0\\(x,t)\in\Gamma_{\alpha}^h(x_0)}}{\lim}u(x,t)>0
\end{equation}
\end{lemm}

By an argument similar to the proof of Theorem 2.7 of \cite{DFK} (cf. Theorem 1.8 of \cite{Hu1} ) but with Lemma \ref{u-upper-bd-in-tangential-region-lem} and Lemma \ref{u-lower-bd-in-tangential-region-lem} replacing Lemma 2.1 and Lemma 2.5 in the proof there we have the following result.

\begin{lemm}\label{lem-from-Hu2-and-DFK-convergence-as-t-to-zero}
Let $n\ge 1$, $0<m<1$, $\alpha>0$ and $\Omega\subset\R^n$ be a smooth bounded domain and $0\le u_0\in L_{loc}^{\infty}(\Omega)$. Suppose $u$ is a solution of \eqref{fde-cauchy-problem5}. Then
\begin{equation*}
\lim_{\begin{subarray}{c}|x-x_0|\leq\alpha\sqrt{t}\\t\to 0\end{subarray}}u(x,t)=u_0(x_0) 
\end{equation*}
for any point $x_0\in\Omega$ of continuity of $u_0$.
\end{lemm}

\begin{lemm}\label{lem-local-lower-bound-of-u-near-singular-points}
Let $n\geq 3$, $0<m<\frac{n-2}{n}$, $0<\delta_1<\delta_0$, $0\le f\in L^{\infty}(\partial\Omega\times[0,\infty))$ and $0\leq u_0\in L^p_{loc}(\2{\Omega}\setminus\{a_1,\cdots,a_{i_0}\})$ for some constant $p>\frac{n(1-m)}{2}$ satisfy \eqref{u0-lower-blow-up-rate} for some constants  $\lambda_1$, $\cdots$, $\lambda_{i_0}\in\R^+$, $\gamma_1$, $\cdots$, $\gamma_{i_0}\in\left(\frac{2}{1-m},\infty\right)$.  Let $u$ be the solution of \eqref{fde-Dirichlet-blow-up-problem} in $\widehat{\Omega}\times\left(0,\infty\right)$  given by Lemma \ref{lem-um-converge-to-u}. Then for any $T>0$ and $\delta_2\in (0,\delta_1)$ there exists a constant $0<C_1\leq\min_{1\leq i\leq i_0}\lambda_i$ such that \eqref{eq-lower-limit-of-u-near-blow-up-points} holds.
\end{lemm}
\begin{proof}
Without loss of generality it suffices to prove the lemma for the case $i=i_0=1$ and $a_1=\left(0,\cdots,0\right)$. Let $f_{\3}$ be given by \eqref{eq-definition-of-f-epsilon-by-f},
\begin{equation}\label{eq-def-of-v-0-near-blow-up-point}
v_0(x)=\begin{cases} 
\lambda_1|x|^{-\gamma_1}\quad  \forall 0<|x|<\delta_1\\
 0 \qquad \quad\,\,\,\forall x\in B_{\delta_0}\bs B_{\delta_1},
 \end{cases}
\end{equation}
and $v_{0, M}, v_{0,\3,M}$ be given by \eqref{eq-def-of-u-0-epsilon-M} with $u_0$
being replaced by $v_0$.
By Lemma \ref{approx-soln-lem1} for any $0<\3<1$ and $M>0$ there exists a solution $v_{\3,M}$ of 
\begin{equation}\label{Dirichlet-problem-in-ball}
\begin{cases}
\begin{aligned}
u_t&=\La u^m \qquad\,  \mbox{in $B_{\delta_0}\times(0,\infty)$}\\
u(x,t)&=\3 \qquad\quad \mbox{ on $\partial B_{\delta_0}\times(0,\infty)$}\\
u(x,0)&=v_{0,\3,M}  \quad \,\mbox{ in $B_{\delta_0}$}.
\end{aligned}
\end{cases}
\end{equation} which satisfies
\begin{equation*}
\3\leq v_{\3, M}\le M+\3 \quad \mbox{ in }B_{\delta_0}\times(0,\infty).
\end{equation*}
Then by Lemma \ref{comparison-lem-bd-domain} for any $t>0$ $v_{\3, M}(x,t)$ is  radially symmetric in $x\in B_{\delta_0}$.
Let $u_{\3,M}$ be the solution of \eqref{eq-aligned-problem-related-with-epsilon-main-one} which satisfies \eqref{eq-upper-and-lower-bound-of-u-epsilon-M-09}. Since $u_{0,\3,M}\geq v_{0,\3,M}\ge\3$ in $B_{\delta_0}$ and $u_{\3,M}$ is a supersolution of \eqref{Dirichlet-problem-in-ball}, by Lemma \ref{comparison-lem-bd-domain},
\begin{equation}\label{eq-comparison-between-u-epsilon-M-and-v-epsilon-M}
u_{\3,M}\geq v_{\3,M} \quad \mbox{ in }B_{\delta_0}\times(0,\infty)\quad\forall M>0, 0<\3<1.
\end{equation}
Let $T>0$. By Lemma \ref{lem-property-of-u-M-with-test-function} and Remark \ref{remark-about-maximal-time-T-M-to-infty-as-M-to-infty} for any $M>0$ there exists a maximal existence time $T_M'>0$ such that a solution $v_M$ of 
\begin{equation}\label{Dirichlet-problem-in-ball-bdary=0}
\begin{cases}
\begin{aligned}
u_t&=\La u^m \quad\mbox{ in }B_{\delta_0}\times (0,T_M')\\
u(x,t)&=0 \quad\quad \mbox{ on $\partial B_{\delta_0}\times(0,T_M')$}\\
u(x,0)&=v_{0,M} \quad \mbox{in $B_{\delta_0}$}
\end{aligned}
\end{cases}
\end{equation} 
exists and $v_{\3,M}$  decreases and converges to $v_M$ in $C^{2,1}(K)$ for any compact subset $K$ of $B_{\delta_0}\times (0,T_M')$ as $\3\to 0$. Moreover $T_M'\to\infty$ as $M\to\infty$. Hence for any $0<t<T_M'$ $v_M(x,t)$ is radially symmetric in $x\in B_{\delta_0}$. Moreover there exists $M_1=M_1(T)>0$ such that
\begin{equation*}
T_M'>T\qquad \forall M\geq M_1.
\end{equation*}
Letting $\3\to 0$ in \eqref{eq-comparison-between-u-epsilon-M-and-v-epsilon-M}, 
\begin{equation}\label{eq-comparison-between-u-M-and-v-M}
u_M\geq v_M \quad \mbox { in }B_{\delta_0}\times(0,T)\quad\forall M\geq M_1.
\end{equation}
Let 
\begin{equation*}
\delta_2\in\left(0,\delta_1\right) \qquad \mbox{and} \qquad  M_2=\max\left(\lambda_1\delta_2^{-\gamma_1}, M_1\right).
\end{equation*}
Then
\begin{equation}\label{eq-lower-bound-of-v-0-M-2-in-the-ball-of-radius-delta-2}
v_{0,M_2}\left(x\right)\geq \lambda_1\delta_2^{-\gamma_1}\qquad \forall |x|\leq\delta_2.
\end{equation}
Let $|x_1|=\delta_2$. By Lemma \ref{lem-from-Hu2-and-DFK-convergence-as-t-to-zero}  there exists a constant $t_1>0$ such that
\begin{equation}\label{eq-difference-between-v-M-2-and-v-0-M-2-near-x-y-very-closed}
|v_{M_2}(x_1,t)-v_{0,M_2}(x_1)|\le\frac{1}{2}\lambda_1\delta_2^{-\gamma_1} \quad \forall 0<t\le t_1.
\end{equation}
Since $v_M(x,t)$ is radially symmetric in $x\in B_{\delta_0}$, by \eqref{eq-lower-bound-of-v-0-M-2-in-the-ball-of-radius-delta-2} and \eqref{eq-difference-between-v-M-2-and-v-0-M-2-near-x-y-very-closed},
\begin{equation}\label{eq-lower-bound-of-v-M-2-y-near-ball-B-delta-2}
v_{M_2}(x,t)=v_{M_2}(x_1,t)\geq \frac{1}{2}\lambda_1\delta_2^{-\gamma_1} \quad \forall |x|=|x_1|=\delta_1, 0<t\le t_1.
\end{equation}
Let 
$$
c_1=\min_{\substack{|x|=\delta_2\\t_1\le t\le T}}v_{M_2}(x,t).
$$ 
Then $c_1>0$ and by \eqref{eq-lower-bound-of-v-M-2-y-near-ball-B-delta-2},
\begin{equation}\label{vm-lower-bd}
v_{M_2}(x,t)\geq c_2\quad  \forall |x|=\delta_2, 0\le t\le T.
\end{equation}
where $c_2=\min\left(c_1,\frac{1}{2}\lambda_1\delta_2^{-\gamma_1}\right)>0$. For any $0<\eta<1$ and $A>0$, let $U_{A\eta}$ be the solution of  
\begin{equation}\label{eq-problem-for-upper-lower-bound-U-A-eta}
\begin{cases}
\begin{aligned}
u_t&=\La u^m \qquad\qquad\quad\mbox{ in } B_{\delta_2}\times (0,\infty)\\
u(x,t)&=A\left(\delta_2^2+\eta\right)^{-\frac{\gamma_1}{2}} \quad\,\,\mbox{ on } \1 B_{\delta_2}\times (0,\infty)\\
u(x,0)&=A(|x|^2+\eta)^{-\frac{\gamma_1}{2}} \quad \mbox{ in } B_{\delta_2}
\end{aligned}
\end{cases}
\end{equation}
which satisfies 
\begin{equation}\label{U-a-eta-lower-upper-bd}
A(\delta_2^2+\eta)^{-\frac{\gamma_1}{2}}\le U_{A\eta}\le A\eta^{-\gamma_1}\quad\mbox{ in }
B_{\delta_2}\times (0,\infty).
\end{equation}
Then by Lemma \ref{comparison-lem-bd-domain}, Lemma \ref{lem-property-of-u-M-with-test-function}, \eqref{vm-lower-bd} and \eqref{U-a-eta-lower-upper-bd}, for any $0<\eta<1$ and $0<A\leq \min\left(c_2\eta^{\gamma_1},\lambda_1\right)$,
\begin{equation}\label{u-a-eta-vm-bdary-compare1}
U_{A\eta}(x,t)\le c_2\leq v_{M_2}(x,t)\leq v_{\3,M_2}(x,t)\leq v_{\3,M}(x,t) \quad \forall |x|=\delta_2, 0\le t\le T,M\ge M_2, 0<\3<1
\end{equation}
and
\begin{equation}\label{u-a-eta-vm-bdary-compare2}
U_{A\eta}(x,0)\leq\min\left(A|x|^{-\gamma_1}, A\eta^{-\frac{\gamma_1}{2}}\right)\leq v_{0,\3, M_{\eta}}(x) \qquad \forall |x|\leq\delta_2
\end{equation}
where $M_{\eta}=\max\left(M_2, A\eta^{-\frac{\gamma_1}{2}}\right)$. Then by \eqref{U-a-eta-lower-upper-bd}, \eqref{u-a-eta-vm-bdary-compare1}, \eqref{u-a-eta-vm-bdary-compare2}, Lemma \ref{comparison-lem-bd-domain} and Lemma \ref{lem-property-of-u-M-with-test-function},  
\begin{align}
&v_{\3,M_{\eta}}\geq U_{A\eta} \quad \mbox{ in }B_{\delta_2}\times[0,T]\quad\forall 0<\3<1,0<\eta<1\notag\\
\Rightarrow \quad &v_{M_{\eta}}\geq U_{A\eta} \quad \mbox{ in }B_{\delta_2}\times[0,T]\quad\forall 0<\eta<1\quad \mbox{as $\3\to 0$}.\label{eq-comparison-between-v-and-U-A-eta}
\end{align}
We now divide the proof into two cases.

\noindent{\bf Case 1}: $\gamma_1\in\left(\frac{2}{1-m},\frac{n-2}{m}\right]$.

\noindent By Lemma 3.8 of \cite{VW1}  there exist constants $b_1=b_1\left(T,A,\delta_2\right)>0$ and $\eta_1=\eta_1(T)\in(0,1)$ such that 
\begin{equation}\label{eq-comparison-between-U-A-eta-and-direct-function}
U_{A\eta}(x,t)\geq b_1\left(|x|^2+\eta\right)^{-\frac{\gamma_1}{2}} \quad \mbox{ in }B_{\delta_2}\times[0,T]\quad\forall 0<\eta<\eta_1.
\end{equation}
Then by \eqref{eq-comparison-between-u-M-and-v-M}, \eqref{eq-comparison-between-v-and-U-A-eta}, \eqref{eq-comparison-between-U-A-eta-and-direct-function} and Lemma \ref{lem-um-converge-to-u},
\begin{equation}\label{eq-comparison-between-v-and-direct-function-5}
u(x,t)\geq u_{M_{\eta}}(x,t)\geq b_1\left(|x|^2+\eta\right)^{-\frac{\gamma_1}{2}}\quad \mbox{ in }B_{\delta_2}\times[0,T]\quad\forall 0<\eta<\eta_1.
\end{equation}
Letting $\eta\to 0$ in \eqref{eq-comparison-between-v-and-direct-function-5}, \eqref{eq-lower-limit-of-u-near-blow-up-points} follows.

\noindent{\bf Case 2}: $\gamma_1\in\left[\frac{n-2}{m},\infty\right)$.

\noindent By Lemma 3.10 of \cite{VW1}  for any  $\delta>0$ there exists a constant $\eta_2=\eta_2(\delta,T)\in(0,1)$ such that 
\begin{equation}\label{eq-comparison-between-U-A-eta-and-direct-function-2}
U_{A\eta}(x,t)\geq A\left(|x|^2+\delta\right)^{-\frac{\gamma_1}{2}} \quad \mbox{ in }B_{\delta_2}\times[0,T]\quad\forall 0<\eta<\eta_2.
\end{equation}
Then by \eqref{eq-comparison-between-u-M-and-v-M}, \eqref{eq-comparison-between-v-and-U-A-eta}, \eqref{eq-comparison-between-U-A-eta-and-direct-function-2} and Lemma \ref{lem-um-converge-to-u},
\begin{equation}\label{eq-comparison-between-v-and-direct-function-6}
u(x,t)\ge u_{M_{\eta}}(x,t)\geq A\left(|x|^2+\delta\right)^{-\frac{\gamma_1}{2}}\quad \mbox{ in }B_{\delta_2}\times[0,T]\quad\forall 0<\eta<\eta_2,\delta>0.
\end{equation}
Letting $\delta\to 0$ in \eqref{eq-comparison-between-v-and-direct-function-6}, \eqref{eq-lower-limit-of-u-near-blow-up-points} holds and the lemma follows.
\end{proof}

\begin{proof}[\textbf{Proof of Theorem \ref{first-main-existence-thm}}:]
Theorem \ref{first-main-existence-thm} follows immediately by Lemma \ref{lem-um-converge-to-u} and Lemma \ref{lem-local-lower-bound-of-u-near-singular-points}.
\end{proof}

By Lemma \ref{comparison-lem-bd-domain} and the construction of solution of \eqref{Dirichlet-blow-up-problem} in Lemma \ref{lem-property-of-u-M-with-test-function} and Lemma \ref{lem-um-converge-to-u} we have the following comparison result.

\begin{thm}\label{bd-soln-comparison-thm1}
Let $n\geq 3$, $0<m<\frac{n-2}{n}$, $0<\delta_1<\delta_0$, $0\le f_1\le f_2\in L^{\infty}(\partial\Omega\times[0,\infty))$ and  $0\le u_{0,1}\le u_{0,2}\in L_{loc}^p(\2{\Omega}\setminus\{a_1,\cdots,a_{i_0}\})$ for some constant  $p>\frac{n(1-m)}{2}$ be such that $u_{0,2}$ satisfies \eqref{u0-lower-blow-up-rate} for some constants $\lambda_1$, $\cdots$, $\lambda_{i_0}\in\R^+$ and $\gamma_1,\cdots,\gamma_{i_0}\in \left(\frac{2}{1-m},\infty\right)$ and $u_{0,1}$ satisfies \eqref{u0-lower-blow-up-rate} for some constants $\lambda_1$, $\cdots$, $\lambda_{i_1}\in\R^+$ and $\gamma_1,\cdots,\gamma_{i_1}\in \left(\frac{2}{1-m},\infty\right)$ with $1\le i_1\le i_0$. Suppose  $u_1, u_2$, are the solutions of \eqref{Dirichlet-blow-up-problem} with $u_0=u_{0,1}, u_{0,2}$, $f=f_1,f_2$ respectively given by Theorem \ref{first-main-existence-thm}, then $u_1\le u_2$ in $\widehat{\Omega}\times (0,\infty)$.
\end{thm}

\begin{proof}[\textbf{Proof of Theorem \ref{second-main-existence-thm}}:]
Since the proof of the thoerem is similar to the proof of Theorem \ref{first-main-existence-thm}, we will only give a sketch its proof here. Let $0<\delta_2<\delta_1$ and $\psi$ be given by \eqref{psi-defn} with $\eta$ given by \eqref{eta-defn} for some constants $b_1>\frac{2}{1-m}$ and $\beta_1\in\left[0,n-\frac{2}{1-m}\right)$. Let $\psi_{a_i}(x)=\psi(x-a_i)$ with $\beta_1=\left(n-\gamma_i\right)_+$ for all $i=1,2,\cdots,i_0$.
By Theorem 1.1 of \cite{Hs} and Corollary 2.2 of \cite{DS1}  for any $M>1$ and $0<\3<1$ there exists a unique solution $u_{\3,M}$ of
\begin{equation}\label{cauchy-problem-Rn}
\begin{cases}
\begin{aligned}
u_t=&\La u^m \qquad \qquad\mbox{in $\R^n\times(0,\infty)$}\\
u(x,0)&=u_{0,\3,M}(x) \quad \mbox{ in }\R^n
\end{aligned}
\end{cases}
\end{equation}
 which satisfies
\begin{equation}\label{u-epsilon-m-lower-upper-bd11}
\3\le u_{\3, M_1}(x,t)\le u_{\3, M_2}(x,t)\le M_2+\3 \qquad \forall (x,t)\in\R^n\times(0,\infty), M_2>M_1>1
\end{equation}
and \eqref{Aronson-Benilan-ineqn2} with $T_0=0$ in $\R^n\times(0,\infty)$. Moreover
for any $T>0$ and $C_2>0$ there exists a constant $M_0(T,C_2)>0$ such that for all $M\ge M_0$ the solution $u_{\3,M}$ satisfies \eqref{eq-strictly-positivity-of-solution-u-ep-M}. By \eqref{eq-strictly-positivity-of-solution-u-ep-M}, \eqref{u-epsilon-m-lower-upper-bd11}, and an argument similar to the proof of Lemma \ref{lem-property-of-u-M-with-test-function}  $u_{\3,M}$ decreases and converges  uniformly in $C^{2,1}(K)$ on every compact subset $K$ of $\R^n\times(0,T)$ to a solution $u_M$ of 
\begin{equation}\label{cauchy-problem-Rn-2}
\left\{\begin{aligned}
u_t=&\La u^m\qquad\,\mbox{ in }\R^n\times(0,T)\\
u(x,0)=&u_{0,M}(x)\quad\mbox{ in }\R^n
\end{aligned}\right.
\end{equation}
as $\3\to 0$ which satisfies \eqref{eq-strictly-positivity-of-solution-u-ep-M-1-over-2} and
\begin{equation}\label{u-m1-u-m2-compare}
u_{M_1}(x,t)\leq u_{M_2}(x,t)\le M_2 \quad \mbox{ in }\R^n\times(0,T)\quad\forall M_0\le M_1<M_2.
\end{equation}
Let $T_M$ be the maximal existence time of the solution $u_{M}$. Then by Remark \ref{remark-about-maximal-time-T-M-to-infty-as-M-to-infty} $T_M\to\infty$ as $M\to\infty$.
Thus by \eqref{u-m1-u-m2-compare} and an argument similar in the proof of Lemma \ref{lem-um-converge-to-u},  $u_M$ increases and converges uniformly in $C^{2,1}(K)$ on every compact subset $K$ of $\widehat{\R^n}\times(0,\infty)$ to a solution $u$ of
\begin{equation*}
\left\{\begin{aligned}
u_t=&\La u^m\quad\mbox{ in }\widehat{\R^n}\times(0,\infty)\\
u(x,0)=&u_0(x)\quad\mbox{ in }\widehat{\R^n}
\end{aligned}\right.
\end{equation*} 
as $M\to\infty$. Letting $M\to\infty$ first and then $\mathcal{C}_2\to\infty$ and $T\to\infty$ in \eqref{eq-strictly-positivity-of-solution-u-ep-M-1-over-2}, $u$ satisfies \eqref{eq-strictly-positivity-of-solution-u-ep-M-infty}.
By an argument similar to the proof of Lemma \ref{lem-local-lower-bound-of-u-near-singular-points} for any $T>0$ and $\delta_2\in (0,\delta_1)$ there exists a constant $C_1>0$ such that \eqref{eq-lower-limit-of-u-near-blow-up-points} holds. Putting $u=u_{\3,M}$, $T_0=0$, in \eqref{Aronson-Benilan-ineqn2} and letting $\3\to 0$ and $M\to\infty$ we get \eqref{eq-Aronson-Bernilan-on-R-n}. Hence $u$ is a solution of \eqref{cauchy-blow-up-problem} that satisfies  \eqref{eq-lower-limit-of-u-near-blow-up-points} and \eqref{eq-Aronson-Bernilan-on-R-n} and the theorem follows.
\end{proof}

By the proof of Theorem 1.1 of \cite{Hs}, Lemma \ref{comparison-lem-bd-domain}, and the construction of solution of \eqref{cauchy-blow-up-problem} in the proof of  Theorem \ref{second-main-existence-thm} we have the following comparison result.

\begin{thm}\label{R^n-soln-comparison-thm}
Let $n\geq 3$, $0<m<\frac{n-2}{n}$, $0<\delta_1<\delta_0$ and $0\le u_{0,1}\le u_{0,2}\in L_{loc}^p\left(\widehat{\R^n}\right)$ for some constant  $p>\frac{n(1-m)}{2}$ be such that $u_{0,2}$ satisfies \eqref{u0-lower-blow-up-rate} for some constants $\lambda_1$, $\cdots$, $\lambda_{i_0}\in\R^+$ and $\gamma_1,\cdots,\gamma_{i_0}\in \left(\frac{2}{1-m},\infty\right)$ and $u_{0,1}$ satisfies \eqref{u0-lower-blow-up-rate} for some constants $\lambda_1$, $\cdots$, $\lambda_{i_1}\in\R^+$ and $\gamma_1,\cdots,\gamma_{i_1}\in \left(\frac{2}{1-m},\infty\right)$ with $1\le i_1\le i_0$. Suppose  $u_1, u_2$, are the solutions of \eqref{cauchy-blow-up-problem}
with $u_0=u_{0,1}, u_{0,2}$ respectively given by Theorem \ref{second-main-existence-thm}, then $u_1\le u_2$ in $\widehat{\R^n}\times (0,\infty)$.
\end{thm}

By an argument similar to  the proof of Lemma \ref{lem-cf-lemma-2-3-of-cite-HK1} we have the following lemma.

\begin{lemm}\label{upper-bd-blow-up-soln-Rn}
Let $n\geq 1$, $0<m<1$ and $0<\delta_3<\min(1,\delta_0)$. Let $0\le u_0\in L_{loc}^1(\widehat{\R^n})$  satisfy \eqref{upper-blow-up-rate-initial-data}
for some constants $\lambda_1',\cdots,\lambda_{i_1}'\in\R^+$, $\gamma_i', \cdots, \gamma_{i_1}'\in\left[\frac{2}{1-m},\infty\right)$, and integer $1\le i_1\le i_0$. For any $0<\3<1$, $M>0$, let $u_{\3,M}$ be the solution of \eqref{cauchy-problem-Rn}. Then for any $T>0$, there exists a  constant $A_0>0$ such that \eqref{u-upper-bd-phi-tidle} holds where $\phi_{i, A_0}$ is given by \eqref{phi-A-defn}.
\end{lemm}

By Lemma \ref{lem-cf-lemma-2-3-of-cite-HK1}, Lemma \ref{upper-bd-blow-up-soln-Rn}, and the construction of solutions in Theorem \ref{first-main-existence-thm} and Theorem \ref{second-main-existence-thm} we have the following two corollaries.

\begin{cor}\label{u-upper-blow-up-bd-lem}
Let $n\geq 3$, $0<m<\frac{n-2}{n}$ and $0<\delta_3<\delta_1<\min(1,\delta_0)$,  $0\le f\in L^{\infty}(\partial\Omega\times (0,\infty))$ and $0\le u_0\in L_{loc}^p(\widehat{\Omega})$ satisfy \eqref{u0-lower-blow-up-rate} and 
\eqref{upper-blow-up-rate-initial-data}
for some integer $1\le i_1\le i_0$ and constants $\lambda_1$, $\cdots$, $\lambda_{i_0},\lambda_1',\cdots,\lambda_{i_1}'\in\R^+$, $\gamma_1,\cdots,\gamma_{i_0}, \gamma_i', \cdots, \gamma_{i_1}'\in\left(\frac{2}{1-m},\infty\right)$. 
If $u$ is the solution of \eqref{Dirichlet-blow-up-problem} given by Theorem \ref{first-main-existence-thm}, then for any $T>0$ there exists a constant $A_0>0$ such that 
\begin{equation*}
u(x,t)\leq \phi_{i,A_0}(x-a_i,t) \quad \forall 0<|x-a_i|<\delta_3,0\leq t<T, i=1,\cdots,i_1
\end{equation*}
holds where $\phi_{i,A_0}$ is given by \eqref{phi-A-defn}.
\end{cor}

\begin{cor}\label{u-upper-blow-up-bd-R^n-lem}
Let $n\geq 3$, $0<m<\frac{n-2}{n}$ and $0<\delta_3<\delta_1<\min\left(1,\delta_0\right)$. Let  $0\leq u_0\in L_{loc}^p(\widehat{\R^n})$ satisfy \eqref{u0-lower-blow-up-rate} and \eqref{upper-blow-up-rate-initial-data} for some integer $1\le i_1\le i_0$ and constants  $\lambda_1$, $\cdots$, $\lambda_{i_0},\lambda_1',\cdots,\lambda_{i_1}'\in\R^+$, $\gamma_1,\cdots,\gamma_{i_0}, \gamma_i', \cdots, \gamma_{i_1}'\in\left(\frac{2}{1-m},\infty\right)$. 
If $u$ is the solution of \eqref{cauchy-blow-up-problem} given by Theorem \ref{second-main-existence-thm}, then for any $T>0$ there exists a constant $A_0>0$ such that 
\begin{equation*}
u(x,t)\leq \phi_{i,A_0}(x-a_i,t) \quad \forall 0<|x-a_i|<\delta_3,0\leq t<T, i=1,\cdots,i_1
\end{equation*}
holds where $\phi_{i,A_0}$ is given by \eqref{phi-A-defn}.
\end{cor}

\section{Asymptotic behaviour of solutions}\label{Asymptotic-behaviour-of-solution}
\setcounter{equation}{0}
\setcounter{thm}{0}

In this section we will prove the asymptotic large time behaviour of solutions of \eqref{Dirichlet-blow-up-problem} and \eqref{cauchy-blow-up-problem}. We will first prove some technical lemmas.

\begin{lemm}\label{lem-super-soluiton-v-o-eta-sum-and-powers}
Let $n\geq 3$, $0<m<1$, $A_1$, $\cdots$, $A_{i_1}\in\R^+$ and $\gamma_1'$, $\cdots$, $\gamma_{i_0}'\in\left(0,\frac{n-2}{m}\right]$. For any $C_0>0$ and $\eta>0$, let
\begin{equation}\label{eq-def-of-function-v-i-eta-and-v-0=eta}
v_{\eta}(x)=\left[C_0^m+\sum_{i=1}^{i_0}\left(v_{i,\eta}(x)\right)^m\right]^{\frac{1}{m}} \qquad \forall x\in\R^n
\end{equation} 
where
\begin{equation}\label{eq-def-of-v-i-eta-for-super-solution-of-harmonic}
v_{i,\eta}(x)=A_i\left(|x-a_i|^2+\eta\right)^{-\frac{\gamma_i'}{2}} \qquad \forall i=1,\cdots,i_0.
\end{equation}
Then
\begin{equation}\label{eq-super-harmonic-of-v-eta-to-m}
\La v_{\eta}^m\leq 0 \qquad \mbox{in $\R^n$}.
\end{equation}
\end{lemm}
\begin{proof}
By direct computation, $v_{\eta}$ satisfies
\begin{equation*}
\La v_{\eta}^m=-\sum_{i=1}^{i_0}mA_i^m\gamma_i'\left(|x-a_i|^2+\eta\right)^{-\frac{m\gamma_i'+4}{2}}\left[\left(n-2-m\gamma_i'\right)|x-a_i|^2+n\eta\right]\leq 0 \qquad \mbox{in $\R^n$}
\end{equation*}
and the lemma follows.
\end{proof}

\begin{lemm}\label{lem-local-upper-bound-of-u-near-singular-points}
Let $n\geq 3$, $0<m<\frac{n-2}{n}$, $0<\delta_3<\delta_1<\min(1,\delta_0)$, $0\le f\in L^{\infty}(\partial\Omega\times[0,\infty))$ and $0\leq u_0\in L^p_{loc}(\2{\Omega}\setminus\{a_1,\cdots,a_{i_0}\})$ for some constant $p>\frac{n(1-m)}{2}$ such that  \eqref{u0-lower-blow-up-rate} and \eqref{upper-blow-up-rate-initial-data} holds with $i_1=i_0$ for some constants satisfying
 \eqref{gamma-upper-lower-bd3} and $\lambda_1$, $\cdots$, $\lambda_{i_0}$, $\lambda_1'$, $\cdots$, $\lambda_{i_0}'\in\R^+$. Let $u$ be the solution of \eqref{Dirichlet-blow-up-problem} constructed in Theorem \ref{first-main-existence-thm}. Then for any $0<\delta_2<\delta_0$ and $t_0>0$ there exist constants $C_2>0$ and $C_3>0$ such that
\begin{equation}\label{eq-bound-of-solution-on-compact-subset-with-t01-Omega-delta-2}
u(x,t)\le C_2 \quad \forall x\in\overline{\Omega_{\delta_2}}\times[t_0,\infty)
\end{equation} 
and
\begin{equation}\label{eq-local-behaviour-of-solution-u-blow-up=0-above-by-1}
u(x,t)\le C_3|x-a_i|^{-\gamma_i'}\quad\forall 0<|x-a_i|\leq\delta_2,t\ge t_0,i=1,\cdots,i_0
\end{equation}
hold.
\end{lemm}

\begin{proof}
Let $t_0>0$ and $0<\delta_4<\delta_3$. For any $0<\3<1$ and $M>0$, let $u_{\3,M}$, $u_M$, be the solution of \eqref{eq-aligned-problem-related-with-epsilon-main-one} and \eqref{eq-aligned-problem-related-with-epsilon-main-one-as-epsilon-to-zero} respectively given by Lemma \ref{approx-soln-lem1} and Lemma \ref{lem-property-of-u-M-with-test-function}. Then by Lemma \ref{lem-property-of-u-M-with-test-function} and Lemma \ref{lem-um-converge-to-u},
\begin{equation}\label{u-eta-m-limit}
u(x,t)=\lim_{M\to\infty}\lim_{\3\to 0}u_{\3, M}(x,t) \quad \forall x\in\widehat{\Omega},t>0.
\end{equation}
Let $T_M$ be the maximal existence time of $u_M$. By Remark \ref{remark-about-maximal-time-T-M-to-infty-as-M-to-infty} there exists $M_1>0$ such that $T_M>t_0$ for all $M\ge M_1$.
By Lemma \ref{lem-cf-lemma-2-3-of-cite-HK1} there exists a constant $A_0>0$  such that \eqref{u-upper-bd-phi-tidle} holds with $i_1=i_0$ where
$\phi_{i, A_0}$ is given by \eqref{phi-A-defn}. By Lemma \ref{lem-Cor-2-2-of-Hs1}  there exists a constant $C_4>1$ such that
\begin{equation}\label{eq-upper-bound-of-u-epislon=M-at-t-=-t-0-1}
u_{\3,M}(x,t_0)\leq C_4 \quad \mbox{ in }\widehat{\Omega}_{\delta_4}\quad \forall 0<\3<1,M>0.
\end{equation}
Let 
\begin{equation}\label{eq-constant-A-dash-and-M-sub-zoer}
A_0'=\frac{A_0\left(1+t_0\right)^{\frac{1}{1-m}}}{\left(\delta_3-\delta_4\right)^{\frac{2}{1-m}}}  \quad\mbox{ and }\quad M_2=\max(M_1,\|f\|_{L^{\infty}}).
\end{equation} 
By \eqref{u-upper-bd-phi-tidle},
\begin{equation}\label{eq-upper-bound-of-u-epsilon-M-by-M-in-whole-and-r-minus-gamma-dash-on-ball}
u_{\3,M}(x,t_0)\le\frac{A_0'}{|x-a_i|^{\gamma_i'}}\quad \forall 0<|x-a_i|<\delta_4,0<\3<1, M>0,  1\leq i\leq i_0.
\end{equation}
Let $v_{\eta}$ and $v_{i,\eta}$ be given by \eqref{eq-def-of-function-v-i-eta-and-v-0=eta} and \eqref{eq-def-of-v-i-eta-for-super-solution-of-harmonic} with 
\begin{equation}\label{constants-defn1}
0<\eta<\eta_1(M):=\min_{1\leq i\leq i_0}\left(\frac{A_0'}{M+1}\right)^{\frac{2}{\gamma_i'}}, \quad A_i=2^{\frac{\gamma_i'}{2}}A_0'\quad\mbox{ and }\quad C_0=C_4+\|f\|_{L^{\infty}}. 
\end{equation}
Then
\begin{equation}\label{eq-upper-bound-of-v-eta-i-epsilon-M-by-M-in-whole-and-r-minus-gamma-dash-on-ball}
v_{i,\eta}(x)\geq \begin{cases}
\begin{array}{lccl}
A_i\left(2\left(\frac{A_0'}{M+1}\right)^{\frac{2}{\gamma_i'}}\right)^{-\frac{\gamma_i'}{2}}=M+1 &&&\forall 0<|x-a_i|\leq\left(\frac{A_0'}{M+1}\right)^{\frac{1}{\gamma_i'}},1\leq i\leq i_0\\
A_i\left(2|x-a_i|^2\right)^{-\frac{\gamma_i'}{2}}=A_0'|x-a_i|^{-\gamma_i'} &&& \forall \left(\frac{A_0'}{M+1}\right)^{\frac{1}{\gamma_i'}}<|x-a_i|<\delta_4,1\leq i\leq i_0.
\end{array}
\end{cases}
\end{equation}
Hence by \eqref{eq-upper-and-lower-bound-of-u-epsilon-M-09}, \eqref{eq-def-of-function-v-i-eta-and-v-0=eta}, \eqref{eq-upper-bound-of-u-epislon=M-at-t-=-t-0-1}, \eqref{eq-upper-bound-of-u-epsilon-M-by-M-in-whole-and-r-minus-gamma-dash-on-ball}, \eqref{constants-defn1} and \eqref{eq-upper-bound-of-v-eta-i-epsilon-M-by-M-in-whole-and-r-minus-gamma-dash-on-ball},
\begin{equation}\label{u-epsilon-m<v-eta}
u_{\3, M}(x,t)\leq v_{\eta}(x)  \qquad \mbox{on $\Omega\times\{t_0\}\cup\partial\Omega\times[t_0,\infty)$}\quad \forall 0<\eta<\eta_1(M),0<\3<1,\,\,M>M_2.
\end{equation}
By Lemma \ref{lem-super-soluiton-v-o-eta-sum-and-powers} $v_{\eta}$ is a supersolution of \eqref{fde}. Hence by \eqref{u-eta-m-limit}, \eqref{u-epsilon-m<v-eta}, Lemma \ref{comparison-lem-bd-domain} and Remark \ref{remark-about-maximal-time-T-M-to-infty-as-M-to-infty} ,
\begin{align*}
&u_{\3, M}(x,t)\leq v_{\eta}(x) \quad \mbox{ in }\Omega\times[t_0,\infty)\quad \forall 0<\eta<\eta_1(M), 0<\3<1,M>M_2\\
\Rightarrow\quad &u_{\3,M}(x,t)\leq \left[C_0^m+\sum_{i=1}^{i_0}\left(A_i\left|x-a_i\right|^{-\gamma_i'}\right)^m\right]^{\frac{1}{m}} \quad \mbox{ in }\Omega\times[t_0,\infty)\quad \forall 0<\3<1,M>M_2\\
\Rightarrow \quad &u_M(x,t)\leq \left[C_0^m+\sum_{i=1}^{i_0}\left(A_i\left|x-a_i\right|^{-\gamma_i'}\right)^m\right]^{\frac{1}{m}}
\quad \mbox{ in }\Omega\times[t_0,T_M)\quad \forall M>M_2\quad\mbox{ as }\3\to 0\\
\Rightarrow \quad &u(x,t)\leq \left[C_0^m+\sum_{i=1}^{i_0}\left(A_i\left|x-a_i\right|^{-\gamma_i'}\right)^m\right]^{\frac{1}{m}}
\quad \mbox{ in }\Omega\times[t_0,\infty)\quad\mbox{ as }M\to\infty
\end{align*}
and \eqref{eq-bound-of-solution-on-compact-subset-with-t01-Omega-delta-2}, \eqref{eq-local-behaviour-of-solution-u-blow-up=0-above-by-1}, follows.
\end{proof}

By an argument similar to the proof of Corollary 2.2 of \cite{DS1} (cf. proof of Lemma 3.2 of \cite{HK1}) we have the following lemma.

\begin{lemm}\label{L1-contraction-lem}
Let $n\ge 3$, $0<m<\frac{n-2}{n}$, $u_{0,1}\ge 0$, $u_{0,2}\ge 0$, and let $u_1$, $u_2$ be two solutions of \eqref{Cauchy-problem} in $\R^n\times (0,T)$ with $u_0=u_{0,1}$, $u_{0,2}$,  respectively. Suppose $(u_{0,1}-u_{0,2})_+\in L^1(\R^n)$ and for any $0<T_1<T$ there exist constants $r_0>0$, $C>0$,  such that $u_2(x,t)\ge C/|x|^{\frac{2}{1-m}}$ for all $|x|\ge r_0$, $0<t<T_1$ holds. Then 
\begin{equation*}
\int_{\R^n}(u_1-u_2)_+(x,t)\,dx\le\int_{\R^n}(u_{0,1}-u_{0,2})_+\,dx\quad\forall 0<t<T.
\end{equation*}
\end{lemm}

\begin{lemm}\label{upper-bound-blow-up-soln-Rn1}
Let $n\geq 3$, $0<m<\frac{n-2}{n}$, $0<\delta_3<\delta_1<\min(1,\delta_0)$ and $\mu_0>0$. Let $\mu_0\leq u_0\in L^{p}_{loc}(\widehat{\R^n})$ for some constant $p>\frac{n(1-m)}{2}$ satisfy \eqref{u0-lower-blow-up-rate} and \eqref{upper-blow-up-rate-initial-data} with $i_1=i_0$ for some constants satisfying
 \eqref{gamma-upper-lower-bd3} and $\lambda_1$, $\cdots$, $\lambda_{i_0}$, $\lambda_1'$, $\cdots$, $\lambda_{i_0}'\in\R^+$. Suppose that there exist constants $R_1>R_0$ and $C_1>0$ such that \eqref{u0-infty-behaviour} holds. Let $u$ be the solution of \eqref{cauchy-blow-up-problem} constructed in Theorem \ref{second-main-existence-thm}. Then for any $\3_1>0$ there exist constants $t_0>0$ and $A_i>0$, $i=1,\cdots,i_0$, depending of $\3_1$ such that
\begin{equation}\label{upper-bound-of-u-by-constant-and-v-eta}
u(x,t)\leq\left(\left(C_1+\3_1\right)^m+\sum_{i=1}^{i_0}\left(A_i\left|x-a_i\right|^{-\gamma_i}\right)^m\right)^{\frac{1}{m}}\qquad \forall x\in\widehat{\R^n},t\geq t_0.
\end{equation}
\end{lemm}
\begin{proof}
By Theorem 2.5 of \cite{Hu3} there exists a solution $w$ of
\begin{equation*}
\begin{cases}
\begin{aligned}
w_t&=\La w^m \quad \mbox{in $(B_{4R_1}\setminus B_{2R_1})\times (0,\infty)$}\\
w(x,t)&=\infty \quad\,\,\, \mbox{ on }\partial B_{2R_1}\times (0,\infty)\cup B_{4R_1}\times (0,\infty)\\
w(x,0)&=C_1 \quad\,\,\, \mbox{in $B_{4R_1}\setminus B_{2R_1}$}
\end{aligned}
\end{cases}
\end{equation*}
such that if $w_k$, $k>C_1$, is the solution of
\begin{equation}\label{w-lateral-bdary=k}
\begin{cases}
\begin{aligned}
w_t&=\La w^m \quad \mbox{in $(B_{4R_1}\setminus B_{2R_1})\times (0,\infty)$}\\
w(x,t)&=k \qquad \mbox{ on }\partial B_{2R_1}\times (0,\infty)\cup \1 B_{4R_1}\times (0,\infty)\\
w(x,0)&=C_1 \quad\,\,\, \mbox{in $B_{4R_1}\setminus B_{2R_1}$}
\end{aligned}
\end{cases}
\end{equation}
given by Theorem 2.2 of \cite{Hu3}, then $w_k$ increases uniformly on every compact subset of $(B_{4R_1}\setminus B_{2R_1})\times (0,\infty)$ to $w$ as $k\to\infty$. Since $w_k(x,t)$ is radially symmetric in $x$, $w(x,t)$ is radially symmetric in $x$.
Let $\3_1>0$. By \eqref{w-lateral-bdary=k}, Lemma \ref{lem-from-Hu2-and-DFK-convergence-as-t-to-zero} and an argument similar to the proof of Lemma \ref{lem-local-lower-bound-of-u-near-singular-points}, there exists a constant $t_0>0$ such that
\begin{equation}\label{w-lateral-bd}
w(x,t)\le w(x,0)+\3_1=C_1+\3_1 \qquad \forall |x|=3R_1,0\le t\le t_0.
\end{equation}
For any $t_2>t_1>0$, let 
\begin{equation*}
k_0=\max_{\substack{x\in \partial B_{2R_1}\cup \1 B_{4R_1}\\t_1\le t\le t_2}}u(x,t).
\end{equation*}
Then by Lemma 2.1 of \cite{Hu3},
\begin{align}\label{u<w-ineqn}
&\int_{B_{4R_1}\setminus B_{2R_1}}(u-w_k)_+(x,t)\,dx\le\int_{B_{4R_1}\setminus B_{2R_1}}(u-w_k)_+(x,t_1)\,dx\quad\forall t_1\le t\le t_2,k\ge k_0\notag\\
\Rightarrow\quad&\int_{B_{4R_1}\setminus B_{2R_1}}(u-w)_+(x,t)\,dx\le\int_{B_{4R_1}\setminus B_{2R_1}}(u-w)_+(x,t_1)\,dx\quad\forall t_1\le t\le t_2\quad\mbox{ as }k\to\infty\notag\\
\Rightarrow\quad&\int_{B_{4R_1}\setminus B_{2R_1}}(u-w)_+(x,t)\,dx\le 0\quad\forall t>0\quad\mbox{ as }t_1\to 0, t_2\to\infty\notag\\
\Rightarrow\quad&u(x,t)\le w(x,t)\qquad\qquad\qquad\forall 2R_1<|x|<4R_1, t>0. 
\end{align}
For any $M>0$, let $u_M$ be the solution of \eqref{cauchy-problem-Rn-2}. Then
by the proof of Theorem \ref{second-main-existence-thm}, \eqref{w-lateral-bd} and \eqref{u<w-ineqn},
\begin{align}\label{u-bd-above}
&u(x,t)\le C_1+\3_1 \qquad \forall |x|=3R_1,0\le t\le t_0\notag\\
\Rightarrow\quad&u_M(x,t)\le C_1+\3_1 \qquad \forall |x|=3R_1,0\le t\le t_0, M>0.
\end{align}
Since $C_1+\3_1$ is a solution of \eqref{fde} in $(\R^n\setminus B_{3R_1})\times(0,\infty)$ and the solution $u_M$ of \eqref{cauchy-problem-Rn-2} is unique by Theorem 2.3 of \cite{HP}, by \eqref{u0-infty-behaviour}, \eqref{u-bd-above}, the construction of solution of \eqref{cauchy-problem-Rn-2}  with initial values $u_{0,M}$ in Theorem 1.1 of \cite{Hs} which approximate solution of \eqref{cauchy-problem-Rn-2} by a sequence of solutions in bounded cylindrical domains and Lemma \ref{comparison-lem-bd-domain}, 
\begin{equation}\label{u-exterior-bd0}
u_M(x,t)\leq C_1+\3_1 \quad \forall |x|\geq 3R_1,0\leq t\leq t_0, M>C_1.
\end{equation}
Let $\Omega=B_{3R_1}$ and $0<\delta_4<\min (\delta_3,R_1)$. By Lemma \ref{lem-Cor-2-2-of-Hs1a} and a compactness argument there exists a constant $C_2=C_2\left(\delta_4\right)>0$ such that
\begin{equation}\label{eq-bound-of-w-epsilon-M-on-B-R-1-and-over-delta-4-holes}
u_M(x,t_0)\le C_2 \qquad \mbox{ in }\Omega_{\delta_4}\quad\forall M>0.
\end{equation}
By \eqref{upper-blow-up-rate-initial-data} and an argument similar to the proof of Lemma \ref{lem-local-upper-bound-of-u-near-singular-points}, there exists a constant $A_0'>0$ such that \eqref{eq-upper-bound-of-u-epsilon-M-by-M-in-whole-and-r-minus-gamma-dash-on-ball} and \eqref{eq-upper-bound-of-v-eta-i-epsilon-M-by-M-in-whole-and-r-minus-gamma-dash-on-ball} hold  with  $v_{\eta}$, $v_{i,\eta}$ being given by \eqref{eq-def-of-function-v-i-eta-and-v-0=eta} and \eqref{eq-def-of-v-i-eta-for-super-solution-of-harmonic} with  
$$
0<\eta<\eta_1(M):=\min\left(1,\min_{1\leq i\leq i_0}\left(\frac{A_0'}{M+1}\right)^{\frac{2}{\gamma_i'}}\right),\qquad
A_1=\max(2^{\gamma_1'/2}A_0', C_2(9R_1^2+1)^{\gamma_1/2}),
$$
$A_i=2^{\gamma_i'/2}A_0'$, $i=2, \cdots, i_0$, and $C_0=C_1+\3_1$.    Since
\begin{equation}\label{eq-bounded-of-w-small=away-from-the-hol23--0}
v_{1,\eta}(x)=A_1(|x-a_i|^2+\eta)^{-\frac{\gamma_1}{2}}\ge A_1(9R_1^2+1)^{-\frac{\gamma_1}{2}}\ge C_2 \quad \mbox{ in }\Omega_{\delta_4}\quad\forall 0<\eta<\eta_1(M),
\end{equation}
by  \eqref{eq-upper-bound-of-u-epsilon-M-by-M-in-whole-and-r-minus-gamma-dash-on-ball},  \eqref{eq-upper-bound-of-v-eta-i-epsilon-M-by-M-in-whole-and-r-minus-gamma-dash-on-ball}, \eqref{eq-bound-of-w-epsilon-M-on-B-R-1-and-over-delta-4-holes}, \eqref{eq-bounded-of-w-small=away-from-the-hol23--0},
\begin{equation}\label{eq-upper-bound-of-w-by-C-1-eosulon-on-boundary-inside-B-R-2} 
u_M(x,t_0)\leq v_{\eta}(x)\quad \forall x\in\R^n, 0<\eta<\eta_1(M), M>C_1.
\end{equation}
By \eqref{eq-upper-bound-of-w-by-C-1-eosulon-on-boundary-inside-B-R-2}  and Lemma \ref{L1-contraction-lem},
\begin{align*}
&u_M(x,t)\leq v_{\eta}(x)\quad \forall x\in\R^n, t\ge t_0, 0<\eta<\eta_1(M), M>C_1\\
\Rightarrow\quad&u(x,t)\leq\left(\left(C_1+\3_1\right)^m+\sum_{i=1}^{i_0}\left(A_i\left|x-a_i\right|^{-\gamma_i}\right)^m\right)^{\frac{1}{m}}\qquad \forall x\in\widehat{\R^n},t\geq t_0\quad\mbox{ as }M\to\infty
\end{align*}
and \ref{upper-bound-of-u-by-constant-and-v-eta} follows.
\end{proof}

\begin{proof}[\textbf{Proof of Theorem \ref{convergence-thm1}}:]
For any $0<\3<1$ and $M>0$, let $u_{\3,M}$ be the solution of \eqref{eq-aligned-problem-related-with-epsilon-main-one}. Then \eqref{u-eta-m-limit} holds. By \eqref{u-eta-m-limit} and Lemma \ref{comparison-lem-bd-domain},
\begin{align}
&u_{\3,M}\ge\mu_0+\3 \quad\mbox{ in }\2{\Omega}\times(0,\infty)\label{eq-lower-bound-of-u-from-intial-condition-u-0-mu}\\
\Rightarrow\quad& u(x,t)\geq \mu_0 \qquad\mbox{ in }(\overline{\Omega}\bs\left\{a_1,\cdots,a_{i_0}\right\})\times(0,\infty)\quad\mbox{ as }\3\to 0, M\to\infty.\label{u-lower-bd10}
\end{align}
Let $\{t_k\}_{k=1}^{\infty}\subset\R^+$ be a sequence such that $t_k\to\infty$ as $k\to\infty$ and $u_{k}(x,t)=u(x,t_k+t)$. Let $N_1>0$. We choose $k_{N_1}\in\Z^+$ such that $t_k>N_1$ for any $k\ge k_{N_1}$. Then by \eqref{u-lower-bd10} and Lemma \ref{lem-local-upper-bound-of-u-near-singular-points} the equation \eqref{fde} for the sequence $\{u_k\}_{k={N_1}}^{\infty}$  is uniformly parabolic on $\overline{\Omega_{\delta}}\times(T_1+1-N_1,\infty)$ for any $0<\delta<\delta_0$. By the Schauder estimates \cite{LSU}  the sequence $\{u_k\}_{k={N_1}}^{\infty}$ is equi-H\"older continuous in $C^{2,1}(K)$ for any compact subset $K$ of $\left\{\overline{\Omega}\bs\left\{a_1,\cdots,a_{i_0}\right\}\right\}\times(T_1+1-N_1,\infty)$. Hence by Ascoli theorem and a diagonalization argument the sequence $\{u_k\}_{k=1}^{\infty}$ has a subsequence which we may assume without loss of generality to be the sequence itself that converges uniformly in $C^{2,1}(K)$ on every compact subset $K$ of $(\overline{\Omega}\bs\left\{a_1,\cdots,a_{i_0}\right\})\times\left(-\infty,\infty\right)$ to a solution $u_{\infty}$ of \eqref{fde} in $(\overline{\Omega}\bs\left\{a_1,\cdots,a_{i_0}\right\})\times(-\infty,\infty)$ as $k\to\infty$ which satisfies
\begin{equation}\label{eq-boundary-condtion-of-u-infty-as-t-to-infty}
u_{\infty}=\mu_0 \quad \mbox{  on }\partial\Omega\times\R.
\end{equation}
We now divide the proof into two cases.

\noindent\textbf{Case 1}: There exists a constant $t_0>0$ such that 
$f\equiv\mu_0$ on $\partial\Omega\times(t_0,\infty)$.

\noindent Then 
\begin{equation}\label{eq-condition-of-boundary-value-after-time-T-1}
u_{\3,M}=\mu_0+\3 \quad \mbox{ on }\partial\Omega\times(t_0,\infty).
\end{equation}
Let $\gamma_0=\max_{1\leq i\leq i_0}\gamma_i'$. Then by \eqref{gamma-upper-lower-bd1},
\begin{equation*}
\frac{2}{1-m}<\gamma_0<n.
\end{equation*}
Let $v_0(x)=u_{\infty}(x,0)$ and $1<p_1<n/\gamma_0$. By Lemma \ref{lem-local-upper-bound-of-u-near-singular-points} there exist constants $C_2>0$ and $C_3>0$ such that
\eqref{eq-bound-of-solution-on-compact-subset-with-t01-Omega-delta-2}
and \eqref{eq-local-behaviour-of-solution-u-blow-up=0-above-by-1} holds with $\delta_2=\delta_1$. Then by \eqref{eq-bound-of-solution-on-compact-subset-with-t01-Omega-delta-2}
and \eqref{eq-local-behaviour-of-solution-u-blow-up=0-above-by-1},
\begin{equation*}
\int_{\widehat{\Omega}}u^{p_1}\left(x,t_0\right)\,dx\leq C_3^p\omega_n\sum_{i=1}^{i_0}\int_{0}^{\delta_1}r^{n-p_1\gamma_0-1}\,dr+\|u(\cdot,t_0)\|^{p_1}_{L^{\infty}(\Omega_{\delta_1})}|\Omega|\leq C_1(1+\delta_1^{n-p_1\gamma_0})<\infty
\end{equation*}
for some constant $C_1>0$ where $\omega_n$ is the surface area of the unit sphere $S^{n-1}$ in $\R^n$. Hence $u(\cdot,t_0)\in L^{p_1}(\widehat{\Omega})$.
By \eqref{eq-lower-bound-of-u-from-intial-condition-u-0-mu} and \eqref{eq-condition-of-boundary-value-after-time-T-1}, 
\begin{equation}\label{eq-boundary-condition-negative-of-our-normal-direction-on-boundary}
\frac{\partial u_{\3,M}}{\partial\nu}\leq 0\qquad \mbox{ on }\partial\Omega\times(t_0,\infty)
\end{equation}
where $\frac{\partial}{\partial\nu}$ is the derivative with respect to the unit outward normal on $\partial\Omega\times(t_0,\infty)$. Then by \eqref{eq-boundary-condition-negative-of-our-normal-direction-on-boundary},
\begin{align*}
\frac{1}{p_1}\left(\int_{\Omega_{\delta}}u_{\3,M}^{p_1}(x,t)\,dx-\int_{\Omega_{\delta}}u_{\3,M}^{p_1}(x,t_1)\,dx\right)
=&\frac{1}{p_1}\int_{t_0}^t\frac{d}{d\tau}\left[\int_{\Omega_{\delta}}u_{\3,M}^{p_1}(x,\tau)\,dx\right]d\tau\\
=&\int_{t_0}^t\int_{\Omega_{\delta}}u_{\3,M}^{p_1-1}(u_{\3,M})_{\tau}\,dx\,d\tau=\int_{t_0}^t\int_{\Omega_{\delta}}u_{\3,M}^{p_1-1}\La u_{\3,M}^m\,dx\,d\tau\\
\le&-\int_{t_0}^t\int_{\Omega_{\delta}}\nabla u_{\3,M}^m\cdot\nabla u_{\3,M}^{p_1-1}\,dx\,d\tau\\ 
=&-m(p_1-1)\int_{t_0}^{t}\int_{\Omega_{\delta}}u_{\3,M}^{m+p_1-3}|\nabla u_{\3,M}|^2\,dx\,d\tau\quad\forall t>t_0,0<\delta\leq \delta_1. 
\end{align*}
Hence
\begin{equation}\label{eq-integration-by-parts-for-u-epsilon-M-to-p-before-letting-epsilon-m-to}
\frac{1}{p_1}\int_{\Omega_{\delta}}u_{\3,M}^{p_1}(x,t)\,dx+m(p_1-1)\int_{t_0}^t\int_{\Omega_{\delta}}u_{\3,M}^{m+p_1-3}|\nabla u_{\3,M}|^2\,dx\,d\tau\leq \frac{1}{p_1}\int_{\Omega_{\delta}}u_{\3,M}^{p_1}(x,t_1)\,dx\quad\forall t>t_0, 0<\delta\leq \delta_1.  
\end{equation}
Letting $\3\to 0$, $M\to\infty$ in \eqref{eq-integration-by-parts-for-u-epsilon-M-to-p-before-letting-epsilon-m-to}, by \eqref{u-eta-m-limit},
\begin{align}\label{eq-bound-of-u-to-sudden-power-grad-u-squre-on-Omega-deta}
&\frac{1}{p_1}\int_{\Omega_{\delta}}u^{p_1}(x,t)\,dx+m(p_1-1)\int_{t_0}^t\int_{\Omega_{\delta}}u^{m+p_1-3}|\nabla u|^2\,dx\,d\tau\leq \frac{1}{p_1}\int_{\Omega_{\delta}}u^{p_1}(x,t_0)\,dx\quad\forall t>t_0, 0<\delta\le \delta_1\notag\\
\Rightarrow \quad &\int_{t_0}^{\infty}\int_{\Omega_{\delta}}\frac{1}{u^{3-m-p_1}}|\nabla u|^2\,dx\,dt<\frac{1}{p_1(p_1-1)m}\int_{\widehat{\Omega}}u^{p_1}(x,t_0)\,dx<\infty\quad\mbox{ as }t\to\infty.
\end{align}
Let $0<\delta<\delta_1$. By Lemma \ref{lem-local-upper-bound-of-u-near-singular-points} there exists a constant $C_{\delta}>0$ such that
\begin{equation}\label{eq-upper-bound-of-u-by-C-delta-on-Omega-delta-t-1-over-2-to-infty}
u(x,t)\leq C_{\delta} \qquad \mbox{on $\Omega_{\delta}\times(t_0,\infty)$}.
\end{equation}
By \eqref{u-lower-bd10}, \eqref{eq-bound-of-u-to-sudden-power-grad-u-squre-on-Omega-deta} and \eqref{eq-upper-bound-of-u-by-C-delta-on-Omega-delta-t-1-over-2-to-infty},
\begin{equation}\label{eq-upper-bound-of-integral-of-u-by-C-delta-on-Omega-delta-t-1-over-2-to-infty}
\int_{t_0}^{\infty}\int_{\Omega_{\delta}}|\nabla u|^2\,dx\,dt\le C_{\delta}' 
\end{equation}
for some constant $C_{\delta}'>0$ depending on $\delta>0$. Observe that
\begin{equation}\label{eq-equality-of-integral-u-and-integral-u-k-with-different-range}
\int_{0}^{\infty}\int_{\Omega_{\delta}}|\nabla u_k(x,t)|^2\,dx\,dt=\int_{t_k}^{\infty}\int_{\Omega_{\delta}}|\nabla u(x,s)|^2\,dx\,ds\quad\forall k\in\Z^+, 0<\delta<\delta_1.
\end{equation}
Letting $k\to\infty$ in \eqref{eq-equality-of-integral-u-and-integral-u-k-with-different-range}, by \eqref{eq-upper-bound-of-integral-of-u-by-C-delta-on-Omega-delta-t-1-over-2-to-infty} and Fatou's lemma,
\begin{align}
&\int_{0}^{\infty}\int_{\Omega_{\delta}}|\nabla u_{\infty}|^2\,dx\,dt=0\quad \forall 0<\delta<\delta_1\notag\\
\Rightarrow \quad & \nabla u_{\infty}=0 \qquad \mbox{ in $\widehat{\Omega}\times[0,\infty)$}\qquad \mbox{as $\delta\to 0$}.\label{eq-gradient-of-u-sub-infty-equal-to-zero}
\end{align}
Thus by \eqref{eq-boundary-condtion-of-u-infty-as-t-to-infty} and \eqref{eq-gradient-of-u-sub-infty-equal-to-zero},
\begin{align*}
&u_{\infty}=\mu_0 \quad \mbox{ on }(\overline{\Omega}\bs\{a_1,\cdots,a_{i_0}\})\times[0,\infty)\\
\Rightarrow \qquad & v_{0}=\mu_0 \quad \mbox{ on }\overline{\Omega}\bs\{a_1,\cdots,a_{i_0}\}.
\end{align*}
Since the sequence $\{t_k\}_{k=1}^{\infty}$ is arbitrary, $u$ satisfies \eqref{u-infty-to-mu0} for any compact subset $K$ of $\overline{\Omega}\bs\left\{a_1,\cdots,a_{i_0}\right\}$.

\noindent\textbf{Case 2}: $f$ satisfies \eqref{eq-first-condition-on-f-constant-mu-sub-zero} and \eqref{eq-second-condition-on-f-convergence-to-mu-sub-one}.\\
By \eqref{eq-second-condition-on-f-convergence-to-mu-sub-one} for any $i\geq 2$ there exists a constant $T_i>0$ such that
\begin{equation}\label{eq-condition-of-f-trapped-by-two-constants-one-plus-and-the-ohter-minus}
\mu_0\leq f\leq \mu_0\left(1+\frac{1}{i}\right) \qquad \mbox{on $\partial\Omega\times(T_i,\infty)$}.
\end{equation}
For any $i\ge 2$, let $\underline{f}_i$, $\overline{f}_i$, be given by
\begin{equation}\label{lower-upper-approx-lateral-values}
\left\{\begin{aligned}
&\underline{f}_i(x,t)=\mu_0\qquad\qquad\qquad\qquad\qquad \mbox{ on }\partial\Omega\times(0,\infty)\\
&\overline{f}_i(x,t)=\max\left(f(x,t),\mu_0\left(1+\frac{1}{i}\right)\right) \quad \mbox{ on }\partial\Omega\times(0,\infty)
\end{aligned}\right.
\end{equation}
and $\underline{u}_{0,i}$, $\overline{u}_{0,i}$ be given by
\begin{equation}\label{lower-upper-approx-initial-value}
\underline{u}_{0,i}(x,0)=u_0\quad\mbox{ and }\quad \overline{u}_{0,i}(x,0)=u_0+\frac{\mu_0}{i}.
\end{equation}
Let $\underline{u}_i$, $\overline{u}_i$, be solutions of \eqref{Dirichlet-blow-up-problem} with $f=\underline{f}_i$, $\overline{f}_i$, and $u_0=\underline{u}_{0,i}, \overline{u}_{0,i}$, respectively given by Theorem \ref{first-main-existence-thm}.
Since 
\begin{equation*}
\underline{f}_i(x,t)=\mu_0 \quad \mbox{ and }\quad
\overline{f}_i(x,t)=\mu_0\left(1+\frac{1}{i}\right)\quad \mbox{ on }\partial\Omega\times(T_i,\infty),
\end{equation*}
by  case 1, 
\begin{equation}\label{eq-uniform-convergence-of-lower-one-by-Case-1}
\left\{\begin{aligned}
&\underline{u}_i\to \mu_0 \qquad\qquad\quad \mbox{in $C^2(K)$ as $t\to\infty$}\\
&\overline{u}_i\to \mu_0\left(1+\frac{1}{i}\right) \quad \mbox{ in $C^2(K)$ as $t\to\infty$}
\end{aligned}\right.
\end{equation}
for any compact subset $K$ of $\overline{\Omega}\bs\left\{a_1,\cdots,a_n\right\}$.
By \eqref{lower-upper-approx-lateral-values}, \eqref{lower-upper-approx-initial-value} and Theorem \ref{bd-soln-comparison-thm1},
\begin{align}
&\underline{u}_i\leq u\leq \overline{u}_i \quad \mbox{ in }\Omega\times(0,\infty)\quad\forall i\geq 2\notag\\
\Rightarrow \quad &\underline{u}_i\left(x,t+t_k\right)\leq u_{k}(x,t)\leq \overline{u}_i\left(x,t+t_k\right) \qquad \forall (x,t)\in\widehat{\Omega}\times(-t_k,\infty),\,\,i\geq 2.\label{eq-comparison-bteween-underline-u-k-and-u-and-overline-u-k}
\end{align}
Thus letting $k\to\infty$ in \eqref{eq-comparison-bteween-underline-u-k-and-u-and-overline-u-k}, by \eqref{eq-uniform-convergence-of-lower-one-by-Case-1}, 
\begin{equation*}
\begin{aligned}
&\mu_0\leq u_{\infty}(x,t)\leq \mu_0\left(1+\frac{1}{i}\right) \qquad \forall (x,t)\in\widehat{\Omega}\times\R,i\geq 2\\
\Rightarrow \quad&u_{\infty}=\mu_0 \quad\mbox{ in } \widehat{\Omega}\times\R\quad\mbox{as $i\to\infty$} 
\end{aligned}
\end{equation*}
and \eqref{u-infty-to-mu0} holds for any compact subset $K$ of $\overline{\Omega}\bs\left\{a_1,\cdots,a_n\right\}$.
\end{proof}

\begin{remar}
If one only assume that $f\in L^{\infty}(\partial\Omega\times\left(0,\infty\right))$ 
and
\begin{equation*}
f(x,t)\to\mu_0 \quad \mbox{uniformly on }\partial\Omega\mbox{ as }t\to\infty
\end{equation*}
instead of $f\in L^{\infty}(\partial\Omega\times\left(0,\infty\right))\cap C^{\infty}(\partial\Omega\times\left(T_1,\infty\right))$ for some constant $T_1>0$ and condition \eqref{eq-second-condition-on-f-convergence-to-mu-sub-one} in Theorem \ref{convergence-thm1}, then by an argument similar to the proof of Theorem \ref{convergence-thm1} one can prove that the solution $u$ of \eqref{Dirichlet-blow-up-problem} given by Theorem \ref{first-main-existence-thm} satisfy
\eqref{u-infty-to-mu0} for any compact set $K\subset\widehat{\Omega}$. Moreover
\begin{equation*}
u(x,t)\to \mu_0 \quad \mbox{ in }L_{loc}^{\infty}(\overline{\Omega}\bs\left\{a_1,\cdots,a_{i_0}\right\}) \quad \mbox{ as }t\to\infty.
\end{equation*}
\end{remar}

\begin{proof}[\textbf{Proof of Theorem \ref{convergence-thm2}}:]
Let $\left\{t_k\right\}_{k=1}^{\infty}\subset\R^+$ be a sequence such that $t_k\to\infty$ as $k\to\infty$ and let $u_k(x,t)=u(x,t+t_k)$. By the same argument as the proof of Theorem \ref{convergence-thm1} the sequence $\{u_k\}_{k=1}^{\infty}$ has a subsequence which we may assume without loss of generality to be the sequence itself that converges uniformly in $C^{2,1}(K)$ for every compact subset $K$ of $(\overline{\Omega}\bs\left\{a_1,\cdots,a_{i_0}\right\})\times\left(-\infty,\infty\right)$ to a solution $u_{\infty}$ of \eqref{fde} in $\widehat{\Omega}\times(-\infty,\infty)$ as $k\to\infty$ which satisfies 
\begin{equation}\label{eq-boundary-condition-of-u-infty-by-f-mu-1}
u_{\infty}=\mu_1 \qquad \mbox{on $\partial\Omega\times\left(-\infty,\infty\right)$.}
\end{equation} 
Without loss of generality we will assume that $T_1=\frac{1}{2}$.
We divide the proof into two cases.

\noindent\textbf{Case 1}: $\gamma_i=\gamma_i'$ for all $1\leq i\leq i_0$.\\
Let $T>1$. Then 
\begin{align}\label{eq-for-thm-a-aligned-lower-and-uppoer-bound-of-integral-Laplace-of-u-epsilon-M}
\int_1^T\int_{\Omega_{\delta}}(u^m)_tu_t\,dx\,dt&=\int_{1}^{T}\int_{\Omega_{\delta}}(u^m)_t\La u^m\,dx\,dt\notag\\
& =-\frac{1}{2}\int_1^T\frac{d}{dt}\left[\int_{\Omega_{\delta}}|\nabla u^m|^2\,dx\right]dt +\int_1^T\int_{\partial\Omega_{\delta}}(u^m)_{t}\frac{\partial u^m}{\partial\nu}\,d\sigma\, dt\notag\\
& \leq \frac{m^2}{2}\int_{\Omega_{\delta}}u^{2(m-1)}|\nabla u|^2(x,1)\,dx+m^2\sum_{i=1}^{i_0}\int_1^T\int_{\partial B_{\delta}(a_i)}u^{2(m-1)}|u_t||\nabla u|\,d\sigma\,dt\notag\\
&\qquad +m^2\int_1^T\int_{\partial\Omega}f^{2(m-1)}|f_t||\nabla u|\,d\sigma\, dt\quad\forall 0<\delta<\delta_1
\end{align}
where $\frac{\partial}{\partial\nu}$ is the derivative with respect to the unit outward normal on $\partial\Omega_{\delta}$.

Let $\delta_2=\frac{\delta_1}{2}$ and $i\in\{1,\cdots,i_0\}$.  We claim that  there exists a constant $C_1(T)>0$ such that
\begin{equation}\label{v-i-tidle-x-derivative-bd}
|\nabla u(x,t)|\le\frac{C_1(T)}{|x-a_i|^{\gamma_i+\frac{\gamma_i(1-m)}{2}}}\quad \forall (x,t)\in \widehat{B}_{\delta_2}(a_i)\times[1,T),i=1,\cdots,i_0 
\end{equation}
and
\begin{equation}\label{v-i-tidle-t-derivative-bd}
|u_t(x,t)|\le\frac{C_1(T)}{|x-a_i|^{\gamma_i}}\quad \forall (x,t)\in \widehat{B}_{\delta_2}(a_i)\times[1,T),i=1,\cdots,i_0
\end{equation}
hold. To prove the claim we let $x_0\in \widehat{B}_{\delta_2}\left(a_i\right)$,  $R_i=|x_0-a_i|$,
\begin{equation*}\label{x-y-relation}
y_{0,i}=R_i^{-\frac{\gamma_i\left(1-m\right)}{2}}(x_0-a_i) \qquad \mbox{and} \qquad y=R_i^{-\frac{\gamma_i\left(1-m\right)}{2}}\left(x-a_i\right).
\end{equation*}
Then $\left|y_{0,i}\right|=R_i^{1-\frac{\gamma_i(1-m)}{2}}$, 
\begin{equation*}
y\in B\left(y_{0,i},\frac{1}{2} R_i^{1-\frac{\gamma_i(1-m)}{2}}\right) \quad \iff \quad x\in B\left(x_0,\frac{1}{2}R_i\right)
\end{equation*} 
and
\begin{equation}\label{eq-range-of-y-and-x-when-y-in-R-1-to-1-gamma-times-1-m-over-2}
 y\in B\left(y_{0,i},\frac{1}{2}R_i^{1-\frac{\gamma_i\left(1-m\right)}{2}}\right)\quad \Rightarrow \quad \frac{1}{2}R_i^{1-\frac{\gamma_i\left(1-m\right)}{2}}\leq \left|y\right|\leq \frac{3}{2}R_i^{1-\frac{\gamma_i\left(1-m\right)}{2}} \quad \Rightarrow \quad \frac{1}{2}R_i\leq \left|x-a_i\right|\leq \frac{3}{2}R_i.
\end{equation} 
Let
\begin{equation*}
\tilde{v}_{i}(y,t)=R_i^{\gamma_i}u\left(a_i+R_i^{\frac{\gamma_i\left(1-m\right)}{2}}y,t\right) \qquad \forall (y,t)\in B\left(y_{0,i},\frac{1}{2}R_i^{1-\frac{\gamma_i\left(1-m\right)}{2}}\right)\times\left(0,T\right) ,i=1,\cdots,i_0.
\end{equation*}
Then $\tilde{v}_i$ satisfies
\begin{equation*}
\tilde{v}_t=\La_y\tilde{v}^m=m\tilde{v}^{m-1}\La_y\tilde{v}+m(m-1)\tilde{v}^{m-2}\left|\nabla_y\tilde{v}\right|^2 \quad \mbox{in $B\left(y_{0,i},\frac{1}{2} R_i^{1-\frac{\gamma_i\left(1-m\right)}{2}}\right)\times\left(0,T\right)$}\quad\forall i=1,\cdots,i_0.
\end{equation*}
By Lemma \ref{lem-local-lower-bound-of-u-near-singular-points}, Lemma \ref{lem-local-upper-bound-of-u-near-singular-points} and \eqref{eq-range-of-y-and-x-when-y-in-R-1-to-1-gamma-times-1-m-over-2}, there exist constants $C_5>C_2(T)>0$ such that
\begin{equation}\label{v-tidle-i-upper-bd0}
C_2(T)\leq \tilde{v}_i(y,t) \leq C_5 \qquad \forall y\in B\left(y_{0,i},\frac{1}{2}R_i^{1-\frac{\gamma_i\left(1-m\right)}{2}}\right),\frac{1}{2}\leq t<T, i=1,\cdots,i_0.
\end{equation}
Then the equation for $\tilde{v}_i$ is uniformly parabolic equation in $B\left(y_{0,i},\frac{1}{2}R_i^{1-\frac{\gamma_i\left(1-m\right)}{2}}\right)\times\left[\frac{1}{2},T\right)$ for any $i=1,\cdots,i_0$. Since  $R_i<1$ and $\gamma_i>\frac{2}{1-m}$, $R_i^{1-\frac{\gamma_i\left(1-m\right)}{2}}>1$. Hence by \eqref{v-tidle-i-upper-bd0} and the parabolic Schauder estimate \cite{LSU}  there exists a constant $C_4(T)>0$ independent of $R_i$ such that 
\begin{equation*}
\sup_{|y-y_{0,i}|\leq \frac{1}{4}}|\nabla_y \tilde{v}_i|(y,t_0)+\sup_{|y-y_{0,i}|\leq \frac{1}{4}}|\tilde{v}_{i,t}|(y,t_0)\leq C_4(T)\sup_{\begin{subarray}{c} |y-y_{0,i}|\leq\frac{1}{2}\\t_0-\frac{1}{4}\leq t\leq t_0\end{subarray}}|\tilde{v}_i|(y,t)=C_5C_4(T)=C_1(T)\,\forall 1\leq t_0<T, i=1,\cdots,i_0. 
\end{equation*}
Thus
\begin{align*}
&|\nabla_y \tilde{v}_i(y_{0,i},t_0)|\leq C_1(T) \qquad\qquad \forall 1\leq t_0<T,|y_{0,i}|<R_i^{-\frac{\gamma_i(1-m)}{2}}\delta_2,\,\,i=1,\cdots,i_0\notag\\
\Rightarrow\quad&|\nabla u(x,t)|\leq \frac{C_1\left(T\right)}{|x-a_i|^{\gamma_i+\frac{\gamma_i(1-m)}{2}}} \qquad \forall (x,t)\in \widehat{B}_{\delta_2}(a_i)\times[1,T),i=1,\cdots,i_0
\end{align*}
and 
\begin{align*}
&|\tilde{v}_{i,t}(y_{0,i},t_0)|\leq C_1(T) \qquad \forall 1\leq t_0<T\,\,|y_{0,i}|<R_i^{-\frac{\gamma_i(1-m)}{2}}\delta_2,i=1,\cdots,i_0\notag\\
\Rightarrow\quad &\left|u_t(x,t)\right|\le\frac{C_1(T)}{|x-a_i|^{\gamma_i}} \qquad \forall (x,t)\in \widehat{B}_{\delta_2}\left(a_i\right)\times[1,T), i=1,\cdots,i_0
\end{align*}
and the claim follows. Let $t_0=1/2$.
By Lemma \ref{lem-local-upper-bound-of-u-near-singular-points} there exist constants $C_2>0$, $C_3>0$, such that \eqref{eq-bound-of-solution-on-compact-subset-with-t01-Omega-delta-2} and \eqref{eq-local-behaviour-of-solution-u-blow-up=0-above-by-1} hold. Since $u$ satisfies \eqref{u-lower-bd10}, by \eqref{eq-bound-of-solution-on-compact-subset-with-t01-Omega-delta-2} and \eqref{eq-local-behaviour-of-solution-u-blow-up=0-above-by-1} the equation \eqref{fde} for the solution $u$ is uniformly parabolic on $\overline{\Omega_{\delta}}\times\left(\frac{1}{2},\infty\right)$ for any $0<\delta<\delta_2$. Hence by the parabolic Schauder estimates \cite{LSU} there exists a constant $C_6>0$ such that
\begin{equation}\label{eq-nabla-u-bounded-by-constant-over-overline-Omega-delta-2-to=1=toinfty-by-Schauder}
\left|\nabla u\right|\leq C_6 \qquad \mbox{in $\overline{\Omega_{\delta_2}}\times[1,\infty)$}.
\end{equation}
Since $0<m<\frac{n-2}{n+2}$ and $\gamma_i\in \left(\frac{2}{1-m},\frac{n}{m+1}\right)$ for all $i=1,\dots,i_0$,
\begin{equation*}
\gamma_i<\frac{n}{m+1}<\frac{2(n-1)}{3m+1}\quad\forall i=1,\dots,i_0.
\end{equation*}
Hence by Lemma \ref{lem-local-lower-bound-of-u-near-singular-points}, \eqref{eq-bound-of-solution-on-compact-subset-with-t01-Omega-delta-2}, \eqref{eq-local-behaviour-of-solution-u-blow-up=0-above-by-1}, \eqref{v-i-tidle-x-derivative-bd}, \eqref{v-i-tidle-t-derivative-bd}  and \eqref{eq-nabla-u-bounded-by-constant-over-overline-Omega-delta-2-to=1=toinfty-by-Schauder},
\begin{align}\label{eq-control-of-the-second-term-on-tghe-right-hand-side-of-to-zero-as-delta-to-zero}
\int_1^T\int_{\partial B_{\delta}\left(a_i\right)}u^{2\left(m-1\right)}|u_t||\nabla u|\,d\sigma\, dt
&\leq CT\delta^{n-1-\frac{\gamma_i}{2}\left(3m+1\right)} \quad \forall 0<\delta<\delta_2,i=1,\cdots,i_0\notag\\
&\to 0 \quad\,\,\, \mbox{ as }\delta\to 0\quad\forall i=1,\cdots,i_0,
\end{align}
\begin{align}\label{eq-control-of-the-first-term-on-tghe-right-hand-side-of-L-1-norm-nabla-u-to-m-squre}
\int_{\Omega_{\delta}}u^{2(m-1)}|\nabla u|^2(x,1)\,dx
=&\int_{\Omega_{\delta_2}}u^{2(m-1)}|\nabla u|^2(x,1)\,dx+\sum_{i=1}^{i_0}\int_{B_{\delta_2}(a_i)}u^{2(m-1)}|\nabla u|^2(x,1)\,dx\notag\\
\le&C\left(1+\int_0^{\delta_2}r^{n-1-\gamma_i\left(m+1\right)}\,dr\right)=C'\left(1+\delta_2^{n-\gamma_i(m+1)}\right) \quad \forall 0<\delta<\delta_2,i=1,\cdots,i_0\notag\\
\Rightarrow\quad\int_{\widehat{\Omega}}u^{2(m-1)}|\nabla u|^2(x,1)\,dx\le& C'\left(1+\sum_{i=1}^{i_0}\delta_2^{n-\gamma_i(m+1)}\right)\qquad \forall i=1,\cdots,i_0\quad\mbox{ as }\delta\to 0
\end{align}
and
\begin{equation}\label{eq-control-of-the-second-term-on-tghe-right-hand-side-of-L-1-norm-nabla-f-to-m-squre-in-hat-Omega}
\int_1^{\infty}\int_{\partial\Omega}|f|^{2(m-1)}|f_t||\nabla u|\,d\sigma\, dt\leq C\|f_t\|_{L^1(\partial\Omega\times(1,\infty))}.
\end{equation}
for some constants $C>0$, $C'>0$.
Letting first $\delta\to 0$   and then $T\to\infty$ in \eqref{eq-for-thm-a-aligned-lower-and-uppoer-bound-of-integral-Laplace-of-u-epsilon-M}, by \eqref{eq-control-of-the-second-term-on-tghe-right-hand-side-of-to-zero-as-delta-to-zero}, \eqref{eq-control-of-the-first-term-on-tghe-right-hand-side-of-L-1-norm-nabla-u-to-m-squre} and \eqref{eq-control-of-the-second-term-on-tghe-right-hand-side-of-L-1-norm-nabla-f-to-m-squre-in-hat-Omega} we have
\begin{equation}\label{eq-for-thm-a-aligned-lower-and-uppoer-bound-of-integral-Laplace-of-u}
\int_1^{\infty}\int_{\widehat{\Omega}}u^{m-1}(x,t)u^2_{t}(x,t)\,dxdt\leq C\left(1+\sum_{i=1}^{i_0}\delta_2^{n-\gamma_i(m+1)}+\|f_t\|_{L^1(\partial\Omega\times(1,\infty))}\right)<\infty.
\end{equation}
We now observe that
\begin{equation}\label{eq-equality-of-integral-u-and-integral-u-k-with-different-range-2}
\int_0^{\infty}\int_{\widehat{\Omega}}u_k^{m-1}\left|\La u_{k}^m\right|^2\,dxdt
=\int_0^{\infty}\int_{\widehat{\Omega}}u_k^{m-1}u_{k,t}^2\,dxdt
=\int_{t_k}^{\infty}\int_{\widehat{\Omega}}u^{m-1}(x,s)u_t^2(x,s)\,dxds.
\end{equation}
Letting $k\to\infty$ in \eqref{eq-equality-of-integral-u-and-integral-u-k-with-different-range-2}, by \eqref{eq-bound-of-solution-on-compact-subset-with-t01-Omega-delta-2}, \eqref{eq-local-behaviour-of-solution-u-blow-up=0-above-by-1}, \eqref{eq-for-thm-a-aligned-lower-and-uppoer-bound-of-integral-Laplace-of-u} and Fatou's lemma, for any  $0<\delta<\delta_2$ there exists a constant $C_{\delta}>0$ such that
\begin{align}
C_{\delta}\int_0^{\infty}\int_{\Omega_{\delta}}\left|\La u_{\infty}^m\right|^{2}\,dx\,dt&\le\int_{0}^{\infty}\int_{\Omega_{\delta}}u_{\infty}^{m-1}\left|\La u_{\infty}^m\right|^{2}\,dx\,dt\leq\int_{0}^{\infty}\int_{\widehat{\Omega}}u_{\infty}^{m-1}\left|\La u_{\infty}^m\right|^{2}\,dxdt=0 \notag\\
\Rightarrow \qquad \qquad \qquad \qquad   \La u^m_{\infty}(x,t)&=0 \qquad \mbox{ in $\Omega_{\delta}\times[0,\infty)$}\notag\\
\Rightarrow \qquad \qquad \qquad \qquad   \La u^m_{\infty}(x,t)&=0 \qquad \mbox{ in $\widehat{\Omega}\times[0,\infty)$}\qquad \mbox{as $\delta\to 0$}\label{eq-La-of-u-sub-infty-equal-to-zero}.
\end{align}
Since $0<m<\frac{n-2}{n+2}$, by \eqref{gamma-upper-lower-bd2},
\begin{equation*}
\gamma_i'<\frac{n}{m+1}<\frac{n-2}{m}\quad\forall i=1,\cdots,i_0.
\end{equation*}
Hence putting $t=t+t_k$ in \eqref{eq-local-behaviour-of-solution-u-blow-up=0-above-by-1} and letting $k\to\infty$,
\begin{equation}\label{u^m-near-ai-decay}
u_{\infty}^{m}(x,t)\leq C_3^m|x-a_i|^{-m\gamma_i'}=o\left(\left|x-a_i\right|^{2-n}\right) \qquad \forall x\in \widehat{B_{\delta_2}}(a_i),\,\,t\in\R, i=1,\cdots,i_0.
\end{equation}
Note that (\cite{F}) when $n\geq 3$, a harmonic function $v$ in $D\bs\left\{x_0\right\}$, $x_0\in D\subset\R^n$, has a removable singularity at $x_0$ if and only if  there exists $\delta>0$ such that $|u(x)|\leq o\left(|x-x_0|^{2-n}\right)$ for any $x\in\widehat{B_{\delta}}(x_0)$. Hence by \eqref{u^m-near-ai-decay} 
$a_i$ is a removable singularity of $u_{\infty}^m$ for all $i=1,\cdots,i_0$. Thus by \eqref{eq-La-of-u-sub-infty-equal-to-zero} the function $u^m_{\infty}(\cdot,t)$ can be extended to a harmonic function in $\Omega$ for any $t\in [0,\infty)$. Hence  by \eqref{eq-boundary-condition-of-u-infty-by-f-mu-1} and the maximum principle for harmonic function,
\begin{equation*}\label{eq-aligned-La-u-infty-equiv-0-a-e-in-hat-Ojmega}
\begin{aligned}
&u_{\infty}(x,t)=\mu_1 \qquad \mbox{in $(\overline{\Omega}\bs\{a_1,\cdots,a_{i_0}\})\times[0,\infty)$}\\
\Rightarrow \qquad & \lim_{k\to\infty}u(x,t_k)=\mu_1 \qquad \mbox{in $\overline{\Omega}\bs\{a_1,\cdots,a_{i_0}\}$}.
\end{aligned}
\end{equation*}
Since the sequence $\{t_k\}_{k+1}^{\infty}$ is arbitrary, \eqref{eq-for-thm-a-global-behaviour-of-solution-u-blow-up=0-above-by} holds for any compact subset $K$ of $\overline{\Omega}\bs\{a_1,\cdots,a_{i_0}\}$.

\noindent\textbf{Case 2}: $\gamma_i<\gamma_i'$ for all $i=1,\cdots, i_0$.\\
Let $\delta_4=\min \left(\delta_3,\min_{1\le i\le i_0}(\lambda_i/\mu_0)^{\frac{1}{\gamma_i}}\right)$,
\begin{equation*}
\overline{u}_0(x)=\begin{cases}
\begin{array}{lcl}
u_0(x)&& x\in\Omega_{\delta_4}\\
\lambda_{i}'|x-a_i|^{-\gamma_i'}&&x\in B_{\delta_4}(a_i)\quad\forall i=1,\cdots,i_0
\end{array}
\end{cases}
\end{equation*}
and
\begin{equation*}
\underline{u}_0(x)=\begin{cases}
\begin{array}{lcl}
u_0(x)&& x\in\Omega_{\delta_4}\\
\lambda_{i}|x-a_i|^{-\gamma_i}&&x\in B_{\delta_4}(a_i)\quad\forall i=1,\cdots,i_0.
\end{array}
\end{cases}
\end{equation*}
Then 
\begin{equation}\label{eq-comparison-between-initial-datas-u-0-under-u-0-and-over-u-0}
\mu_0\le\underline{u}_0\le u_0\le\overline{u}_0\in L^p_{loc}(\widehat{\Omega}\bs\{a_1,\cdots,a_{i_0}\}).
\end{equation}
Let $\overline{u}$, $\underline{u}$, be the solutions of \eqref{Dirichlet-blow-up-problem} with $u_0=\overline{u}_0, \underline{u}_0$, respectively given by Theorem \ref{first-main-existence-thm}. Then by \eqref{eq-comparison-between-initial-datas-u-0-under-u-0-and-over-u-0} and Theorem \ref{bd-soln-comparison-thm1},
\begin{equation}\label{eq-ordering-of-u-and-overline-u-and-underline-u}
\underline{u}\leq u\leq \overline{u} \quad \mbox{ in }\widehat{\Omega}\times\left(0,\infty\right).
\end{equation}
By case 1,
\begin{equation}\label{eq-limit-of-overline-u-and-underline-u-as-t-to-infty}
\overline{u}(x,t)\to \mu_1 \qquad \mbox{and} \qquad \underline{u}(x,t)\to\mu_1 \qquad \mbox{in $C^2(K)$} \quad \mbox{ as }t\to\infty
\end{equation}
for any compact subset $K$ of $\2{\Omega}\setminus\{a_1,\cdots,a_{i_0}\}$.
Thus by \eqref{eq-ordering-of-u-and-overline-u-and-underline-u} and \eqref{eq-limit-of-overline-u-and-underline-u-as-t-to-infty}, \eqref{eq-for-thm-a-global-behaviour-of-solution-u-blow-up=0-above-by} holds for any compact subset $K$ of $\overline{\Omega}\bs\{a_1,\cdots,a_{i_0}\}$ and the theorem follows.
\end{proof}

\begin{remar}
If one only assume that $f\in L^{\infty}(\partial\Omega\times\left(0,\infty\right))$ and 
\begin{equation*}
f(x.t)\to\mu_1\mbox{ uniformly in } L^{\infty}(\partial\Omega)\mbox{  as }t\to\infty
\end{equation*} 
instead of $f\in L^{\infty}(\partial\Omega\times(0,\infty))\cap C^3(\partial\Omega\times(T_1,\infty))$, $f_t\in L^1(\partial\Omega\times(T_1,\infty))$, for some constant $T_1>0$ and condition \eqref{eq-for-thm-a-second-condition-on-f-convergence-to-mu-sub-one} in Theorem \ref{convergence-thm2}, then by an argument similar to the proof of Theorem \ref{convergence-thm2} one can prove that the solution $u$ of \eqref{Dirichlet-blow-up-problem} given by Theorem \ref{first-main-existence-thm} satisfy
\eqref{eq-for-thm-a-global-behaviour-of-solution-u-blow-up=0-above-by}  for any compact set $K\subset\widehat{\Omega}$. Moreover\begin{equation*}
u(x,t)\to \mu_1 \quad \mbox{ in }L_{loc}^{\infty}(\overline{\Omega}\bs\left\{a_1,\cdots,a_{i_0}\right\}) \quad \mbox{ as }t\to\infty
\end{equation*} 
holds.
\end{remar}

\begin{proof}[\textbf{Proof of Theorem \ref{convergence-thm3}}:]
Let $\delta_2=\delta_1/2$, $\{t_k\}_{k=1}^{\infty}\subset\R^+$ be a sequence such that $t_k\to\infty$ as $k\to\infty$, and $u_k(x,t)=u(x,t+t_k)$. Let $t_0=\frac{1}{2}$. By Lemma \ref{lem-local-upper-bound-of-u-near-singular-points} there exist constants $C_2>0$, $C_3>0$ such that \eqref{eq-bound-of-solution-on-compact-subset-with-t01-Omega-delta-2} and \eqref{eq-local-behaviour-of-solution-u-blow-up=0-above-by-1} hold. Then by \eqref{eq-bound-of-solution-on-compact-subset-with-t01-Omega-delta-2} and \eqref{eq-local-behaviour-of-solution-u-blow-up=0-above-by-1},
\begin{equation}\label{eq-bound-of-solution-on-compact-subset-with-t01-Omega-delta-2-for-u-sub-k}
u_{k}\leq C_2 \quad \mbox{ in }\overline{\Omega_{\delta_2}}\times \left(\frac{1}{2}-t_k,\infty\right)
\end{equation} 
and
\begin{equation}\label{eq-local-behaviour-of-solution-u-blow-up=0-above-by-for-u-sub-k}
u_k(x,t)\leq C_3\left|x-a_i\right|^{-\gamma_i'}  \qquad \forall 0<|x-a_i|\le\delta_2,t\ge\frac{1}{2}-t_k,i=1,\cdots,i_0.
\end{equation}
By \eqref{eq-first-condition-on-f-constant-mu-sub-zero} and an argument similar to the proof of Theorem \ref{convergence-thm1}, \eqref{u-lower-bd10} holds. By \eqref{u-lower-bd10}, \eqref{eq-bound-of-solution-on-compact-subset-with-t01-Omega-delta-2-for-u-sub-k}, \eqref{eq-local-behaviour-of-solution-u-blow-up=0-above-by-for-u-sub-k} and an argument similar to the proof of Theorem \ref{convergence-thm1} the sequence $\{u_k\}_{k=1}^{\infty}$ has a subsequence which we may assume without loss of generality to be the sequence itself that converges uniformly in $C^{2,1}(K)$ on every compact subset $K$ of $(\overline{\Omega}\bs\{a_1,\cdots,a_{i_0}\})\times(-\infty,\infty)$ to a solution $u_{\infty}$ of \eqref{fde} in $\widehat{\Omega}\times(-\infty,\infty)$ as $k\to\infty$ which satisfies
\begin{equation}\label{u-infty-lower-bd}
u_{\infty}\ge\mu_0\quad\mbox{ in }\widehat{\Omega}\times \R
\end{equation}
and
\begin{equation}\label{eq-boundary-condtion-of-u-infty-as-t-to-infty-for-extend}
u_{\infty}=g \qquad \mbox{ on }\partial\Omega\times\R.
\end{equation}
Letting $k\to\infty$ in \eqref{eq-bound-of-solution-on-compact-subset-with-t01-Omega-delta-2-for-u-sub-k} and \eqref{eq-local-behaviour-of-solution-u-blow-up=0-above-by-for-u-sub-k}, 
\begin{equation}\label{eq-bound-of-solution-on-compact-subset-with-t01-Omega-delta-2-for-u-sub-infty}
u_{\infty}\leq C_2 \qquad \mbox{in $\overline{\Omega_{\delta_2}}\times\left(-\infty,\infty\right)$}
\end{equation} 
and
\begin{equation}\label{eq-local-behaviour-of-solution-u-blow-up=0-above-by-for-u-sub-infty}
u_{\infty}(x,t)\leq C_3\left|x-a_i\right|^{-\gamma_i'}  \qquad \forall 0<\left|x-a_i\right|\leq\delta_2,t\in\R,i=1,\cdots,i_0.
\end{equation}
We now divide the proof into two cases.

\noindent\textbf{Case 1}: There exists a constant $T_0\ge 0$ such that \eqref{eq-condition-monotone-decreasingness-of-f2-34} holds.
By \eqref{eq-condition-monotone-decreasingness-of-f2-34} and Theorem \ref{first-main-existence-thm},
\begin{equation}\label{eq-Aronson-Benelan-for-u-sub-k-on-hat-Omeag-and-t-bigger-than-T-0-T-k}
u_{k,t}\leq \frac{u_k}{(1-m)\left(t+t_k-T_0\right)} \quad \mbox{ in }\widehat{\Omega}\times(T_0-t_k,\infty).
\end{equation}
Letting $k\to\infty$ in \eqref{eq-Aronson-Benelan-for-u-sub-k-on-hat-Omeag-and-t-bigger-than-T-0-T-k}, by \eqref{eq-bound-of-solution-on-compact-subset-with-t01-Omega-delta-2-for-u-sub-k} and \eqref{eq-local-behaviour-of-solution-u-blow-up=0-above-by-for-u-sub-k},
\begin{equation}\label{eq-negativity-of-time-derivaitive-of-u-infty-0}
u_{\infty,t}\leq 0 \quad \mbox{ in }\widehat{\Omega}\times\R.
\end{equation}
Thus by \eqref{u-infty-lower-bd}, \eqref{eq-bound-of-solution-on-compact-subset-with-t01-Omega-delta-2-for-u-sub-infty} and \eqref{eq-local-behaviour-of-solution-u-blow-up=0-above-by-for-u-sub-infty}, the equation \eqref{fde} for $u_{\infty}$ is uniformly parabolic in $\overline{\Omega_{\delta}}\times\left(-\infty,\infty\right)$ for any $0<\delta<\delta_2$. By the Schauder estimates \cite{LSU}  the family $\{u_{\infty}(\cdot,t)\}_{t\in\R}$ is equi-H\"older continuous in $C^2(K)$ for any compact subset $K$ of $\overline{\Omega}\bs\left\{a_1,\cdots,a_{i_0}\right\}$. Hence by \eqref{u-infty-lower-bd}, \eqref{eq-boundary-condtion-of-u-infty-as-t-to-infty-for-extend}, \eqref{eq-bound-of-solution-on-compact-subset-with-t01-Omega-delta-2-for-u-sub-infty}, \eqref{eq-local-behaviour-of-solution-u-blow-up=0-above-by-for-u-sub-infty} and \eqref{eq-negativity-of-time-derivaitive-of-u-infty-0} $u_{\infty}\left(\cdot,t\right)$ decreases (increases respectively) and converges uniformly in $C^{2}(K)$ on every compact subset $K$ of $\overline{\Omega}\bs\{a_1,\cdots,a_{i_0}\}$ to some function $w_1\in C^2(\overline{\Omega}\bs\{a_1,\cdots,a_{i_0}\})$ ($w_2\in C^2(\overline{\Omega}\bs\{a_1,\cdots,a_{i_0}\})$, respectively) as $t\to\infty$ $\left(t\to-\infty\,\, \mbox{respectively}\right)$ which satisfies
\begin{equation}\label{w-1-2-bd}
\left\{\begin{aligned}
&w_j\ge\mu_0\qquad\qquad\qquad\quad\mbox{in }\widehat{\Omega}\quad\,\,\,\forall j=1,2\\
&w_j=g \qquad\qquad\qquad\quad\mbox{ on }\partial\Omega\quad \forall j=1,2\\
&w_j\leq C_2 \qquad\qquad\qquad\,\,\mbox{ on }\overline{\Omega_{\delta_2}}\quad\forall j=1,2\\
&w_j(x)\leq C_3|x-a_i|^{-\gamma_i'} \quad\,\,\, \forall 0<|x-a_i|\leq\delta_2,i=1,\cdots,i_0,j=1,2.
\end{aligned}\right.
\end{equation}
By \eqref{eq-negativity-of-time-derivaitive-of-u-infty-0} and \eqref{w-1-2-bd},
\begin{equation*}
\int_{-\infty}^{\infty}\int_{\Omega_{\delta}}\left|\La u_{\infty}^m\right|\,dx\,dt=\int_{-\infty}^{\infty}\int_{\Omega_{\delta}}|u_{\infty,t}|\,dx\,dt=-\int_{-\infty}^{\infty}\int_{\Omega_{\delta}}u_{\infty,t}\,dx\,dt\leq \int_{\Omega_{\delta}}w_2\,dx\leq C_{\delta}\left|\Omega\right|<\infty \qquad \forall 0<\delta<\delta_2
\end{equation*}
where $C_{\delta}=\max (C_2,C_3\max_{1\le i\le i_0}\delta^{-\gamma_i'})$.
Hence there exist sequences $s_i\to\infty$ and $s'_i\to-\infty$ as $i\to\infty$ such that
\begin{align}
&\int_{\Omega_{\delta}}\left|\La u_{\infty}^m\left(x,s_i\right)\right|dx\to 0 \qquad \mbox{and} \qquad \int_{\Omega_{\delta}}\left|\La u_{\infty}^m\left(x,s'_i\right)\right|dx\to 0 \qquad \mbox{as $i\to\infty$}\notag\\
&\Rightarrow \qquad \int_{\Omega_{\delta}}\left|\La w_j^m\right|dx=0 \qquad \forall 0<\delta<\delta_0,\,\, j=1,2\notag\\
&\Rightarrow \qquad \La w_j^{m}=0 \qquad \mbox{ in }\widehat{\Omega}\quad\forall j=1,2.\label{eq-harmonicity-of-function-w-1-2-to-m-in-hat-Omega}
\end{align}
Since by \eqref{gamma-upper-lower-bd3}  and \eqref{w-1-2-bd},
\begin{equation*}
0\le w_j^{m}(x)\leq C_3^m|x-a_i|^{-m\gamma_i'}=o\left(\left|x-a_i\right|^{2-n}\right) \qquad \forall x\in \widehat{B}_{\delta_2}(a_i), i=1,\cdots,i_0,\,\, j=1,2.
\end{equation*}
Hence (\cite{F}) $a_i$ is a removable singularity of $w_j^m$ for all $i=1,\cdots,i_0$, $j=1,2$. Thus 
$w_j$  can be extended to a function on $\2{\Omega}$ for $j=1,2$, such that
\begin{equation}\label{eq-harmonicity-of-function-w-1-2-to-m-in-Omega}
\La w_j^m=0 \quad \mbox{ in }\Omega\quad\forall j=1,2.
\end{equation}
By \eqref{w-1-2-bd}, \eqref{eq-harmonicity-of-function-w-1-2-to-m-in-Omega}, and the maximum principle for harmonic functions,
\begin{equation}\label{eq-equivalent-of-w-on=whole-domain-by-mu-sub-zero}
w_j^m=\phi \quad \mbox{ in }\Omega\quad\forall j=1,2.
\end{equation}
By \eqref{eq-negativity-of-time-derivaitive-of-u-infty-0} and \eqref{eq-equivalent-of-w-on=whole-domain-by-mu-sub-zero},
\begin{align}
&\phi^{\frac{1}{m}}(x)=w_1(x)\leq u_{\infty}\left(x,t\right)\leq w_2(x)=\phi^{\frac{1}{m}}(x) \qquad \forall x\in \overline{\Omega}\bs\left\{a_1,\cdots,a_{i_0}\right\},t\in\R\notag\\
&\Rightarrow \qquad u_{\infty}(x,t)=\phi^{\frac{1}{m}}(x) \qquad \forall x\in\overline{\Omega}\bs\left\{a_1,\cdots,a_{i_0}\right\},t\in\R\label{u-infty-harmonic-fcn}\\
&\Rightarrow \qquad \lim_{k\to\infty}u\left(x,t_k\right)=\phi^{\frac{1}{m}}(x)\qquad \forall x\in\overline{\Omega}\bs\left\{a_1,\cdots,a_{i_0}\right\}\label{uk-limit=harmonic-fcn}.
\end{align}
Since the sequence $\left\{t_k\right\}$ is arbitrary, \eqref{u-phi-1/m-limit1} holds for any compact subset $K$ of $\overline{\Omega}\bs\left\{a_1,\cdots,a_{i_0}\right\}$.

\noindent\textbf{Case 2}: $f$ satisfies \eqref{eq-first-condition-on-f-constant-mu-sub-zero} and \eqref{eq-for-thm-a-second-condition-on-f-convergence-to-mu-sub-one-general}.

By \eqref{eq-for-thm-a-second-condition-on-f-convergence-to-mu-sub-one-general} for any $i\geq 2$ there exists a constant $T_i>0$ such that
\begin{equation*}
g-\frac{\mu_0}{i}\leq f \leq g+\frac{\mu_0}{i} \qquad \mbox{on $\partial\Omega\times\left(T_i,\infty\right)$}.
\end{equation*}
For any $i\ge 2$, let $\underline{f}_i$, $\underline{g}_i$, $\overline{f}_i$ and $\overline{g}_i$ be given by 
\begin{equation}\label{fi-gi-defn}
\begin{cases}
\begin{aligned}
\underline{f}_i(x,t)&=\min\left(f(x,t), g(x)-\frac{\mu_0}{i}\right) ,\quad \overline{f}_i(x,t)=\max\left(f(x,t),\,g(x)+\frac{\mu_0}{i}\right) \quad \mbox{ on $\partial\Omega\times(0,\infty)$}\\
\underline{g}_i(x)&=g(x)-\frac{\mu_0}{i},\quad \overline{g}_i(x)=g(x)+\frac{\mu_0}{i}\quad \mbox{ on $\partial\Omega$},
\end{aligned}
\end{cases}
\end{equation}
and $\underline{u}_{0,i}$, $\overline{u}_{0,i}$ be given by \eqref{lower-upper-approx-initial-value}.
Let $\underline{u}_i$, $\overline{u}_i$ be solutions of \eqref{Dirichlet-blow-up-problem} with $f=\underline{f}_i$, $\overline{f}_i$ and $u_0=\underline{u}_i, \overline{u}_i$, respectively  given by Theorem \ref{first-main-existence-thm}.
Let $\underline{\phi}_i$, $\overline{\phi}_i$,  be the solutions of \eqref{harmonic-eqn} with $=\underline{g}_i, \overline{g}_i$, respectively. 
Since 
\begin{equation*}
{f}_i(x,t)=g(x)-\frac{\mu_0}{i}\quad\mbox{ and }\quad\overline{f}_i(x,t)=g(x)+\frac{\mu}{i}\quad\mbox{ on } \partial\Omega\times(T_i,\infty),
\end{equation*} 
by case 1,
\begin{equation}\label{eq-uniform-convergence-of-lower-one-by-Case-1-01}
\left\{\begin{aligned}
&\underline{u}_i\to \underline{\phi}_i^{\frac{1}{m}}\quad \mbox{ in $C^2(K)$ as $t\to\infty$}\\
&\overline{u}_i\to \overline{\phi}_i^{\frac{1}{m}}\quad \mbox{ in $C^2(K)$ as $t\to\infty$}
\end{aligned}\right.
\end{equation}
for any  compact subset $K$ of $\overline{\Omega}\bs\left\{a_1,\cdots,a_n\right\}$.
Since by \eqref{fi-gi-defn},
\begin{equation*}
\left\{\begin{aligned}
&\underline{f}_i\le f\le\overline{f}_i\quad\mbox{ on }\partial\Omega\times(0,\infty)\quad\forall i\ge 2\\
&\underline{g}_i\le g\le\overline{g}_i\quad\,\mbox{ on }\partial\Omega\qquad\qquad\,\,\,\,\forall i\ge 2,
\end{aligned}\right.
\end{equation*}
by Theorem \ref{bd-soln-comparison-thm1} and \eqref{eq-uniform-convergence-of-lower-one-by-Case-1-01},
\begin{align}\label{u-i-u-ineqn}
&\underline{u}_i(x,t+t_k)\leq u_k(x,t)\leq \overline{u}_i(x,t+t_k) \quad\forall x\in \overline{\Omega}\bs\left\{a_1,\cdots,a_{i_0}\right\},t\in\R, k\in\Z^+,i\ge i_2\notag\\
\Rightarrow \quad&\underline{\phi}^{\frac{1}{m}}_i(x)\leq u_{\infty}(x,t)\leq \overline{\phi}^{\frac{1}{m}}_i(x) \qquad\qquad\,\,\,\forall x\in \overline{\Omega}\bs\left\{a_1,\cdots,a_{i_0}\right\}, t\in\R, i\ge i_2\quad\mbox{ as }k\to\infty.
\end{align}
Since both $\underline{g}_i$ and $\overline{g}_i$ converges uniformly on $\1\Omega$ to $g$ as $i\to\infty$, both $\underline{\phi}_i$ and $\overline{\phi}_i$ converges uniformly on $\2{\Omega}$ to $\phi$ as $i\to\infty$. Hence letting $i\to\infty$ in \eqref{u-i-u-ineqn}, we get \eqref{u-infty-harmonic-fcn} and \eqref{uk-limit=harmonic-fcn}.
Since the sequence $\{t_k\}_{k=1}^{\infty}$ is arbitrary, \eqref{u-phi-1/m-limit1} holds for any compact subset $K$ of $\overline{\Omega}\bs\left\{a_1,\cdots,a_{i_0}\right\}$  and the theorem follows.
\end{proof}

\begin{remar}
If one only assume that $f\in L^{\infty}(\partial\Omega\times(0,\infty))$, $g\in C(\1\Omega)$ and 
\begin{equation*}
f(x.t)\to g(x) \quad\mbox{ uniformly in }L^{\infty}(\partial\Omega)\mbox{ as }t\to\infty
\end{equation*}
instead of $f\in L^{\infty}(\partial\Omega\times\left(0,\infty\right))\cap C^3(\partial\Omega\times\left(T_1,\infty\right))$ for some constant $T_1>0$, $g\in C^3(\1\Omega)$ and  condition \eqref{eq-for-thm-a-second-condition-on-f-convergence-to-mu-sub-one-general} in Theorem \ref{convergence-thm3}, then by an argument similar to the proof of Theorem \ref{convergence-thm3} one can prove that the solution $u$ of \eqref{Dirichlet-blow-up-problem} given by Theorem \ref{first-main-existence-thm} satisfy \eqref{u-phi-1/m-limit1} for any compact set $K\subset\widehat{\Omega}$. Moreover \begin{equation*}
u(x,t)\to \phi^{\frac{1}{m}} \quad \mbox{ in }L_{loc}^{\infty}(\overline{\Omega}\bs\left\{a_1,\cdots,a_{i_0}\right\}) \quad \mbox{ as }t\to\infty
\end{equation*} 
holds.
\end{remar}

\begin{proof}[\textbf{Proof of Theorem \ref{convergence-thm4}}:]
For any $0<\3<1$ and $M>0$, let $u_{\3,M}$ be the solution of \eqref{cauchy-problem-Rn} which satisfies \eqref{u-epsilon-m-lower-upper-bd11}. Then by the proof of Theorem \ref{second-main-existence-thm},
\begin{equation*}
u(x,t)=\lim_{M\to\infty}\lim_{\3\to 0} u_{\3, M}(x,t) \qquad \forall x\in\widehat{\R^n},\,\,t>0.
\end{equation*}
Moreover by an argument similar to the proof of Theorem \ref{convergence-thm1},
\begin{equation}\label{eq-lower-bound-of-u-through-widehat-R-n-times-0-infty-3-2-8}
u\geq\mu_0 \qquad \mbox{in $\widehat{\R^n}\times(0,\infty)$}.
\end{equation}
By \eqref{upper-blow-up-rate-initial-data} and \eqref{gamma-upper-lower-bd1}, $u_0\in L^1(\widehat{B}_{R_1})$. Hence by \eqref{eq-initial-condition-over-ball-L-1-bound-of-u-0-and-mu-0}, $u_0-\mu_0\in L^1(\widehat{\R^n})$. We extend $u_0$ to a function on $\R^n$ by setting $u_0(a_i)=0$ for all $i=1,2,\cdots,i_0$. Then $u_0-\mu_0\in L^1(\R^n)$.
By Corollary 2.2 of \cite{DS1},
\begin{align}\label{eq-L-1-contraction-of-u-minus-mu-sub-zeor-as-m-to-infty-and-epsilon-to-zoer}
&\int_{\R^n}|u_{\3,M}(x,t)-\mu_0|\,dx\leq\int_{\R^n}|u_{0,M}-\mu_0|\,dx\quad \forall t>0, M>\max(\mu_0,C_1), 0<\3<1\notag\\
\Rightarrow\quad&\int_{\widehat{\R^n}}|u(x,t)-\mu_0|\,dx\le\int_{\R^n}|u_0-\mu_0|\,dx<\infty\quad\forall t>0\quad \mbox{ as }\3\to0, M\to\infty.
\end{align}
Let $\left\{t_k\right\}_{k=1}^{\infty}\subset\R^+$ be a sequence such that $t_k\to\infty$ as $k\to\infty$ and  $u_k(x,t)=u(x,t+t_k)$.  Then by Lemma \ref{upper-bound-blow-up-soln-Rn1} for any $0<\3_1<1$ there exist constants $t_0>0$ and $A_i>0$, $i=1,\cdots, i_0$,  such that \eqref{upper-bound-of-u-by-constant-and-v-eta} holds. 
By \eqref{upper-bound-of-u-by-constant-and-v-eta} and \eqref{eq-lower-bound-of-u-through-widehat-R-n-times-0-infty-3-2-8},
\begin{equation}\label{eq-upper-and-lower-bound-of-u-k-by-mu-0-and-terms-from-Theorem-3-3}
\mu_0\le u_k(x,t)\leq \left(\left(C_1+\3_1\right)^m+\sum_{i=1}^{i_0}\left(A_i\left|x-a_i\right|^{-\gamma_i}\right)^m\right)^{\frac{1}{m}}\qquad \forall x\in\widehat{\R^n},\,\,t\geq t_0-t_k.
\end{equation}
For any $N_0>0$, there exists $k_{N_0}\in\Z^+$ such that $-N_0>t_0-t_k$  for all $k\geq k_{N_0}$. Then by \eqref{eq-upper-and-lower-bound-of-u-k-by-mu-0-and-terms-from-Theorem-3-3} the equation \eqref{fde} for the sequence $\left\{u_k\right\}_{k=k_{_{N_0}}}^{\infty}$ is uniformly parabolic on $\R^n_{\delta}\times(-N_0,\infty)$ for any $0<\delta<\delta_0$. Hence by the Schauder estimates \cite{LSU} the sequence $\{u_k\}_{k=k_{_{N_0}}}^{\infty}$ is equi-H\"older continuous in $C^{2,1}(K)$ for any compact subset $K$ of $\widehat{\R^n}\times\left(-N_0,\infty\right)$. Thus by the Ascoli theorem and a diagonalization argument the sequence $\{u_k\}_{k=1}^{\infty}$ has a subsequence which we may assume without loss of generality to be the sequence itself that converges uniformly in $C^{2,1}(K)$ on every compact subset $K$ of $\widehat{\R^n}\times\left(-\infty,\infty\right)$ to a solution $u_{\infty}$ of \eqref{fde} in $\widehat{\R^n}\times(-\infty,\infty)$ as $k\to\infty$. Letting $k\to\infty$ in \eqref{eq-upper-and-lower-bound-of-u-k-by-mu-0-and-terms-from-Theorem-3-3},
\begin{equation}\label{eq-upper-bound-of-u-infty-on-widehat-R-n-by-Lemma-3-3}
\mu_0\leq u_{\infty}(x,t)\leq \left(\left(C_1+\3_1\right)^m+\sum_{i=1}^{i_0}\left(A_i\left|x-a_i\right|^{-\gamma_i}\right)^m\right)^{\frac{1}{m}}\qquad \forall x\in\widehat{\R^n},\,\,t\in\R.
\end{equation} 
Putting $u=u_k$ and $t=t+t_k$ in \eqref{eq-Aronson-Bernilan-on-R-n} and letting $k\to\infty$, by \eqref{eq-upper-and-lower-bound-of-u-k-by-mu-0-and-terms-from-Theorem-3-3},
\begin{equation}\label{eq-negativity-of-time-derivaitive-of-u-infty-0-on-R-to-n}
u_{\infty,t}\le 0 \quad \mbox{ in }\widehat{\R^n}\times\left(-\infty,\infty\right).
\end{equation}
Putting $t=t+t_k$ and letting $k\to\infty$ in \eqref{eq-L-1-contraction-of-u-minus-mu-sub-zeor-as-m-to-infty-and-epsilon-to-zoer},
\begin{equation}\label{eq-L-1-contraction-of-u-infty-minus-mu-sub-zeor-as-k-to-infty}
\int_{\widehat{\R^n}}|u_{\infty}(x,t)-\mu_0|\,dx\leq\int_{\R^n}|u_0-\mu_0|\,dx<\infty\quad\forall t\in\R.
\end{equation}
By \eqref{eq-upper-bound-of-u-infty-on-widehat-R-n-by-Lemma-3-3} the equation \eqref{fde} for  $u_{\infty}$ is uniformly parabolic on $\R^n_{\delta}\times\left(-\infty,\infty\right)$ for any $0<\delta<\delta_2$. By the Schauder estimates \cite{LSU} the family $\{u_{\infty}(\cdot,t)\}_{t\in\R}$ is equi-H\"older continuous in $C^2(K)$ for any compact subset $K$ of $\widehat{\R^n}$. Hence by \eqref{eq-negativity-of-time-derivaitive-of-u-infty-0-on-R-to-n} $u_{\infty}(\cdot,t)$ decreases (increases respectively) and converges uniformly in $C^{2}\left(K\right)$ on every compact subset $K$ of $\widehat{\R^n}$ to some functions $w_1\in C^2\left(\widehat{\R^n}\right)$ ($w_2\in C^2\left(\widehat{\R^n}\right)$, respectively) as $t\to\infty$ $\left(t\to-\infty\,\, \mbox{respectively}\right)$ which satisfies
\begin{equation}\label{eq-bound-of-w-limit-of-u-sub-infty-by-C0delta-53-on-R-to-n}
\mu_0\leq w_j(x)\leq \left(\left(C_1+\3_1\right)^m+\sum_{i=1}^{i_0}\left(A_i\left|x-a_i\right|^{-\gamma_i}\right)^m\right)^{\frac{1}{m}} \qquad \forall x\in\widehat{\R^n},\,\, j=1,2.
\end{equation}
Letting $t\to\pm\infty$ in \eqref{eq-L-1-contraction-of-u-infty-minus-mu-sub-zeor-as-k-to-infty},
\begin{equation*}
\int_{R_0}^{\infty}\int_{\1 B_r}|w_j(y)-\mu_0|\,d\sigma (y)dr\le\int_{\widehat{\R^n}}|w_j(x)-\mu_0|\,dx\leq\int_{\R^n}|u_0-\mu_0|\,dx<\infty\quad\forall j=1,2.
\end{equation*}
Then for any $j=1,2$, there exists a sequence $\{R_{j,i}\}_{i=1}^{\infty}$, $R_{j,i}\to\infty$ as $i\to\infty$ and $R_{j,i}>R_0$ for all $i\in\Z^+$, such that
\begin{equation*}
\int_{\1 B_{R_{j,i}}}|w_j(y)-\mu_0|\,d\sigma (y)\to 0\quad\mbox{ as }i\to\infty.
\end{equation*}
Then for any $0<\3<1$ there exists $n_{\3}\in\Z^+$ such that 
\begin{equation}\label{eq-inequality-for-behaviour-of-u-infyt-at-infty-equail--mu-sub-zero-before-r-to=infty}
\int_{\1 B_{R_{j,i}}}|w_j(y)-\mu_0|\,d\sigma (y)<\3 \quad \forall i\ge n_{\3}, j=1,2.
\end{equation}
By \eqref{gamma-upper-lower-bd1}, \eqref{eq-negativity-of-time-derivaitive-of-u-infty-0-on-R-to-n}, \eqref{eq-bound-of-w-limit-of-u-sub-infty-by-C0delta-53-on-R-to-n} and an argument similar to the proof of Theorem \ref{convergence-thm3},
\begin{equation*}
\La w_j^m=0 \qquad \mbox{in $\widehat{\R^n}$}\qquad  \forall j=1,\,2
\end{equation*}
and $w_j^m$ has removable singularities at $a_1$, $\cdots$, $a_{i_0}$. Hence $w_j$ can be extended to a function on $\R^n$ for $j=1,2$, such that
\begin{equation}\label{eq-harmonicity-of-function-w-1-2-to-m-in-Omega-on-R-to-n}
\La w_j^m=0 \qquad \mbox{in $\R^n$} \qquad \forall j=1,\,2.
\end{equation}
By \eqref{eq-harmonicity-of-function-w-1-2-to-m-in-Omega-on-R-to-n} and  the Green representation formula for harmonic functions for any $r>R_1$,
\begin{equation}\label{eq-express-of-limit-function-w-j-by-Green-functions}
w_j^m(x)=\int_{y\in\partial B_{r}}w_{j}^m(y)\frac{\partial G_r(x,y)}{\partial n}\,d\sigma (y) \quad \forall x\in B_{r},j=1,2
\end{equation}
where $G_r$ is the Green function for $B_r$ and $\frac{\partial}{\partial n}$ is the derivative with respect to the unit outward normal on $\partial B_r$ at the point $y$. Let $R_2>1$. Since 
\begin{equation*}
\left|\nabla_yG_r\left(x,y\right)\right|\leq C \qquad \forall |x|< R_2,|y|=r>2R_2
\end{equation*} 
for some constant $C>0$, by \eqref{eq-bound-of-w-limit-of-u-sub-infty-by-C0delta-53-on-R-to-n}, \eqref{eq-inequality-for-behaviour-of-u-infyt-at-infty-equail--mu-sub-zero-before-r-to=infty}, \eqref{eq-express-of-limit-function-w-j-by-Green-functions} and the mean value theorem,
\begin{align}\label{eq-equality-of-limit-function-w-j-m-with-mu-0-m-by-Green-functions}
|w_j^m(x)-\mu_{0}^m|&=\left|\int_{\partial B_{R_{j,i}}}(w_j^m(y)-\mu_{0}^m)\frac{\partial G_{r_i}(x,y)}{\partial n}\,d\sigma (y)\right| \qquad \forall x\in B_{R_2},i\geq n_{\3}\mbox{ and $R_{j,i}>2R_2$},j=1,2\notag\\
&\leq C\int_{\partial B_{R_{j,i}}}|w_{j}^m(y)-\mu_{0}^m|\,d\sigma (y)\qquad \qquad\qquad\forall x\in B_{R_2},i\geq n_{\3}\mbox{ and $R_{j,i}>2R_2$},j=1,2\notag\\
&\leq Cm\mu_0^{m-1}\int_{\partial B_{R_{j,i}}}|w_j(y)-\mu_{0}|\,d\sigma (y)\qquad \quad\forall x\in B_{R_2},i\geq n_{\3}\mbox{ and $R_{j,i}>2R_2$},j=1,2\notag\\
&\leq Cm\mu_0^{m-1}\3 \qquad \qquad\qquad\qquad\qquad\qquad\forall x\in B_{R_2},0<\3<1.
\end{align}
Since $\3>0$ and $R_2>1$ are arbitrary, letting $\3\to 0$ first and then $R_2\to\infty$ in \eqref{eq-equality-of-limit-function-w-j-m-with-mu-0-m-by-Green-functions},
\begin{equation}\label{eq-equivalent-of-w-on=whole-domain-by-mu-sub-zero-on-R-to-n}
w_j=\mu_0 \quad \mbox{ in }\R^n\quad\forall j=1,2.
\end{equation}
By \eqref{eq-negativity-of-time-derivaitive-of-u-infty-0-on-R-to-n} and \eqref{eq-equivalent-of-w-on=whole-domain-by-mu-sub-zero-on-R-to-n},
\begin{align}
&\mu_0=w_1(x)\leq u_{\infty}\left(x,t\right)\leq w_2(x)=\mu_0 \qquad \forall x\in \widehat{\R^n},t\in\R\label{w1-u-infty-w2-ineqn}\\
\Rightarrow \quad&u_{\infty}(x,t)=\mu_0\qquad\qquad\qquad\qquad\qquad\quad\,\forall x\in\widehat{\R^n},t\in\R\notag\\
\Rightarrow \quad&\lim_{k\to\infty}u\left(x,t_k\right)=\mu_0\qquad\qquad\qquad\qquad\quad\,\,\forall x\in\widehat{\R^n}.\label{u-infty-limit-mu0}
\end{align}
Since the sequence $\left\{t_k\right\}$ is arbitrary, \eqref{u-infty-to-mu0} holds for any compact subset $K$ of $\widehat{\R^n}$. and the theorem follows.
\end{proof}

\begin{proof}[\textbf{Proof of Theorem \ref{convergence-thm5}}:]
Since the proof of the theorem is similar to the proof of Theorem \ref{convergence-thm4} we will only sketch the argument here.
Let $0<\3_1<1$. Then by \eqref{eq-initial-condition-of-u-0-to-mu-0-at-infty} there exists a constant $R_1>R_0$ such that
\begin{equation}\label{eq-upper-bound-of-u-0-by-condition-at-infty-equal-mu-0}
u_0(x)\leq \mu_0+\frac{\3_1}{2} \qquad \forall |x|\geq R_1.
\end{equation}
By \eqref{eq-upper-bound-of-u-0-by-condition-at-infty-equal-mu-0} and Lemma \ref{upper-bound-blow-up-soln-Rn1} there exist constants $t_0>0$ and $A_i>0$, $i=1,\cdots,i_0$,  such that \eqref{upper-bound-of-u-by-constant-and-v-eta} holds with $C_1=\mu_0$. Let $\{t_k\}_{k=1}^{\infty}\subset\R^+$ be a sequence such that $t_k\to\infty$ as $k\to\infty$ and  $u_{k}(x,t)=u(x,t+t_k)$. Then by an argument similar to the proof of Theorem \ref{convergence-thm4}, $\{u_k\}_{k=1}^{\infty}$ has a subsequence which we may assume without loss of generality to be the sequence itself that converges  in $C^{2,1}(K)$ for any compact subset $K$ of $\widehat{\R^n}\times\left(-\infty,\infty\right)$ to a solution $u_{\infty}$ of \eqref{fde} in $\widehat{\R^n}\times(-\infty,\infty)$ which satisfies \eqref{eq-negativity-of-time-derivaitive-of-u-infty-0-on-R-to-n}  as $k\to\infty$. 

Moreover by \eqref{eq-negativity-of-time-derivaitive-of-u-infty-0-on-R-to-n} $u_{\infty}(\cdot,t)$ decreases (increases, respectively) and converges uniformly in $C^2(K)$ on every compact subset $K$ of $\widehat{\R^n}$ to some functions 
\begin{equation*}
w_1\in C^2(\widehat{\R^n})\quad (w_2\in C^2(\widehat{\R^n}),\mbox{ respectively})\mbox{ as }t\to\infty \quad(t\to-\infty\mbox{ respectively})
\end{equation*}
which satisfies \eqref{eq-bound-of-w-limit-of-u-sub-infty-by-C0delta-53-on-R-to-n}. 
By \eqref{gamma-upper-lower-bd3} and an argument similar to the proof of Theorem  \ref{convergence-thm3}, $w_j$ can be extended to a function on $\R^n$ for $j=1,2$, which satisfies \eqref{eq-harmonicity-of-function-w-1-2-to-m-in-Omega-on-R-to-n}.
By \eqref{eq-bound-of-w-limit-of-u-sub-infty-by-C0delta-53-on-R-to-n},
\begin{align}\label{eq-bound-of-w-j-at-infty-between-mu-0-and-mu-0-epsilon}
&\mu_0\leq \liminf_{|x|\to\infty}w_j(x)\leq\limsup_{|x|\to\infty}w_j(x)\leq\mu_0+\3_1 \qquad \forall 0<\3_1<1,j=1,2\notag\\
\Rightarrow\quad&\lim_{|x|\to\infty}w_{j}(x)=\mu_0 \quad \forall j=1,2\quad\mbox{ as }\3_1\to 0.
\end{align}
By \eqref{eq-harmonicity-of-function-w-1-2-to-m-in-Omega-on-R-to-n}, \eqref{eq-bound-of-w-j-at-infty-between-mu-0-and-mu-0-epsilon} and the maximum principle for harmonic function in a bounded domain,
\begin{equation}\label{eq-equivalent-of-w-on=whole-domain-by-mu-0-on-R-to-n-by-Liouville-thm}
w_j(x)=\mu_0 \qquad \forall x\in\R^n\,\,j=1,2.
\end{equation}
By \eqref{eq-equivalent-of-w-on=whole-domain-by-mu-0-on-R-to-n-by-Liouville-thm}, we get
\eqref{w1-u-infty-w2-ineqn}. Hence \eqref{u-infty-limit-mu0} holds.
Since the sequence $\left\{t_k\right\}$ is arbitrary, \eqref{u-infty-to-mu0} holds for any compact subset $K$ of $\widehat{\R^n}$ and the theorem follows.
\end{proof}

\begin{proof}[\textbf{Proof of Theorem \ref{convergence-thm6}}:]
Let 
\begin{equation*}
v_0(x)=\begin{cases}
\begin{array}{lcl}
\lambda_1|x-a_1|^{-\gamma_1}&&\forall x\in \widehat{B}_{\delta_1}(a_1)\\
\mu_0&&\forall x\in\2{\Omega}\bs B_{\delta_1}(a_1).
\end{array}
\end{cases}
\end{equation*}
By Theorem \ref{first-main-existence-thm} there exists a solution $v$ of
\begin{equation*}
\begin{cases}
\begin{aligned}
v_t&=\La v^m\quad \mbox{ in $\left(\Omega\bs\left\{a_1\right\}\right)\times(0,\infty)$}\\
v(x,t)&=\mu_0 \quad\,\,\,\, \mbox{ on $\partial\Omega$}\\
v(x,0)&=v_0(x) \quad \mbox{ in $\Omega\bs\left\{a_1\right\}$}.
\end{aligned}
\end{cases}
\end{equation*}
Since $u_0\geq v_0$ in $\Omega$,
by \eqref{eq-first-condition-on-f-constant-mu-sub-zero} and Theorem \ref{bd-soln-comparison-thm1},
\begin{equation}\label{eq-compare-solution-u-and-v-sol-of-minimized-initial-function-min-u-0-and-mu-0}
u\ge v \quad\mbox{ in }\widehat{\Omega}\times(0,\infty).
\end{equation}
By Theorem 4.13 of \cite{VW1},
\begin{equation}\label{eq-known-fact-by-VW-Theorem-4-13-locally-uniformly-blow-up-whole-domain}
v\to\infty \qquad \mbox{locally uniformly in $\Omega$ }\quad \mbox{as $t\to\infty$}.
\end{equation}
Thus by \eqref{eq-compare-solution-u-and-v-sol-of-minimized-initial-function-min-u-0-and-mu-0} and \eqref{eq-known-fact-by-VW-Theorem-4-13-locally-uniformly-blow-up-whole-domain}, \eqref{eq-blow-up-of-soluiton-u-when-gamma-greater-than-n-2-over-m-2c} holds and the theorem follows.
\end{proof}

\begin{proof}[\textbf{Proof of Theorem \ref{convergence-thm7}}:]
By an argument similar to the proof of Theorem \ref{convergence-thm1},
\begin{equation}\label{eq-lower-bound-of-u-through-widehat-R-n-times-0-infty-3-2-10}
u\geq\mu_0 \qquad \mbox{in $\widehat{\R^n}\times(0,\infty)$}.
\end{equation}
Let 
\begin{equation*}
v_0(x)=\begin{cases}
\begin{array}{lcl}
\lambda_1|x-a_1|^{-\gamma_1}&&\forall x\in \widehat{B}_{\delta_1}(a_1)\\
\mu_0&&\forall x\in\R^n\bs B_{\delta_1}(a_1).
\end{array}
\end{cases}
\end{equation*}
For any $M>0$, let $v_{0,M}(x)=\min (v_0(x),M)$. Let $R>R_0+2\delta_0$ and $M>\mu_0$. We will now prove the existence of a unique solution $v_M$ of 
\begin{equation}\label{v-bd-domain-problem}
\begin{cases}
\begin{aligned}
v_t&=\La v^m\qquad \mbox{ in }B_R(a_1)\times(0,\infty)\\
v(x,t)&=\mu_0 \qquad\quad \mbox{on $\partial B_R(a_1)\times(0,\infty)$}\\
v(x,0)&=v_{0,M}(x) \quad \mbox{in }B_R(a_1)
\end{aligned}
\end{cases}
\end{equation}
which satisfies 
\begin{equation}\label{lower-upper-bd10}
\mu_0\le v_M(x,t)\le M\quad\forall x\in \2{B_R(a_1)}, t>0.
\end{equation}
Since the proof is similar to the proof of Theorem 3.5 of \cite{Hu2}, we will only sketch the argument here.
We choose a monotone decreasing function $\psi\in C^{\infty}(\R)$, $\psi>0$ on $\R$, such that $\psi(s)=ms^{-\frac{1-m}{m}}$ for $(\mu_0/2)^m\le s\le (2M)^m$, $\psi(s)=m(\mu_0/4)^{-\frac{1-m}{m}}$ for $s\le (\mu_0/4)^m$, $\psi(s)=m(4M)^{-\frac{1-m}{m}}$ for $s\ge (4M)^m$. By standard parabolic theory for non-degenerate parabolic equation \cite{LSU} there exists a unique solution $w_M\in C^{2,1}(\2{B_R(a_1)}\times(0,\infty))$ to the problem,
\begin{equation*}
\begin{cases}
\begin{aligned}
w_t&=\psi (w)\Delta w\qquad \mbox{ in }B_R(a_1)\times(0,\infty)\\
w(x,t)&=\mu_0^m \qquad\qquad \mbox{on $\partial B_R(a_1)\times(0,\infty)$}\\
w(x,0)&=v_{0,M}(x)^m \quad\,\, \mbox{ in }B_R(a_1).
\end{aligned}
\end{cases}
\end{equation*}
Moreover by the maximum principle, $\mu_0^m\le w_M\le M^m$ in $\2{B_R(a_1)}\times(0,\infty)$. Hence $\psi (u_M)=mw_M^{-\frac{1-m}{m}}$. Thus $v_M=w_M^{1/m}$ is a solution of \eqref{v-bd-domain-problem} which satisfies \eqref{lower-upper-bd10}. By Lemma \ref{comparison-lem-bd-domain} the solution $v_M$ is unique.

By Lemma \ref{lem-um-converge-to-u}, as $M\to\infty$, $v_M$ converges in $C^{2,1}(K)$ for any compact subset $K$ of $(\2{B_R}\setminus\{a_1\})\times(0,\infty)$ to a solution $v$ of
\begin{equation*}
\begin{cases}
\begin{aligned}
v_t&=\La v^m\quad\mbox{ in }(B_R\setminus\{a_1\})\times(0,\infty)\\
v(x,t)&=\mu_0 \qquad\mbox{on $\partial B_R\times(0,\infty)$}\\
v(x,0)&=v_0(x) \quad \mbox{in }B_R\setminus\{a_1\}.
\end{aligned}
\end{cases}
\end{equation*}
Let $u_{0,M}$ be given by \eqref{eq-def-of-u-0-epsilon-M} and $u_M$ be the solution of \eqref{cauchy-problem-Rn-2} which satisfies \eqref{u-m1-u-m2-compare} and
\begin{equation}\label{um-lower-bd11}
u_M\ge\mu_0\quad\mbox{ in }\R^n\times (0,\infty)\quad\forall M>\mu_0.
\end{equation} 
By the proof of Theorem \ref{second-main-existence-thm} $u_M$ increases and converges uniformly in $C^{2,1}(K)$ on every compact subset $K$ of $\widehat{\R^n}\times(0,\infty)$ to the solution $u$ of \eqref{cauchy-blow-up-problem} as $M\to\infty$.
Since
\begin{equation*}
u_{0,M}\geq v_{0,M}\quad\mbox{ in }B_R(a_1)\quad\forall M>\mu_0,
\end{equation*} 
by \eqref{um-lower-bd11} and Lemma \ref{L1-contraction-lem},
\begin{align}\label{u>v-ineqn2}
&u_M\geq v_M\quad \mbox{ in }B_R(a_1)\times(0,\infty)\quad\forall M>\mu_0\notag\\
\Rightarrow\quad&u\geq v \quad\quad\,\, \mbox{ in }\widehat{B}_R(a_1)\times(0,\infty)\quad\mbox{ as }M\to\infty.
\end{align}
By Theorem 4.13 of \cite{VW1}, 
\begin{equation}\label{v-limit}
v(x,t)\to\infty \quad \mbox{ on } K\quad \mbox{ as }t\to\infty
\end{equation}
for any compact subset $K$ of $\widehat{B}_R(a_1)$.
Hence by \eqref{u>v-ineqn2} and \eqref{v-limit}, \eqref{eq-blow-up-of-soluiton-u-when-gamma-greater-than-n-2-over-m-2c} holds for any compact subset $K$ of $\widehat{B}_R(a_1)$.
Since $R>R_0+2\delta_0$ is arbitrary, \eqref{eq-blow-up-of-soluiton-u-when-gamma-greater-than-n-2-over-m-2c} holds for any compact subset $K$ of $\widehat{\R^n}$ and the theorem follows.
\end{proof}

\noindent {\bf Acknowledgement:} Sunghoon Kim was supported by the National Research Foundation of Korea (NRF) grant funded by the Korea government (MSIP, no. 2015R1C1A1A02036548). Sunghoon Kim was also supported by the Research Fund, 2017 of The Catholic University of Korea.

\end{document}